\documentclass[a4paper, 10pt, twoside]{amsart}
\usepackage[top=3cm, bottom=3cm, left=3cm, right=3cm]{geometry}
\usepackage[pdftex]{graphicx}

\usepackage[utf8]{inputenc}
\usepackage[T1]{fontenc}

\usepackage[english]{babel}
\usepackage[noadjust]{cite}

\usepackage{amsthm,amssymb,epsfig,mathtools}
\usepackage{mathrsfs}  
\usepackage{bbm}        
\usepackage{esint}     
\usepackage{dsfont}
\usepackage{color}
\usepackage{enumerate}
\usepackage{url}
\usepackage[ruled,vlined]{algorithm2e} 

\usepackage{calc}
\usepackage[inline]{enumitem}
\usepackage[normalem]{ulem}
\usepackage{url}

\usepackage[bookmarks, bookmarksnumbered, pdfstartview={XYZ null null 1.00}]{hyperref}

\usepackage{tikz}
\usetikzlibrary{shapes, arrows, arrows.meta, decorations.markings}

\usepackage{todonotes}
\usepackage[ulem=normalem,draft]{changes}
\usepackage{cancel}

\makeindex

\numberwithin{equation}{section}

\newcommand{\theoname}{Theorem}
\newcommand{\lemmname}{Lemma}
\newcommand{\coroname}{Corollary}
\newcommand{\propname}{Proposition}
\newcommand{\definame}{Definition}

\newcommand{\remkname}{Remark}
\newcommand{\explname}{Example}

\theoremstyle{plain}
\newtheorem{theorem}{\theoname}[section]
\newtheorem{lemma}[theorem]{\lemmname}
\newtheorem{corollary}[theorem]{\coroname}
\newtheorem{proposition}[theorem]{\propname}

\theoremstyle{definition}
\newtheorem{definition}[theorem]{\definame}

\newlist{hypothesis}{enumerate}{1}
\setlist[hypothesis]{label={\textup{(H\arabic*)}}, ref={(H\arabic*)}, leftmargin=*, widest*=10}

  {\list{}{\leftmargin=.5in\rightmargin=.5in}  \item[]  }%
  {\endlist}

\def\dd{{\rm d}}
\newcommand{\eqdef}{\ensuremath{\stackrel{\mbox{\upshape\tiny def.}}{=}}}

\newcommand{\norm}[1]{\left\lVert#1\right\rVert}
\newcommand{\inner}[1]{\left\langle#1\right\rangle}

\def\1B{{\bf  1}}

\newcommand\diam{\mathop{\rm diam}}
\def\dist{\mathop{\rm dist}}
\def\ddiv{\mathop{\rm div}}

\def\intt{\mathop{\rm int}}

\def\supp{\mathop{\rm supp}}

\def\id{{\mathop{\rm id}}}

\def\divv{{\mathop{\rm div}}}

\newcommand{\cvstar}[1]{\xrightharpoonup[#1]{\star}}
\newcommand{\cvweak}[2]{\xrightharpoonup[#1]{#2}}
\newcommand{\cvstrong}[2]{\xrightarrow[#1]{#2}}

\def\inf{\mathop{\rm inf}}
\def\sup{\mathop{\rm sup}}

\def\min{\mathop{\rm min}}

\def\argmin{\mathop{\rm argmin}}

\newcommand{\mres}{\mathbin{\vrule height 1.6ex depth 0pt width
		0.13ex\vrule height 0.13ex depth 0pt width 1.3ex}}

\newcommand\BV{\mathop{\rm BV}}

\newcommand\proj{\mathop{\rm proj}}

\newcommand\C{\mathscr{C}}

\newcommand{\rhotau}[1]{\varrho^\tau_{#1}}
\newcommand{\mutau}[1]{\mu^\tau_{#1}}
\newcommand{\xtau}[1]{x^\tau_{#1}}
\newcommand{\atau}[1]{a^\tau_{#1}}
\newcommand{\Xtau}[1]{\mathbf{x}^\tau_{#1}}
\newcommand{\Atau}[1]{\mathbf{a}^\tau_{#1}}
\newcommand{\Ztau}[1]{\mathbf{z}^\tau_{#1}}
\newcommand{\x}[1]{\mathbf{x}}

\newcommand{\vtau}[1]{\mathbf{v}^\tau_{(#1)}}

\newcommand{\dwl}{d_{W_2,\ell_2}}

\DeclareMathOperator{\Lag}{\mathrm{}{Lag}}
\DeclareMathOperator{\Pac}{\mathscr{P}_{\mathrm{ac}}}

\title[Wasserstein gradients flow of semi-discrete energies]{Wasserstein gradient flows of semi-discrete energies: evolution of urban areas and
uniform quantization}

\author{João Miguel Machado}
\address{Lagrange Mathematical and Computational Center\\
103 rue de Grenelle\\
Paris, 75007, France}
\email{joao-miguel.machado@ceremade.dauphine.fr}

\begin{document}

\begin{abstract}
We study the Wasserstein gradient flow of semi-discrete energies in the space of probability measures, that is functionals depending on two measures-- one being an absolutely continuous density and the other an atomic measure. These energies appear naturally in the field of urban planning. This is done via the celebrated JKO scheme, for which we prove convergence to a limiting system composed of a parabolic PDE with singular advection coupled with an ODE, also presenting singular dynamics. This is first done under more general assumptions using classical tools, and in a second moment convergence is proven to hold in $L^2_t H^1_x$ for the cases of linear and Porous-Medium type diffusions. We then pass to the study of some qualitative properties of this system, such as the convergence of the atoms towards the baricenters of their corresponding Laguerre cells. We finish this work with extensive numerical simulations that aid in formulating conjectures for the qualitative behavior of this system; in the case of linear diffusion, for instance, we observe a dynamic crystallization phenomenon. 

\bigskip

\noindent\textbf{Keywords.} Optimal Transport, Gradient Flows, Urban Planning, Optimal Quantization

\medskip

\noindent\textbf{2020 Mathematics Subject Classification.} 
49Q22, 35A15, 91B52

\end{abstract}

\maketitle

\tableofcontents

\section{Introduction}\label{sec.introduction}
In the present work we study a class of coupled PDE-ODE system that is relevant in the mathematical modeling for the evolution of an urban area and for the optimal quantization of probability measures. 

A central question in the mathematical modeling of urban systems is how to describe the interplay between a population distributed over a territory and the location of a finite number of working or service sites. A variational approach to this problem was proposed in~\cite{buttazzo2005model} by Buttazzo and Santambrogio, who introduced an energy functional describing the competition of three major effects: congestion of the population through the penalization of its density, the cost of operating working sites, and the global cost of transportation from residences to workplaces. 

In this scenario, the population density is described by an absolutely continuous probability measure $\varrho$, while the distribution of working or service sites is given by an atomic measure $\mu \displaystyle = \sum_{i = 1}^N a_i \delta_{x_i}$, where $x_i$ corresponds to the position of the $i$-th center and $a_i$ the population percentage of population attending it. The energy described above is given by
\begin{equation}\label{eq.energy}
  \mathscr{E}(\varrho, \mu) 
  \eqdef 
  \mathscr{F}(\varrho) + \mathscr{G}(\mu) + W_2^2(\varrho, \mu), 
\end{equation}
where 
\begin{equation}\label{eq.internal_energy}
  \mathscr{F}(\varrho) 
  \eqdef 
  \begin{cases}
    \displaystyle 
     \int_\Omega F(\varrho(x))\dd x,& \text{ if } \varrho \ll \mathscr{L}^d\mres \Omega,\\ 
     +\infty,& \text{ otherwise,}
  \end{cases}
\end{equation}
represents a congestion term for the population density and 
\begin{equation}
  \mathscr{G}(\mu) 
  \eqdef 
  \begin{cases}
    \displaystyle 
     \sum_{i = 1}^N g(a_i),& 
     \text{ if } \displaystyle \mu = \sum_{i = 1}^N a_i \delta_{x_i} \text{ with } {\left(x_i\right)}_{i = 1}^N \subset \overline{\Omega},\\ 
     +\infty,& \text{ otherwise,}
  \end{cases}
\end{equation}
represents the cost of operation of a given atomic distribution of working sites. The coupling term given by the squared-Wasserstein distance represents the global transportation cost of the population to their working sites. See~\cite{villani2009optimal,santambrogio2015optimal} for a definition and properties, or Section~\ref{sec.convergenceMMS} for a concise presentation. Besides the work of Butazzo and Santambrogio, optimal transport techniques have been vastly used to model the coupling of an absolutely continuous probability density and an atomic measure~\cite{bourne2014optimality,bourne2015centroidal,bourne2021asymptotic,bouchitte2011asymptotic}, following therefore in the category of semi-discrete transport, see~\cite{kitagawa2019convergence,merigot2021optimal} for more details on the semi-discrete setting. 




Starting from the variational principle described above, our first goal is to derive a system of evolution equations for the population density, workplaces and the corresponding proportions of the population that work on each site, which is done through the gradient flow of the energy~\ref{eq.energy}. The equation obtained with this methodology is the following
\begin{equation}\label{eq.PDE_ODEsysthem}
  \begin{cases}
      &\displaystyle
      \partial_t \varrho_t 
      = \Delta P(\varrho_t) 
      + 
      \ddiv\left(
        \varrho_t \left(\sum_{i = 1}^N(x - x_i(t))\mathbbm{1}_{\Omega_i(t)}
      \right)\right)\\
      &\displaystyle
      \dot{x}_{i}(t) 
      =
      - a_i(t)x_i(t) + \int_{\Omega_i(t)}x \dd \varrho_t + N_{\Omega}(x_i), \text{ for $i=1,\dots,N$}\\ 
      &\displaystyle 
        \dot{a}_i(t) = \left(-g'(a_i(t)) + \psi_i(t)\right)\mathds{1}_{\{a_i>0\}}, \text{ for $i=1,\dots,N$} \\ 
      &\displaystyle 
      \Omega_i(t) = \Lag_i(\psi_t, \mathbf{x}_t), \ \psi(t) \text{ is a potential for } W_2^2(\varrho_t, \mu_t), 
  \end{cases}
\end{equation}
where $N_\Omega(x)$ the normal cone of $\Omega$ at $x$, and $P(\rho) \eqdef \rho F'(\rho) - F(\rho)$ denotes the pressure. Notice the indicator function multiplying the dynamics of $a_i$, by this we mean that if at a given time $\bar t$ the atom $i$ reaches null mass, then it vanishes from the dynamics from this point on. 

This will be done with a \text{minimizing movement scheme} (MMS) in the product topology induced by the $\ell^2$ Wasserstein distance for the densities $\varrho$ and the $\ell^2$ euclidean norm for the sets of atoms and weights in $\Omega^N$ and the $(N-1)$-dimensional simplex. Although it can be defined in any metric space~\cite{Ambrosio2008GigliSavare}, in the context of optimal transport the MMS is better known as the JKO scheme, first introduced in~\cite{jordan1998variational,otto2001geometry} to interpret the Fokker-Planck and Porous Medium equations as gradient flows in this topology. Since~\cite{Ambrosio2008GigliSavare}, there is a standard theory for gradent flows in the Wasserstein topology with respect to geodesically convex energies in the sense of McCann~\cite{mccann1997convexity}.  Although many interesting properties can be proven for solutions of well-established equations having this variational structure, it has also been used as a tool to prove existence of solutions to more complicated models, see~\cite{carlier2019total,maury2010macroscopic}. In these references, as well as in the present work, although the energy we study is not geodesically convex, we can still define the minimizing movement scheme and study its properties.

System~\ref{eq.PDE_ODEsysthem} can be interpreted as follows: The first equation on $\varrho_t$ describes a competition between the tendency of diffusion of the pressure $P(\varrho_t)$ and a concentration of its mass inside each Laguerre cell $\Omega_i$ towards the atom $x_i$. The evolution describing the atoms' positions is proportional to their distance to the barycenter of there corresponding Laguerre cell. On the other hand, the evolution of their corresponding weights is balanced by the fluctuations between the resistance to growth at the current size $g'(a_i)$ and the Kantorovitch potential $\psi_t$. 

The hole system is interconnected through the optimality of the optimal transportation problem which is encoded by $\psi(t)$ and the definition of Laguerre cells
\[
    \Lag_i(\psi_t, \mathbf{x}_t)
    \eqdef 
    \left\{
        x \in \Omega: 
        \frac{1}{2}|x - x_i(t)|^2 - \psi_{i}(t) 
        \le 
        \frac{1}{2}|x - x_j(t)|^2 - \psi_{j}(t)
        \text{ } j = 1,\dots,N
    \right\}. 
\]
Such optimality conditions have a clear economic equilibrium interpretation; indeed $\psi_i$ can be interpreted as the average salary an individual should expect working at $x_i$. Therefore, an optimal $\psi$ achieves the equilibrim between compensation and geographic displacement under a distribution of employment demand $\mu_t = \displaystyle \sum_{i = 1}^N a_i(t)\delta_{x_i(t)}$ and a geographic distribution of workforce offer given by the population $\varrho_t$. In addition, it is known from the theory of optimal transport that Kantorovitch potentials are unique up to an additive constant; in our case $\psi$ is uniquely determined by the fact that $\displaystyle \sum_{i=1}^N \dot{a}_i(t) = 0$.

\subsubsection*{Hypothesis} In order for this problem to be meaningful from a modelling perspective and mathematically challenging, we make the following hypothesis. 
\begin{hypothesis}[resume]
  \item\label{hypo.domain} $\Omega$ is a convex and bounded subset of $\mathbb{R}^d$;
  \item\label{hypo.congestion_term} $F \in \mathscr{C}^2(\mathbb{R}_+)$, is convex, with superlinear growth
  \[
    \lim_{t \to +\infty} \frac{F(t)}{t} = +\infty
  \]
  and satisfies McCann's condition of displacement convexity, namely
  \[
    t \mapsto t^d F(t^{-d}) \text{ is convex and nonincreasing.}
  \]
  \item 
   $g \in \mathscr{C}^1_{loc}((0,1])$, satisfies $g(0) = 0$ and has a cusp at $0$, in the sense that
  \[
    \lim_{t \to 0^+} \frac{g(t)}{t} = +\infty.
  \]
\end{hypothesis}




By fixing the weights to be constant equal to $a_i \eqdef 1/N$ and the density penalization to be Boltzmann entropy $\mathscr{F}(\varrho) = \mathcal{H}(\varrho) \displaystyle \eqdef \int_{\Omega} \varrho \log \varrho \dd x$, we obtain the following \textit{dynamic quantization equation} 
\begin{equation}\label{eq.dynamic_quantization}
    \begin{cases}
        \partial_t \varrho = 
        \displaystyle
        \Delta \varrho + 
        \divv
        \left(
          \varrho_t 
          \sum_{i = 1}^N(x - x_i(t))\mathbbm{1}_{\Omega_i(t)}
        \right) , &  \\
        \dot{x}_{i}(t)
        =
        b_{i}(t)
        -
        x_{i}(t), \text{ for } i=1,\dots,N,
    \end{cases}
\end{equation}
where $b_i(t) \eqdef \displaystyle N \int_{\Omega_i(t)} x \dd \varrho_t$ corresponds to the barycenter of the $i$-th optimal Laguerre cells.

\subsection{Contributions} 

As mentioned above, one of the major contributions of the present manuscript is the derivation, and proof of existence of weak solutions to the equation~\eqref{eq.PDE_ODEsysthem} via the JKO scheme. The coupling with the discrete system introduces many non-trivial steps. The first difficulty is as dealing with boundary effects, in principle the atoms can evolve on the boundary of $\Omega$ making the normal cone constraints active in their dynamics. What prevents this from happening is the fact that the atoms are pulled towards the barycenter of their Laguerre cell and the uniform in time integrability of $\varrho_t$ conferred by the internal energy functional $\mathscr{F}$. The fact that $g$ has a cusp at $0$ also becomes problematic to the dynamics of $a_i$. Our approach is to show that any limiting curve of the discrete scheme has the property that, whenever an atom reaches $0$, it remains $0$ for the rest of the evolution. Therefore, it suffices to characterize their dynamics in any open interval until the first time they vanish. In any such interval, $g'(a_i(t))$ never blows-up. 

Next, following newer developments in the theory of the JKO scheme~\cite{santambrogio2024strong}, we show that the JKO scheme associated to the energy~\eqref{eq.energy} converges strongly in $L^2([0,T];H^1(\Omega))$, instead of the classical convergence in the strong topology of $L^1([0,T]\times\Omega)$, obtained with more standard tools. This is done when the internal energy $\mathscr{F}$ is either Boltzman's entropy or of porous medium type, that is a diffusion of the form $\Delta \varrho^m$ for some $m > 1$. Even with smoother advections, than the one coming from the semi-discrete coupled dynamics, is an improvement from~\cite{santambrogio2024strong}, since up until now these types of strong convergence results have only been obtained in the linear diffusion regime.  

This through convergence analysis allows us to study the qualitative properties of the equation~\eqref{eq.PDE_ODEsysthem}, as well as some simplified versions of it, where the boundary effects do not take place. This is done in Section~\ref{sec.qualitative_props}, where we show that atoms never touch the boundary, unless at the exact moment that their mass reaches $0$. We also show that if an atom is initialized in the boundary, then it is immediately pushed inside the domain, in other words if $x_i(0) \in \partial \Omega$, then $x_i(t) \in \intt \Omega$ for $t$ in a sufficiently small neighborhood of $0$.  These properties immediately imply that the variant equation~\eqref{eq.dynamic_quantization} is globally well-posed over $\mathbb{R}_+$. In this case, we show in Theorem~\ref{thm.conv_bary} that the distance between atoms and their respective barycenters converge to $0$ as $t\to +\infty$. This is done by combining the dissipation of energy of the gradient flow with a careful analysis of the regularity of the evolution of potentials $t \mapsto \psi_i(t)$, which implies global absolute continuity of the evolution of barycenters. This done in Lemma~\ref{lemma.cont_laguerre}. 

We finish the present work with extensive numerical simulations in Section~\ref{sec.numerics} that not only corroborate the theoretical results we have proven, but also allow us to formulate many conjectures on the long time behavior of the system, for both linear and porous medium type diffusions. 

\section*{Acknowledgments}
The author wishes to thank Guillaume Carlier, Quentin Mérigot and Filippo Santambrogio for suggesting this problem and their numerous remarks that enhanced this work. He also warmly thanks the support of the Lagrange Mathematical and Computational Research Center.

\section{Optimal Transport and Minimizing Movement Schemes}\label{sec.MMS}
In this section we review some well known results in the literature of Optimal Transport, Wasserstein gradient flows and establish some notation used in the sequel. 

\subsection{Optimal Transport, Wasserstein distances and the semi-discrete problem}\label{sec.OT}

Let $\Omega \subset \mathbb{R}^d$ be a compact, convex set, we let $\mathscr{M}(\Omega)$  denote the space of finite (scalar) Radon measures on $\Omega$ and
$
	\mathscr{M}^d(\Omega) \eqdef \bigl(\mathscr{M}(\Omega)\bigr)^d
$
the space of finite vector-valued Radon measures. The set of (Borel) probability measures on $\Omega$ is written as $\mathscr{P}(\Omega)\subset\mathscr{M}(\Omega)$. Given $\mu,\nu \in \mathscr{P}(\Omega)$. The $2$-Wasserstein distance $W_2^2$ is defined via the value function of the quadratic optimal transport problem with the cost $c(x,y) = |x-y|^2$ and admits three equivalent formulations:
\begin{equation}\label{eq.Wasserstein_distance}
    W_2^2(\mu, \nu) 
    \eqdef 
	\min_{\gamma \in \Pi(\mu,\nu)} \int_{\Omega \times \Omega} |x-y|^2 \dd\gamma(x,y)
    =
	\sup_{\substack{\varphi, \psi \in \mathscr{C}_b(\Omega)\\ \varphi \oplus \psi \le c}} 
    \int_\Omega \varphi \dd \mu + \int_\Omega \psi \dd \nu,
\end{equation}
where $\Pi(\mu,\nu)$ is the set of transport plans with marginals $\mu,\nu$. It is well known that this quantity defines a distance in $\mathscr{P}(\Omega)$ which metrizes the narrow convergence of probability measures, \textit{i.e.} convergence in duality with the continuous and bounded functions $\mathscr{C}_b(\Omega)$. The reader is referred to the numerous monographs on the topic for further information~\cite{santambrogio2015optimal,Ambrosio2008GigliSavare,villani2009optimal} for more information. In the sequel we shall give a brief discussion of the relevant properties for the subsequent work. 

For the quadratic case we can always obtain existence of optimal Kantorovich potentials $(\varphi, \psi)$ for the dual problem from~\eqref{eq.Wasserstein_distance}, see for instance~\cite[Prop.~1.11]{santambrogio2015optimal} for a proof in a compact domain. Optimizers for the primal and dual problems are related as 
\[
    \varphi(x) + \psi(y) = |x - y|^2 \quad \gamma-\text{a.e.}
\]
and whenever $\mu \ll \mathscr{L}^d$, Brenier’s theorem says that $\gamma = {(\id, T)}_\sharp\mu$ is concentrated on the graph of a map $T$ which is the gradient of a convex function 
\[
    T = \nabla u, \text{ and $u$ is related to the optimal potential } 
    u = \frac{1}{2}|\cdot|^2 - \varphi.
\]
In this case we also know that the optimal map $T$ is unique and the Kantorovitch potentials $(\varphi, \psi)$ unique up to constants.

The semi-discrete variant of the optimal transportation problem consists of assuming that the first marginal is given by an absolutely continuous measure $\varrho \ll \mathscr{L}^d$ and a discrete measure. Given $\mathbf{x} = {\left(x_i\right)}_{i = 1}^N \in \Omega^N$ and $\mathbf{a}= {\left(a_i\right)}_{i = 1}^N \in \Delta_{N-1}$ the simplex, we let 
\[
    \mu_{\mathbf{x},\mathbf{a}} \eqdef \sum_{i = 1}^N a_i \delta_{x_i},
\]
denote the atomic measure with atoms concentrated over $\mathbf{x}$ and weights $\mathbf{a}$. In this case, dual formulation of the transportation cost can be written as 
\begin{equation}
    \frac{1}{2}
    W_2^2(\varrho, \mu_{\mathbf{x},\mathbf{a}}) 
    = 
    \sup_{\psi \in \mathbb{R}^N} 
    \sum_{i = 1}^N 
    \int_{\Lag_{i}(\mathbf{x}, \psi)}
        \left(
            \frac{1}{2}|x - x_i|^2 - \psi_i
        \right)
    \dd \varrho
    + 
    \sum_{i = 1}^N
    a_i\psi_i,
\end{equation}
where $\mathbf{x} = {\left(x_i\right)}_{i = 1}^N$ and $\Lag_i(\mathbf{x}, \psi)$ denotes the $i$-th \textit{Laguerre cell} induced by $\psi$ and $\mathbf{x}$, defined as
\begin{equation}
    \Lag_i(\mathbf{x}, \psi)
    \eqdef 
    \left\{
        x \in \Omega :  \frac{1}{2}|x - x_i|^2 - \psi_i \le  \frac{1}{2}|x - x_j|^2 - \psi_j, \ \forall j = 1,\dots,N
    \right\}.
\end{equation}
The maximizer is the unique $\psi \in \mathbb{R}^N$, up to adding a constant vector, such that $\varrho(\Lag_i(\mathbf{x}, \psi)) = a_i$ for all $i = 1,\dots, N$. These facts can be obtained via Brenier’s Theorem applied to the semi-discrete case. Throughout the text, we shall fix the notation 
\[
    \Omega_i = \Lag_i(\mathbf{x}, \psi), 
    \quad 
    \Sigma_{ij} = 
    \partial\Lag_i(\mathbf{x}, \psi)\cap 
    \partial\Lag_j(\mathbf{x}, \psi) \text{ for } i\neq j,
\]
whenever the context is clear for instance the Laguerre tesselation is obtained from $W_2^2(\varrho, \sum a_i \delta_{x_i})$ and $\psi$ is optimal, and define the corresponding optimal Kantorovich potential induced by this tesselation
\begin{equation}\label{eq.potential_tesselation}
    \Phi[\mathbf{x},\psi](x) 
    \eqdef 
    \sum_{i = 1}^N \left[
        \frac{1}{2}|x - x_i|^2 - \psi_i
    \right]\mathds{1}_{\Omega_i}(x).
\end{equation}

Since Laguerre cells are convex polygons whose faces are determinate by the points $\mathbf{x}$ and the potential $\psi$, they are stable with respect to these quantities, as stated in the following. 
\begin{lemma}\label{lemma.cont_laguerre}
    Given a sequence ${(\mathbf{x}_n,\psi_n)}_{n \in \mathbb{N}}$ converging to $(\mathbf{x},\psi)$, it holds that
    \[
        \Lag_i(\mathbf{x}_n, \psi_n) \cvstrong{n \to \infty}{L^1} \Lag_i(\mathbf{x}, \psi). 
    \]
\end{lemma}

It will be particularly important for us to perform variations of the optimal transport cost with respect to a measure, or in the semi-discrete case compute gradients with respect to the atoms. The former can be found for instance in~\cite[Prop.~7.17]{santambrogio2015optimal}. To the derivative of the semi-discrete cost can be expressed in terms of the barycenter of the optimal Laguerre tesselation. Therefore, we introduce the following notation 
\begin{equation}\label{eq.barycenter_def}
    \mathbf{b}[\mathbf{x},\psi] = {\left(b_i[\mathbf{x},\psi]\right)}_{i = 1}^N 
    \text{ where }
    b_i[\mathbf{x},\psi]
    \eqdef 
    \fint_{\Lag_i(\mathbf{x}, \psi)} x \dd \varrho(x).
\end{equation}
Differentiating the semi-discrete transport cost is a very delicate matter, the major issue arises when two points coincide, with the derivatives being more and more singular as the become too close. For this it is useful to introduce the following notation for the generalized diagonal 
\begin{equation}\label{eq.generalized_diagonal}
    \mathbb{D}_{N,\varepsilon} 
    \eqdef 
    \left\{
        |x_i - x_j| \le \varepsilon \text{ for some pair } i\neq j
    \right\},
\end{equation} 
In the sequel we summarize the results which are relevant to us.

\begin{lemma}\label{lemma.first_variation}
    Let $\varrho \in \Pac(\Omega)$, then the following hold: 
    \begin{itemize}
        \item[(i)] The first variation of $W_2^2(\varrho, \nu)$ is given by the unique, up to constants, Kantorovitch potential $\varphi$ w.r.t.~$\varrho$, \textit{i.e.} for any $\bar \varrho \in \Pac(\Omega)$ we have
        \[
            \left. \frac{\dd}{\dd \varepsilon}\right|_{\varepsilon = 0^+}
            W_2^2(\varrho + \varepsilon (\bar\varrho - \varrho), \nu) 
            = 
            \int_{\Omega} \varphi \dd (\bar\varrho - \varrho).
        \]
        \item[(ii)] In the semi-discrete transport case, for $\varrho \in \Pac(\Omega)$ fixed, set
        \[
            F(\mathbf{x}, \mathbf{a}) 
            \eqdef 
            \frac{1}{2}W_2^2(\varrho, \mu_{\mathbf{x},\mathbf{a}}).
        \]
        If $\mathbf{x} \in \Omega^N \setminus \mathbb{D}_{N,\varepsilon}$ and $a_i > 0$ for all $i=1,\dots,N$ then $\mathbf{x} \mapsto F(\mathbf{x},\mathbf{a})$ is twice differentiable, with gradients given by
        \[
            \nabla_{x_i}
            F(\mathbf{x},\mathbf{a}) 
            = 
            \int_{\Omega_i}(x_i - x)\dd \varrho
            =
            a_i\left(x_i - b_i[\mathbf{x},\psi]\right),
        \]
        where $\psi$ is the unique Kantorovitch potential associated with $\mu_{\mathbf{x},\mathbf{a}}$. The second derivatives are given by: for $i \neq j$
        \[
            \nabla_{x_i,x_j} F(\mathbf{x},\mathbf{a})
            = 
            \int_{\Sigma_{ij}} 
            (x - x_i) \otimes (x - x_j) \frac{\varrho(x)}{|x_i - x_j|} \dd \mathscr{H}^{d-1}(x),
        \] 
        and in the diagonal for all $i = 1,\dots, N$
        \[
            \nabla_{x_i,x_i} F(\mathbf{x},\mathbf{a})
            = 
            a_i I_d - \sum_{j \neq i} \nabla_{x_i,x_j} F(\mathbf{x},\mathbf{a}).
        \]
        In addition, if $\varrho$ is continuous, then $\mathbf{x}\mapsto F(\mathbf{x},\mathbf{a})$ is $\mathscr{C}^2$.
    \end{itemize}
\end{lemma}

We do not give a proof for this result, for item $(i)$ the reader is referred to~\cite[Prop.7.17]{santambrogio2015optimal}, while for item $(ii)$ the reader is referred to~\cite{de2019differentiation}. In the latter, there is a simple argument using the enveloppe theorem: letting $\psi$ denote an optimal Kantorovich potential we have 
\begin{align*}
    \nabla_{x_i} F(\mathbf{x}, \mathbf{a}) 
    &= 
    \nabla_{x_i} 
    \left\{
        \int_{\Omega} \min_{j = 1,\dots, N} 
        \left(
            \frac{1}{2}|x - x_j|^2 - \psi_j
        \right)\dd \varrho
        + 
        \sum_{j = 1}^N \psi_j a_j
    \right\}\\ 
    &= 
    \int_{\Lag_i(\mathbf{x}, \psi)}
        \left(
            x_i - x
        \right)\dd \varrho
    = 
    a_i(x_i - b_i[\mathbf{x},\psi]).
\end{align*}
With this characterization of the gradient there is a very simple argument to see that $\mathbf{x} \mapsto F(\mathbf{x},\mathbf{a})$ is semi-concave, see also~\cite{merigot2016minimal}. For simplicity, let us take $a_i = 1/N$ for all $i$, let $F(\mathbf{x})$,$F(\mathbf{y})$ denote the minimal cost of transporting $\varrho$ to the empirical measures over $\mathbf{x}$ and $\mathbf{y}$, respectively. Let $\Omega_i$ denote the optimal Laguerre tesselation associated with $F(x)$, then setting $\bar F(\mathbf{x}) = F(\mathbf{x}) - \frac{1}{2N}\norm{\mathbf{x}}^2$ we have
\begin{align*}
    F(\mathbf{y}) 
    &\le \sum_{i = 1}^N \int_{\Omega_i} \frac{1}{2}|x - y_i|^2\dd \varrho  
    = F(\mathbf{x}) + \sum_{i = 1}^N \int_{\Omega_i} \left(\frac{1}{2}|x - y_i|^2 - \frac{1}{2}|x - x_i|^2\right) \dd \varrho\\ 
    &\le 
    F(\mathbf{x}) - \frac{1}{N}\sum_{i = 1}^N b_i\cdot(y_i-x_i) + \frac{1}{2N} \sum_{i = 1}^N |y_i|^2 - |x_i|^2.
\end{align*}
As a result 
\[
    \bar F(\mathbf{y}) \le \bar F(\mathbf{x}) + \nabla_{\mathbf{x}}\bar F(\mathbf{x})\cdot(\mathbf{y} - \mathbf{x}),
\]
which shows that $\mathbf{x} \mapsto F(\mathbf{x})$ is $1/(2N)$ semi-concave.

\subsection{Minimizing movement schemes}
Our approach to show existence of the PDE-ODE system, we use the so called \textit{minimizing movement scheme}. It consists of an implicit Euler scheme that can be easily adapted to metric spaces, the reader is referred to~\cite{Ambrosio2008GigliSavare} for more details on this general formulation.

For simplicity, we introduce the notation $\mathbf{z} = (\varrho, \mathbf{x}, \mathbf{a}) \in \Pac(\Omega)\times \Omega^{\otimes N} \times \Delta_{N-1}$, where $\mathbf{x} = {(x_i)}_{i = 1}^N$ and $\mathbf{a} = {(a_i)}_{i = 1}^N$. The set $\Delta_{N-1}$ denotes the $(N-1)$-dimensional simplex. Consider the distance between $\mathbf{z} = (\varrho, \mathbf{x}, \mathbf{a})$ and $\mathbf{z}' = (\varrho', \mathbf{x}', \mathbf{a}')$ in this product space defined as
\begin{equation}
    d_{W_2,\ell_2}^2(\mathbf{z},\mathbf{z}') 
    \eqdef 
    W_2^2(\varrho, \bar\varrho) + 
    |\mathbf{x} - \mathbf{x}'|^2 + 
    |\mathbf{a} - \mathbf{a}'|^2,
\end{equation}
where $|\cdot|$ denotes the euclidean distance. Given some $\mathbf{z} = (\varrho, \mathbf{x}, \mathbf{a})$, and we make the slight abuse of writing $\mathscr{E}(\mathbf{z}) = \mathscr{E}(\varrho, \mu_{\mathbf{x},\mathbf{a}})$. 

A \textit{minimizing movement scheme} in this topology consists of a sequence ${\left(\rhotau{k}, \Xtau{k}, \Atau{k}\right)}_{k \in \mathbb{N}}$ such that ${\left(\rhotau{0}, \Xtau{0}, \Atau{0}\right)} = {\left(\varrho_0, \mathbf{x}_0, \mathbf{a}_0\right)}$ is given and for $k \ge 0$ we have 
\begin{equation}\label{eq.MMS_def}
    \begin{multlined}
        {\left(\rhotau{k+1}, \Xtau{k+1}, \Atau{k+1}\right)}_{k \in \mathbb{N}} 
        \in 
        \argmin 
        \mathscr{F}(\varrho) 
        + 
        \mathscr{G}(\mu_{\mathbf{x},\mathbf{a}})
        +
        W_2^2(\varrho, \mu_{\mathbf{x},\mathbf{a}})\\
        \qquad \qquad \qquad \qquad \qquad+ 
        \frac{1}{2\tau}
        \left(
            W_2^2(\rhotau{k}, \varrho)
            + 
            \norm{\Xtau{k} - \mathbf{x}}^2
            +
            \norm{\Atau{k} - \mathbf{a}}^2
        \right).
    \end{multlined}
\end{equation}
It is not difficult to show existence of minimizers for each step of this scheme, however we cannot in general expect uniqueness due to the fact that this energy is semi-concave in $\mathbf{x}$, as discussed in the previous section. Whenever the context is clear we will omit the dependence on $\tau$, and we use also the notation $\mu_k = \mu_{\mathbf{x}_k,\mathbf{a}_k}$ to designate the atomic measure induced by an element of the minimizing movement scheme.  

The following result comes directly from the definition and is standard. 
\begin{lemma}\label{lemma.a_priori_estimates}
    Assuming that $\mathscr{E}(\varrho_0, \mu_0) < +\infty$, there is a constant $C > 0$ depending only on $\Omega$ and $T$ such that
    \begin{align}
        \label{eq.a_priori_energy}
        \mathscr{E}(\rhotau{k+1}, \mutau{k+1}) \le 
        \mathscr{E}(\rhotau{k}, \mutau{k})
        \le \dots \le 
        \mathscr{E}(\varrho_0, \mu_0)
        &\le C, \quad \text{ for all } k\in \mathbb{N}, \tau > 0\\
        \label{eq.a_priori_distance} 
        \sum_{k = 0}^{N_\tau - 1}
        \dwl^2(\Ztau{k}, \Ztau{k+1}) 
        &\le C\tau. 
    \end{align}
\end{lemma}
\begin{proof}
    To prove~\eqref{eq.a_priori_energy}, use $(\rhotau{k}, \mutau{k})$ as a competitor for the problem defining $(\rhotau{k+1}, \mutau{k+1})$ to obtain for all $k$ that
    \[
        \frac{1}{2\tau}\dwl^2(\Ztau{k}, \Ztau{k+1}) 
        \le 
        \mathscr{E}(\rhotau{k}, \mutau{k})
        - 
        \mathscr{E}(\rhotau{k+1}, \mutau{k+1}).
    \]
    Since the distance $\dwl$ is non-negative~\eqref{eq.a_priori_energy} follows from a simple induction. 

    Conversely, sum the above inequality in $k$, since the right-hand side telescopes and $\mathscr{E} \ge 0$, we get that
    \[
        \sum_{k = 0}^{N_\tau - 1}
        \dwl^2(\Ztau{k}, \Ztau{k+1}) 
        \le 2\mathscr{E}(\varrho_0, \mu_0)\tau,
    \]
    and the result follows. 
\end{proof}

\subsection{Two interpolations} 
Given a time interval $[0,T]$ and a sequence ${\left(\rhotau{k}, \Xtau{k}, \Atau{k}\right)}_{k \in \mathbb{N}}$ obtained via the minimizing movement scheme, we consider two time interpolations namely
\begin{itemize}
    \item \textbf{Staircase interpolation:} 
    \begin{equation}
        (\varrho^\tau_t, \mathbf{x}^\tau_t, \mathbf{a}^\tau_t)
        \eqdef 
        \left(\rhotau{k}, \Xtau{k}, \Atau{k}\right), \text{ if } t \in (k\tau, (k+1)\tau]
    \end{equation}
    \item  \textbf{Geodesic interpolation:} The second family of curves ${\left(\bar \varrho^\tau,  \mathbf{\bar x}^\tau, \mathbf{\bar a}^\tau \right)}_{\tau > 0}$ now contained in $\C([0,T]; \Pac(\Omega))$ is defined as the geodesic between $\left(\rhotau{k}, \Xtau{k}, \Atau{k}\right)$ and $\left(\rhotau{k+1}, \Xtau{k+1}, \Atau{k+1}\right)$ at times $k\tau$ and $(k+1)\tau$ in the topology induced by $\dwl$. Hence, setting 
    \begin{equation}
        s_t \eqdef \frac{1}{\tau}((k+1)\tau - t), \ T_{k+1}^\tau \text{ is the OT map from $\rhotau{k+1}$ to $\rhotau{k}$},
    \end{equation}
    we define the geodesic interpolation as
    \begin{align}
        \bar \varrho^\tau_t 
        &\eqdef
        {\left(
            T_{k+1}^\tau(t)
        \right)}_\sharp\varrho_{k+1}^\tau,
        \text{ with }
        T_{k+1}^\tau(t) \eqdef 
        \left((1 - s_t) \id
        + 
        s_t T_{k+1}^\tau \right)
        \\ 
        \mathbf{\bar x}^\tau_t 
        &\eqdef 
        (1 - s_t) \mathbf{x}_{k+1}^\tau
        + 
        s_t \mathbf{x}_{k}^\tau\\ 
        \mathbf{\bar a}^\tau_t 
        &\eqdef
        (1 - s_t) \mathbf{a}_{k+1}^\tau
        + 
        s_t \mathbf{a}_{k}^\tau.
    \end{align}
\end{itemize}

While the staircase interpolation is easier to relate to the optimality conditions, see Prop.~\ref{prop.optimality_conditions}, the geodesic interpolation solves by construction the continuity equation with a suitable velocity, as this is a general property of geodesics in the Wasserstein space, see~\cite[Thm.~5.14]{santambrogio2015optimal} and~\cite[Chap.~8]{Ambrosio2005Gradient}. 

Hence we define also the set of velocities, starting with a staircase velocity field
\begin{equation}
    \textbf{v}^\tau_t = \vtau{k+1} = \frac{\id - T_{k+1}^\tau}{\tau}, \text{ for } t \in (k\tau, (k+1)\tau].
\end{equation}
The velocity of the geodesic interpolation can be computed as follows: for an arbitrary $\phi \in \C_c^\infty(\Omega)$ and $t \in (k\tau, (k+1)\tau]$ we have
\begin{align*}
    \frac{\dd}{\dd t}
    \int_\Omega \phi \dd \bar \varrho^\tau_t
    &= 
    \frac{\dd}{\dd t}
    \int_\Omega \phi\circ T_{k+1}^\tau(t) \dd \varrho_{k+1}^\tau 
    = 
    \int_\Omega \nabla \phi\circ T_{k+1}^\tau(t) \cdot 
    \left(\frac{\id - T_{k+1}^\tau}{\tau}\right) \dd \varrho_{k+1}^\tau\\
    &= 
    \int_\Omega \nabla \phi \cdot 
    \underbrace{
        \left(\frac{\id - T_{k+1}^\tau}{\tau}\right)\circ {\left(T_{k+1}^\tau(t)\right)}^{-1} 
    }_{
        = \mathbf{v}^\tau_t
    }
    \dd \bar \varrho^\tau_t.
\end{align*}
Therefore, the geodesic velocity interpolation can be equivalently rewritten as
\begin{equation}
    \mathbf{v}^\tau_t \eqdef \vtau{k+1}\circ{\left(\id + (t - (k+1)\tau)\vtau{k+1}\right)}^{-1}, 
    \text{ for } 
    t \in (k\tau, (k+1)\tau],
\end{equation}
and $(\varrho^\tau, \mathbf{v}^\tau)$ solve the continuity equation. 

This feature can also be expressed in terms of the momentum variables, that is the vector measures in $\mathscr{M}^d([0,T]\times\Omega)$, defined as
\begin{equation}
    E^\tau \eqdef \mathbf{v}^\tau \varrho^\tau
    \text{ and }
    \bar E^\tau \eqdef \mathbf{\bar v}^\tau\bar \varrho^\tau,
\end{equation} 
it holds trivially that
\[
    \partial_t \bar \varrho^\tau_t + \ddiv \bar E^\tau = 0, \text{ for all $\tau>0$},
\]
in the sense of distributions.

Using the a priori estimates from Lemma~\ref{lemma.a_priori_estimates} we have the following result which says that both families of interpolations ${\left(
    \varrho^\tau, \mathbf{x}^\tau, \mathbf{a}^\tau, E^\tau
\right)}_{\tau > 0}$ and ${\left(
    \bar \varrho^\tau, \mathbf{\bar x}^\tau, \mathbf{\bar a}^\tau, \bar E^\tau
\right)}_{\tau > 0}$ enjoy compactness properties and their limits coincide. 

\begin{proposition}[Chap.8~of~\cite{santambrogio2015optimal}]\label{prop.interpo_apriori_limits}
    There is a curve
    \[
        (\varrho_t, \mathbf{x}_t, \mathbf{a}_t, E_t)_{t \in [0,T]} 
        \in 
        \mathscr{C}^{0,1/2}([0,T]; \Pac(\Omega)\times \overline{\Omega} \times \Delta_{N-1}) \times \mathscr{M}^d([0,T]\times\Omega)
    \]
    which is the limit, up to subsequences that are not relabelled, of both families of interpolations ${\left(
    \varrho^\tau, \mathbf{x}^\tau, \mathbf{a}^\tau, E^\tau
    \right)}_{\tau > 0}$ and ${\left(
        \bar \varrho^\tau, \mathbf{\bar x}^\tau, \mathbf{\bar a}^\tau, \bar E^\tau
    \right)}_{\tau > 0}$.  More precisely, it holds that
    \begin{align*}
        \varrho^\tau, \bar \varrho^\tau 
        &\cvstrong{\tau \to 0}{W_2} \varrho 
        \text{ uniformly in } [0,T],\\
        (\mathbf{x}^\tau, \mathbf{a}^\tau), (\mathbf{\bar x}^\tau, \mathbf{\bar a}^\tau )
        &\cvstrong{\tau \to 0}{} (\mathbf{x},\mathbf{a})
        \text{ uniformly in } [0,T],\\ 
        E^\tau, \bar E^\tau 
        &\cvweak{\tau \to 0} E 
        \text{ in the narrow topology of } \mathscr{M}^d([0,T]\times \Omega).
    \end{align*}

    In addition, $\partial_t \varrho_t + \ddiv E_t = 0$, in the sense of distributions. 
\end{proposition}

\section{Convergence of MMS}\label{sec.convergenceMMS}
Given the results of Section~\ref{sec.MMS}, in other to prove the convergence of the minimizing movement scheme to the coupled system~\eqref{eq.PDE_ODEsysthem}, we only need to characterize the limit momentum variable $E$ and the dynamics of the variables $\mathbf{x},\mathbf{a}$. This will be done in the sequel with the Euler-Lagrange equations characterizing the optimality of the minimizing movement. 

However, there are still multiple difficulties in characterizing the limit curve. To deal with the non-linear term $\Delta P(\varrho)$, we will require a stronger convergence of $\varrho^\tau$ to $\varrho$, which can be done with well-established results in the literature. 

The evolution of the discrete measure is more subtle. Regarding the evolution of $\mathbf{x}$, we need to be careful with the boundary effects, since the points are restricted to $\Omega$, the optimality conditions must push points to the interior of the domain if the minimizer of any step is on the boundary. These boundary effects however cannot be passed onto the limit as $\tau \to 0$, unless we assume that $\partial \Omega$ is smooth. Instead, we show in Prop.~\ref{prop.atoms_interior} that if at an iteration of the minimizing movement scheme a point is in the interior of $\Omega$, then it remains in the interior as long as its mass remains positive. 

This leads us to the heart of the matter: what about the masses $\mathbf{a}$? As the dynamics around $\mathbf{a}$ contain the singular term $g'$, it is not clear what happens to a point when its mass reaches $0$. We show in Prop.~\ref{prop.estimate_ai} that if ${(\varrho_t, \mathbf{x}_t, \mathbf{a}_t)}_{t \ge 0}$ is a curve obtained from the minimizing movement scheme, then if at some time $a_i(t) = 0$, then it remains null of any $s > t$. This means that we only need to characterize the dynamics of $a_i$ until the first time that it vanishes.

\subsection{Optimality conditions}\label{sec.optimality_conditions}

In order to derive optimality conditions, we will need the notion of \textit{Bouligand's tangent cone}
\begin{equation}\label{eq.tangent_cone}
  T_\Omega(z)
  \eqdef 
  \left\{
    v = \lim_{n \to \infty} v_n: 
    z + \varepsilon_n v_n \in \Omega, \ \varepsilon_n \to 0
  \right\}
\end{equation}
see~\cite[Def.~5.1.1]{hiriart1996convex}, which gives the admissible directions to perform variations in $\Omega$. Likewise, the \textit{normal cone} is then given by the polar cone of $T_\Omega(z)$, see for instance~\cite[Part I, Prop.~5.2.4]{hiriart1996convex}. In other words, it can be written as
\begin{equation}\label{eq.normal_cone}
  N_\Omega(z)
  \eqdef 
  \left\{
    n \in \mathbb{R}^d : \inner{n,v} \le 0 \text{ for all } v \in T_\Omega(z)
  \right\}.
\end{equation}
Whenever $\Omega$ has a smooth boundary $N_\Omega(z)$ is generated by the outwards normal vector. 

\begin{proposition}\label{prop.optimality_conditions}
    Let $\left(\rhotau{k+1}, \Xtau{k+1}, \Atau{k+1} \right)$ be a minimizer of~\eqref{eq.MMS_def}, and define the set 
    \begin{equation}\label{eq.inactive_set_weights}
      I_{k+1} 
      \eqdef 
      \left\{
        i = 1,\dots,N : 0 < a_{k+1,i}^\tau < 1
      \right\},
    \end{equation}
    corresponding to the indexes for which the constraints on the weights are not active. Then it holds that
    \begin{equation*}
      \begin{cases}
        \displaystyle
        -\vtau{k+1}
        = 
        \displaystyle
        \nabla F'(\rhotau{k+1}) 
        + 
        \sum_{i= 1}^N (x-\xtau{k+1,i})\mathbbm{1}_{\Omega^\tau_{k+1,i}}& 
        \rhotau{k+1}-a.e.
        \\
        & \\ 
        \displaystyle
        0 
        \displaystyle
        \in
        \left(
          \frac{\xtau{k+1,i} - \xtau{k,i}}{\tau}
          +
          \atau{k+1,i}\cdot\xtau{k+1,i} 
          - \int_{\Omega^\tau_{k+1,i}}x \dd \rhotau{k+1}
        \right) + 
        N_{\Omega}(x_{k+1,i}) ,& \text{} i = 1,\dots,N\\ 
        & \\ 
        \displaystyle 
        -\frac{\atau{k+1,i} - \atau{k,i}}{\tau} 
        =  
        g'(\atau{k+1,i}) + \psi^\tau_{k+1,i},
        & \text{ for all } i \in I_{k+1}\\ 
        & \\ 
        \displaystyle 
        \Omega^\tau_{k+1,i} = \Lag_i\left(\psi^\tau_{k+1}, \Xtau{k+1}\right),& 
        \text{} i = 1,\dots,N.
      \end{cases}
    \end{equation*}
    where $\psi_{k+1}^\tau$ is a Kantorovitch potential associated with $\mutau{k+1}$ in $W_2^2\left(\rhotau{k+1},\mutau{k+1}\right)$.
  \end{proposition}
  In the proof of this result, since $\tau$ is fixed we omit shall it, in order to simplify notation.   
\begin{proof}
    The optimality conditions for $\varrho_{k+1}$ are very similar to the classical theory in bounded domains, see~\cite[Chap.~7,8]{santambrogio2015optimal}, hence we focus mostly on the analysis of $\mathbf{x}_{k+1}$, and $\mathbf{a}_{k+1}$. 

    \textbf{\underline{Optimality conditions for $\varrho$:}}
      Since for any $k \in \mathbb{N}$, the internal energy $\mathscr{F}$ enforces that $\varrho_{k+1} \ll \mathscr{L}^d$, we can use the first variation formulas from Lemma~\ref{lemma.first_variation}. Let us first derive the optimality conditions for $\varrho_{k+1}$, so given $\rho \in \Pac(\Omega)$ consider a variation $\varrho_\varepsilon \eqdef \varrho_{k+1} + \varepsilon(\rho - \varrho_{k+1})$. Optimality gives that
      \[
          -\left. 
          \frac{1}{2\tau}  
          \frac{\dd}{\dd \varepsilon}\right|_{\varepsilon = 0^+}
          W_2^2\left(\varrho_k, \varrho_\varepsilon\right)
          \le 
          \left. \frac{\dd}{\dd \varepsilon}\right|_{\varepsilon = 0^+}
          \mathscr{F}(\varrho_\varepsilon)
          +
          \left. \frac{\dd}{\dd \varepsilon}\right|_{\varepsilon = 0^+}
          \frac{1}{2}W_2^2\left(\varrho_\varepsilon, \mu_{k+1}\right). 
      \]
      Using Lemma~\ref{lemma.first_variation}, we obtain that
      \[
        -\frac{\psi_{k+1}^\varrho}{\tau}
        = 
        F'(\varrho_{k+1}) 
        + 
        \varphi_{k+1}^{\varrho,\mu},
      \]
      where $\psi_{k+1}^\varrho$ and $\varphi_{k+1}^{\varrho,\mu}$ are the Kantorovitch potentials associated with $\varrho_{k+1}$ in $W_2^2(\varrho_k, \varrho_{k+1})$ and $W_2^2(\varrho_{k+1},\mu_{k+1})$, respectively.
  
      On the other hand, from the optimality conditions of the optimal transportation problem, it holds that
      \[
        \varphi_{k+1}^{\varrho,\mu}(x) + \psi_{k+1,i} = \frac{1}{2}\norm{x - x_{k+1,i}}^2, \quad \text{ for } \varrho_{k+1}-a.e.~ x \in \Omega_{k+1,i}. 
      \]
      Since $\Omega_{k+1,i}$ is a convex polyhedra, its boundary is $\varrho_{k+1}$ negligible, and it follows that
      \[
        \nabla \varphi_{k+1}^{\varrho,\mu}(x) 
        = 
        \sum_{i = 1}^N (x - x_{k+1,i}) \mathbbm{1}_{\Omega_{k+1,i}}.
      \]
      In addition, from Brenier's Theorem we know that the optimal map from $\varrho_{k+1}$ to $\varrho_k$ is given by $T_{k+1} = \id - \nabla \psi_{k+1}^{\varrho}$. Combining these results we obtain that
      \[
        \frac{T_{k+1} - \id}{\tau}
        = 
        \nabla F'(\varrho_{k+1}) + 
        \sum_{i= 1}^N (x-x_{k+1,i})\mathbbm{1}_{\Omega_{k+1,i}}.
      \]

      \textbf{\underline{Optimality conditions for $\mathbf{x}$:}}
      To derive the optimality conditions w.r.t.~the atoms $\mathbf{x}_{k+1} = {\left(x_{k+1,i}\right)}_{i=1}^N$, fix some $i$ and consider a direction in the \textit{Bouligand's tangent cone}, $v \in T_\Omega(x_{k+1,i})$, where
      \[
        T_\Omega(z)
        \eqdef 
        \left\{
          v = \lim_{n \to \infty} v_n: 
          z + \varepsilon_n v_n \in \Omega, \ \varepsilon_n \to 0
        \right\},
      \]
      see~\cite[Def.~5.1.1]{hiriart1996convex}. Hence, take $\varepsilon_n$ and $v_n$ converging to $v$ as in the definition above and set 
      \[
        \bar x_{j,\varepsilon_n} 
        \eqdef 
        \begin{cases}
          x_{k+1,j},& \text{ if } j\neq i,\\ 
          x_{k+1,i} + \varepsilon_n v_n,& j = i,
        \end{cases}
      \]
      and the define a variation of the atomic measures as 
      \[
        \mu_{\varepsilon_n} \eqdef 
        \sum_{j = 1}^N a_{k+1,j} \delta_{\bar x_{j,\varepsilon_n}}. 
      \]
      Comparing the energies of $\mu_{k+1}$ and $\mu_{\varepsilon_n}$ with $\varrho_{k+1}$ fixed, we obtain from Lemma~\ref{lemma.first_variation} that
      \begin{align*}
        0 
        &\le 
        \lim_{n \to \infty}
        \left(
          \frac{1}{2\tau}\frac{
            \norm{\bar x_{i,\varepsilon_n}  - x_{k,i}}^2
            -
            \norm{x_{k+1,i} - x_{k,i}}^2
          }{\varepsilon_n}
          + 
          \frac{
            W_2^2(\varrho_{k+1},\mu_{\varepsilon_n})
            -
            W_2^2(\varrho_{k+1},\mu_{k+1})
          }{\varepsilon_n}
        \right)  \\ 
        &= 
        \lim_{n \to \infty}
        \left(
        \inner{
          \frac{x_{k+1,i} - x_{k,i}}{\tau}, v_n
        }
        + 
        \partial_{x_i} \frac{1}{2} W_2^2
        \left(
            \varrho_{k+1}, \mu_{k+1}
        \right)\cdot v_n + \frac{o(\varepsilon_n)}{\varepsilon_n}
        \right)\\ 
        &=
        \inner{
          \frac{x_{k+1,i} - x_{k,i}}{\tau}
          +
          \int_{\Lag_i(\psi_{k+1}; \mathbf{x}_{k+1})}(x_{k+1,i} - x)\dd \varrho_{k+1}
          , v
        }.
      \end{align*}
      Recalling, for instance from~\cite[Part I, Prop.~5.2.4]{hiriart1996convex}, that the normal space $N_\Omega(x_{k+1,i})$ is given by the polar cone to $T_\Omega(x_{k+1,i})$, that $\Omega_{k+1,i} = \Lag_i(\psi_{k+1},X_{k+1})$ and $\varrho_{k+1}(\Omega_{k+1,i}) = a_{k+1,i}$ we obtain that for all $i = 1,\dots,N$
      \[
        -
        \left(
          \frac{x_{k+1,i} - x_{k,i}}{\tau} 
          + a_{k+1,i}\cdot x_{k+1,i} 
          - \int_{\Omega_{k+1,i}}x \dd \varrho_{k+1}
        \right)
        = 
        n_{k+1,i}
        \in N_{\Omega}(x_{k+1,i}).
      \]


      \textbf{\underline{Optimality conditions for $\mathbf{a}$:}}
      In the sequel, we perform variations for the mass variables ${\left(a_{k+1,i}\right)}_{i = 1}^N$. Consider two indexes $i,j \in I_{k+1}$ and construct the variation $\mathbf{a}_\varepsilon = {(a_{\varepsilon,h})}_{h = 1}^N$ given by
      \[
        a_{\varepsilon,h} 
        \eqdef 
        \begin{cases}
          a_{k+1,h},& h \neq i,j,\\ 
          a_{k+1,j} - \varepsilon,& h = j,\\ 
          a_{k+1,i} + \varepsilon,& h = i. 
        \end{cases}
      \]
      And define the new measure $\mu_\varepsilon \eqdef \mu_{\mathbf{x}_{k+1}, \mathbf{a}_\varepsilon}$. For $\varepsilon$ small enough this variation belongs to the simplex $\Delta_{N-1}$ and hence is admissible. As a result, comparing the energies of $(\varrho_{k+1}, \mu_{k+1})$ and $(\varrho_{k+1}, \mu_{\varepsilon})$ we obtain that 
      there is a Kantorovitch potential $\bar \psi_{k+1}$ such that 
      \[
        \frac{a_{k+1,i} - a_{k,i}}{\tau} + g'(a_{k+1,i}) + \bar\psi_{k+1,i}
        \ge
        \frac{a_{k+1,j} - a_{k,j}}{\tau}  + g'(a_{k+1,j}) + \bar \psi_{k+1,j}.
      \]
      In addition, we can assume that $\bar \psi_{k+1}$ is bounded in $\ell^\infty$-norm by a constant $C_\Omega$, since by Brenier's Theorem it the evaluation of a Lipschitz function over the points $\{x_{k+1,i}\}$ and uniquely defined up to constants, so that we can assume that $\bar \psi_{k+1,1} = 0$ and the bound follows from the Lipschitz continuity and the fact that $\Omega$ has finite diameter. 

      Changing the role of $i$ and $j$, it follows that we have the equality
        \begin{multline*}
          \frac{a_{k+1,i} - a_{k,i}}{\tau} + g'(a_{k+1,i}) + \bar\psi_{k+1,i}
        \\ = 
        \frac{a_{k+1,j} - a_{k,j}}{\tau}  + g'(a_{k+1,j}) + \bar \psi_{k+1,j} 
        = c_{k+1}, \text{ for all } i,j \in I_{k+1}.
      \end{multline*}
      The constant $c_{k+1}$ above can be obtained by averaging all these quantities, namely 
      \begin{align*}
        c_{k+1} 
        &= 
        \frac{1}{|I_{k+1}|}
        \left(
        \sum_{j \in I_{k+1}} \bar \psi_{k+1,j} + g'(a_{k+1,j}) + 
        \sum_{j \in I_{k+1}} \frac{a_{k+1,j} - a_{k,j}}{\tau}
        \right)\\ 
        &=
        \frac{1}{|I_{k+1}|}
        \left(
        \sum_{j \in I_{k+1}} \bar \psi_{k+1,j} + g'(a_{k+1,j}) + 
        \sum_{j \in I_k\setminus I_{k+1}} \frac{a_{k,j}}{\tau}
        \right),
      \end{align*}
      where the second equality comes from the fact that both $a_{k+1,i}$ and $a_{k,i}$ sum to $1$.

      As a result, by adding the constant $c_{k+1}$ to the Kantorovitch potentials, \textit{i.e.}
      $
        \psi_{k+1,i} \eqdef \bar \psi_{k+1,i} + c_{k+1},
      $
      they remain optimal we have the equality 
      \[
        -\frac{a_{k+1,i} - a_{k,i}}{\tau} 
        = 
        g'(a_{k+1,i}) + \psi_{k+1,i} \text{ for all } i \in I_{k+1},
      \]  
      which concludes the proof.
  \end{proof}

In the previous proposition, one would like to obtain more information on the minimizers to remove the dependence on the normal cones of $\Omega$ and $\Delta_{N-1}$. The next result gives a natural criterion for a point $x_{k+1,i}$ to be in the interior of $\Omega$, so that the normal cone $N_{\Omega}(x_{k+1,i})$ is null. 

\begin{figure}[t]\label{figure.discrete_time_interior_argument}
  \centering
  \input{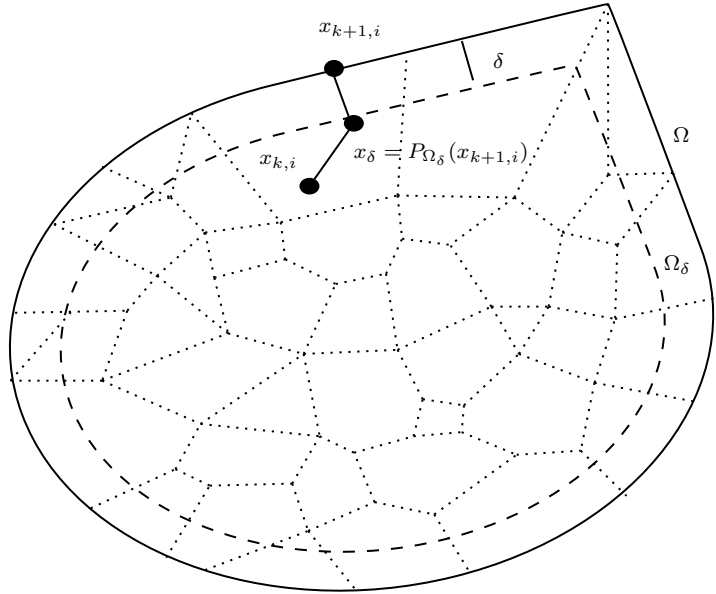}
  \caption{If $x_{k,i}$ belongs in the interior of $\Omega$ and $x_{k+1,i}$ is in the boundary $\partial \Omega$, we can produce a better competitor by projecting $x_{k+1,i}$ to the convex set $\Omega_\delta$ of points whose distance to the boundary is at least $\delta$.}
\end{figure}
\begin{proposition}\label{prop.atoms_interior}
    Let $\Omega \subset \mathbb{R}^d$ be compact and convex, and let $\left(\rhotau{k+1}, \Xtau{k+1}, \Atau{k+1} \right)$ be a minimizer of~\eqref{eq.MMS_def}. Then for all $i=1,\dots,N$ if $a_{k+1,i}>0$ and $x_{k,i} \in \intt \Omega$, it also holds that $x_{k+1,i} \in \intt \Omega$. 
\end{proposition}
  Once again we omit the dependence on $\tau$ in this proof since it is fixed. 
\begin{proof}
  First notice that from the à priori estimates it holds that $\mathscr{F}(\varrho_{k+1}) < +\infty$ and hence, since $F$ is superlinear, $\varrho_{k+1}$ is uniformly integrable, so that for any $\varepsilon > 0$ there is $\delta' > 0$ such that if 
  \[
    E \text{ is a Borelian with } |E| < \delta', \text{ then } \varrho_{k+1}(E) < \varepsilon. 
  \]

  Assume by contradiction that $x_{k+1,i} \in \partial \Omega$. In addition, since $x_{k,i} \in \intt \Omega$, it holds that $\dist(x_{k,i}, \partial \Omega) > 0$. Defining the family of closed and convex sets
  \[
    \Omega_\delta \eqdef 
    \left\{
      x \in \Omega : 
      \dist(x,\partial \Omega) \ge \delta      
    \right\},
  \]
  it holds that for $\delta$ small enough $x_{k,i} \in \Omega_\delta$, and as the Minkowiski content of $\partial \Omega$ coincides with its perimeter, we have that $|\Omega\setminus \Omega_\delta| \cvstrong{\delta \to 0}{} 0$. So that we can choose $\delta$ small enough so that $\varrho_{k+1}(\Omega\setminus \Omega_\delta) < \varepsilon$. 

  For such a choice of $\delta$, we set $x_\delta \eqdef \proj_{\Omega_\delta}(x_{k+1,i})$, the unique orthogonal projection onto the convex set $\Omega_\delta$, and define a new competitor $\mathbf{x}_\delta$ by replacing the entry $x_{k+1,i}$ by $x_\delta$, see Figure~\ref{figure.discrete_time_interior_argument}. As a result, since $x_{k,i} \in \Omega_\delta$, it follows from convexity that
  \[
    \norm{x_\delta - x_{k,i}}^2 \le \norm{x_{k+1,i} - x_{k,i}}^2.
  \]
  Hence, comparing the energies of this competitor and the minimizer, it follows that 
  \begin{align*}
    & 0\!\le\!
    W_2^2(\varrho_{k+1}, \mu_{\mathbf{x}_{\delta}, \mathbf{a}_{k+1}}) 
    -
    W_2^2(\varrho_{k+1}, \mu_{\mathbf{x}_{k+1}, \mathbf{a}_{k+1}}) \\ 
    &\!\le\! 
      \int_{\Omega_{k+1,i}\cap \Omega_\delta} \!\!\!\!
      \left(
      |x - x_\delta|^2 - |x - x_{k+1,i}|^2
    \right)
    \dd \varrho_{k+1}\\
    &+
      \int_{\Omega_{k+1,i}\setminus \Omega_\delta} \!\!\!\!
      \left(
      |x - x_\delta|^2 - |x - x_{k+1,i}|^2
    \right)
    \dd \varrho_{k+1}. 
  \end{align*}

  To estimate the first term, notice that for any $x \in \Omega_{k+1,i}\cap \Omega_\delta$ we have that 
  \begin{align*}
    |x - x_\delta|^2 - |x - x_{k+1,i}|^2 
    =& 
    |x_\delta|^2 - |x_{k+1,i}|^2 - 2\inner{x, x_\delta - x_{k+1,i}} \pm 2\norm{x_\delta - x_{k+1,i}}^2\\ 
    =& 
    |x_\delta|^2 - |x_{k+1,i}|^2 - 2\inner{x_\delta, x_\delta - x_{k+1,i}} \\
    &-2|x_\delta - x_{k+1,i}|^2 - 2\inner{x-x_\delta, x_\delta - x_{k+1,i}} 
    \\ 
    =& -|x_\delta - x_{k+1,i}|^2 - 2\inner{x-x_\delta, x_\delta - x_{k+1,i}} 
    \le -\delta^2,
  \end{align*}
  where in the last inequality we have used that $x_{k+1,i} \in \partial \Omega$ and classical properties of the projection onto closed convex sets.

  On the other hand, to estimate the second term, for any $x \in \Omega_{k+1,i}\setminus \Omega_\delta$ it holds that $|x - x_\delta|^2 - |x - x_{k+1,i}|^2 \le \delta^2$. Indeed, for any such $x$, we have that 
  \begin{align*}
    |x - x_\delta|^2 
    =& 
    |x - x_{k+1,i}|^2
    +
    |x_{k+1,i}-x_\delta|^2 
    +
    2\inner{x - x_{k+1,i}, x_{k+1,i}-x_\delta}\\ 
    \le& 
    |x - x_{k+1,i}|^2
    +
    |x_{k+1,i}-x_\delta|^2 \\ 
    \le& 
    |x - x_{k+1,i}|^2 + \delta^2,
  \end{align*}
  once again the last inequality coming from the properties of projection onto convex sets. 

  Coming back to the estimation of the energies, we obtain that
  \begin{align*}
    0 
    &\le
    W_2^2(\varrho_{k+1}, \mu_{\mathbf{x}_{\delta}, \mathbf{a}_{k+1}}) 
    -
    W_2^2(\varrho_{k+1}, \mu_{\mathbf{x}_{k+1}, \mathbf{a}_{k+1}}) \\ 
    &\le  
    \delta^2 \left(
        -\varrho_{k+1}\left(
          \Omega_{k+1,i}\cap \Omega_\delta
        \right)
        + 
        \varrho_{k+1}\left(
          \Omega_{k+1,i}\setminus \Omega_\delta
        \right)
    \right) \\ 
    &\le 
    \delta^2 \left(
        - a_{k+1,i} + 2\varrho_{k+1}\left(
          \Omega_{k+1,i}\setminus \Omega_\delta
        \right),
    \right).
  \end{align*}
  Choosing $\delta$ small enough so that $2\varrho_{k+1}\left(
    \Omega_{k+1,i}\setminus \Omega_\delta
  \right) \le 2\varepsilon < a_{k+1,i}$, the previous construction contradicts the minimality, so that it must hold that $x_{k+1,i} \in \intt \Omega$. 
\end{proof}

Concerning the behavior of the weights $a_i^\tau$, we use their optimality conditions to show that if $a_i$ is a limit curve as $\tau \to 0$, then it has the property that if at some time $a_i$ reaches $0$, then it remains null for any subsequent time instant. 
\begin{proposition}\label{prop.estimate_ai}
  Let $\Omega \subset \mathbb{R}^d$ be compact and convex, and let $\left(\rhotau{k+1}, \Xtau{k+1}, \Atau{k+1} \right)$ be a minimizer of~\eqref{eq.MMS_def}. For all $i = 1,\dots,N$ let $t \mapsto a_i(t)$ be a limit trajectory of $a_i^\tau(\cdot)$, then if $a_i(t) = 0$ we have $a_i(s) = 0$ for all $s \ge t$. 
\end{proposition}
\begin{proof}
  Let us first recall the inactive set $I_{k+1} = \{ h : 0 < a_{k+1,h} < 1 \}$ defined in Proposition~\ref{prop.optimality_conditions}. Given a limit curve ${a_{i}(\cdot)}_{i = 1, \dots N}$, define 
  \[
      t_1 \eqdef \min\left\{
        t \ge 0 : a_i(t) = 0 \text{ for } i \in \{1,\dots, N\}
      \right\}
      \text{ and }
      \Lambda_1 
      \eqdef 
      \left\{
        i : a_i(t_1) = 0
      \right\},
  \]
  and by induction we define for $n > 1$
  \[
      t_n \eqdef \min\left\{
        t \ge 0 : a_i(t) = 0 \text{ for } i \in \{1,\dots, N\}\setminus \Lambda_{n-1}
      \right\}
      \text{ and }
      \Lambda_n
      \eqdef 
      \left\{
        i : a_i(t_n) = 0
      \right\}.
  \]
  In particular, we expect that $\Lambda_n \subset \Lambda_{n+1}$. 

  Let us start by proving that $a_i(t) = 0$ for all times $t \in (t_1,t_2)$ and all indexes $i \in \Lambda_1$. Fix some time $t \in (t_1, t_2)$, hence by definition there is some $\delta$ such that for all $j \not\in \Lambda_1$ it holds that $a_j(\cdot) > \delta$ over $(t_1,t)$. In particular, let $k \in \mathbb{N}$ be such that $k\tau < t_1 \le (k+1)\tau$, so that for $\tau$ small enough if holds, from the uniform convergence of $a_i^\tau$ to $a_i$, that 
  \[
    a_{k,i}^\tau \le C \sqrt{\tau}, \text{ for all } i \in \Lambda_1 \text{ and }
    a_{n,j}^\tau \ge \delta/2, \text{ for all } j \not\in \Lambda_1 \text{, } t_1 \le n\tau \le t.
  \]
  
  Recall the definition of the set $I_{k+1} \eqdef \{i : 0 < \atau{k+1,i} < 1\}$, and notice that either $I_{k+1}$ is empty or it has at least two elements since if it where to have only one element, there would either be a single positive $a_{k+1,i}^\tau$ but smaller than $1$, or there would be an index with mass $1$ and another in $I_{k+1}$, either way contradicting the fact that the total mass sums to $1$. We will now prove that 
  \[
    \sum_{i \in \Lambda_1\cap I_{k+1}} \atau{k+1,i}
    \le 
    \sum_{i \in \Lambda_1\cap I_{k}} \atau{k,i} 
    \le 
    CN \sqrt{\tau}, 
  \]
  with an argument that can be carried by induction for all $n \ge k$ such that $n\tau < t$. The case that $I_{k+1}$ is empty is trivially true, hence we suppose that $|I_{k+1}| \ge 2$. 
  
  From the optimality conditions of $a_{k+1,i}^\tau$, we have that 
  \begin{align*}
    \sum_{i \in \Lambda_1 \cap I_{k+1}} a_{k+1,i}^\tau - a_{k,i}^\tau 
    = 
    \tau\left(
      \sum_{i \in \Lambda_1 \cap I_{k+1}} 
      -g'(a_{k+1,i}^\tau) - \psi_{k+1,i}^\tau
    \right).
  \end{align*}
  And now we can use the characterization of the Kantorovitch potential $\psi_{k+1}$ in terms of the potential $\bar \psi_{k+1}$, which is bounded by a constant depending on $\Omega$. Indeed, from the proof of Proposition~\ref{prop.optimality_conditions} we have that
  \begin{align*}
    \sum_{i \in \Lambda_1 \cap I_{k+1}} a_{k+1,i}^\tau &- a_{k,i}^\tau
    =
    \tau\left(
      \sum_{i \in \Lambda_1 \cap I_{k+1}} 
      -g'(a_{k+1,i}^\tau) - \bar \psi_{k+1,i}^\tau \right. \\ 
      &+
      \left. 
      \frac{1}{|I_{k+1}|}
      \left(
        \sum_{j \in I_{k+1}} \bar \psi_{k+1,j} + g'(a_{k+1,j}^\tau) + 
        \sum_{j \in I_k\setminus I_{k+1}} \frac{a_{k,j}^\tau}{\tau}
      \right)
    \right)\\ 
    \le& 
    \tau\left(
      C_\Omega 
      -
      \left(1 - \frac{1}{|I_{k+1}|}\right)
      \sum_{i \in \Lambda_1 \cap I_{k+1}} 
      g'(a_{k+1,i}^\tau)
      +
      \frac{1}{|I_{k+1}|}
        \sum_{j \in I_{k+1}\setminus \Lambda_1}  g'(a_{k+1,j}^\tau)
      \right) \\ 
      &+
      \frac{1}{|I_{k+1}|}
      \sum_{j \in I_k\setminus I_{k+1}} a_{k,j}^\tau.
  \end{align*}
  But notice that, for all indexes $j \in \Lambda_1 \setminus I_{k+1}$, the mass $a_{k+1,j}^\tau$ is bounded from below away from zero, hence $g'(a^\tau_{k+1,j})$ is bounded from above by a finite positive constant. On the other hand, as $\tau \to 0$ we have that $g'(a^\tau_{k+1,i}) \to \infty$ for $i \in \Lambda_1\cap I_{k+1}$. As a result the term being multiplied by $\tau$ in the above inequality is negative for $\tau$ small enough. It then follows that
  \begin{align*}
    \sum_{i \in \Lambda_1 \cap I_{k+1}} a_{k+1,i}^\tau 
    &\le 
    \sum_{i \in \Lambda_1 \cap I_{k+1}} a_{k,i}^\tau
    + 
    \sum_{j \in I_k\setminus I_{k+1}} a_{k,j}^\tau\\
    &= 
    \sum_{i \in \Lambda_1 \cap I_{k+1}} a_{k,i}^\tau
    + 
    \sum_{j \in I_k\setminus I_{k+1}} a_{k,j}^\tau 
    = 
    \sum_{i \in \Lambda_1 \cap I_{k}} a_{k,i}^\tau,
  \end{align*}
  where the last equality comes from the fact that $I_k\setminus I_{k+1} \subset \Lambda_1$. Indeed, in our current case $|I_{k+1}|\ge 2$, if $i \in I_k \setminus I_{k+1}$ it must hold that $a_{k+1,i}^\tau = 0$ and $0 < a_{k,i}^\tau \le C\sqrt{\tau}$, which cannot be the case if $i \in \Lambda_1^c$, since in this complement $a_i^\tau(t)$ is uniformly bounded away from $0$ over $(t_1,t)$. 

  We can iterate this argument, obtaining that 
  \[
    \sum_{i \in \Lambda_1 \cap I_{n+1}} a_{n+1,i}^\tau  < \sum_{i \in \Lambda_1 \cap I_{n}} a_{n,i}^\tau, 
    \text{ while } t_1 < (n+1)\tau < t_2.  
  \]
  This implies that for all $t_1 < t < t_2$ we have
  \[
    \sum_{i \in \Lambda_1 \cap I_{n+1}} a_{i}^\tau(t)  < \sum_{i \in \Lambda_1 \cap I_{n}} a_{i}^\tau(t_1) 
    \le |\Lambda_1|C \sqrt{\tau}. 
  \]
  From uniform convergence we conclude that $a_i(\cdot) = 0$ for any $i \in \Lambda_1$ over $(t_1, t)$.

  Since the curves $t \mapsto a_i(t)$ are continuous, it must hold that $\Lambda_1 \subset \Lambda_2$, and we can repeat the same argument for bigger set of indexes $\Lambda_2$ on the time interval $(t_2, t_3)$. It then follows by the same argument and an induction principle that $a_i(t) = 0$ for any $i \in \Lambda_n$ and $t \ge t_n$. The result follows. 

\end{proof}

\subsection{Strong convergence of the pressure}\label{sec.conv_pressure}
Recalling that the pressure variable is defined as $P(\varrho) = \varrho F'(\varrho) - F(\varrho)$, and $P(\cdot)$ is a continuous function, we notice that its gradient is given by
\[
  \nabla P(\varrho) = \varrho \nabla F'(\varrho).
\]
As a result, the optimality conditions from Prop.~\ref{prop.optimality_conditions} yield a uniform estimate for ${(P(\varrho^\tau))}_{\tau > 0}$. 

\begin{lemma}
    There exists a constant $C$, depending only on $\Omega$ and $T$ such that
    \begin{align*}
        \norm{P(\varrho^\tau)}_{L^1\left([0,T]; W^{1,1}(\Omega)\right)} &\le C, \\ 
        \norm{\nabla P(\varrho^\tau)}_{L^1\left([0,T]\times\Omega\right)} &\le C, 
    \end{align*}
    for all $\tau > 0$. 
\end{lemma}
\begin{proof}
    The bound on $\norm{P(\varrho^\tau)}_{L^1\left([0,T]; L^{1}(\Omega)\right)}$ is easily obtainable with the a priori energy estimates~\eqref{eq.a_priori_energy}. From the optimality conditions on $\rhotau{k+1}$, we obtain for the stair case interpolation that
    \[
        |\nabla P(\rhotau{t})| 
        \le 
        (|\mathbf{v}^\tau_t| + \diam \Omega)\rhotau{t}, \text{ a.e.~in } [0,T]\times \Omega,
    \]
    so that
    \begin{align*}
        \int_0^T \norm{\nabla P(\rhotau{t})}_{L^1(\Omega)} \dd t 
        &\le  
        \sum_{k = 0}^{N_\tau - 1}
        \int_{k\tau}^{(k+1)\tau} 
        \left(
            \int_\Omega |\nabla P(\rhotau{k+1})|\dd x 
        \right)\dd t\\ 
        &\le  
        \sum_{k = 0}^{N_\tau - 1}
        \tau
        \left(
            \int_\Omega (|\mathbf{v}^\tau_{(k+1)}| + \diam \Omega)\rhotau{k+1}\dd x 
        \right)\dd t\\ 
        &\le 
        \tau N_\tau 
        + 
        \frac{1}{\tau}
        \sum_{k = 0}^{N_\tau - 1}
        \int_\Omega |\id - T_{k+1}^\tau|^2\dd \rhotau{k+1}\\ 
        &\le 
        \tau N_\tau 
        + 
        \frac{1}{\tau}
        \sum_{k = 0}^{N_\tau - 1}
        W_2^2(\rhotau{k}, \rhotau{k+1})\\ 
        &\le 
        2T + \mathscr{E}(\varrho_0, \mu_0), 
    \end{align*}
    and the result follows.
\end{proof}

In the sequel, we are ready to strengthen the convergence of $\rho^\tau$ to strong convergence in $L^1$ and as a result, prove that $P(\varrho^\tau)$ and $\nabla P(\varrho^\tau)$ also pass to the limit. This will be done, as in~\cite{carlier2017splitting} and~\cite{di2014curves}, resorting to the following result by Savaré and Rossi~\cite[Thm.~2]{rossi2003tightness}. 
\begin{theorem}\label{thm.savare_rossi}
    On a Banach space $\mathcal{X}$, let be given
    \begin{itemize}
        \item a normal coercive integrand, $\Phi : \mathcal{X} \to [0, +\infty]$, i.e.~lower semi-continuous with relatively compact sublevels in $\mathcal{X}$;
        \item a pseudo-distance $d : \mathcal{X} \times \mathcal{X} \to [0, +\infty]$, i.e., $d$ is lower semi-continuous, and if $d(\rho, \rho') = 0$ with $\Phi(\rho), \Phi(\rho') < \infty$, then $\rho = \rho'$.
    \end{itemize}
    Let $\mathcal{U}$ be a set of measurable functions $u : (0, T) \to  \mathcal{X}$. Assuming that
    \[
        \sup_{u \in \mathcal{U}} 
        \int_0^T \Phi(u(t)) \dd t < \infty 
        \text{ and }
        \lim_{h \to 0}
        \sup_{u \in \mathcal{U}} 
        \int_{0}^{T-h}
        d(u(t + h), u(t))\dd t = 0,
    \]
    then $\mathcal{U}$ contains a subsequence sequence that converges in measure (with respect to $t \in (0,T)$) to a limit $u_\star$.
\end{theorem}

\begin{proposition}\label{prop.convPrho}
    Up to subsequences, the staircase approximation ${\left(\varrho^\tau\right)}_{\tau > 0}$ satisfies 
    \begin{align*}
        \varrho^\tau &\cvstrong{\tau \to 0}{} \varrho \text{ in $L^1([0,T]\times \Omega)$ and pointwise},\\
        P(\varrho^\tau) &\cvstrong{\tau \to 0}{} P(\varrho) \text{ in $L^1([0,T]\times \Omega)$ and pointwise},\\
        \nabla P(\varrho^\tau) &\cvweak{\tau \to 0}{\star} 
        \nabla P(\varrho) \text{ in $\mathscr{M}^d([0,T]\times \Omega)$.}
    \end{align*}
\end{proposition}

\begin{proof}
  Throughout this proof, many different convergence results are proven up to subsequences, so at each step we assume without saying explicitly that we have selected a subsequence for which all the previous convergences hold. 

  First, we claim that the conditions of Thm.~\ref{thm.savare_rossi} hold for the family ${(\varrho^\tau, \mathbf{x}^\tau, \mathbf{a}^\tau)}_{\tau > 0}$ with the choice 
  \[
    d(\xi, \xi') \eqdef 
    \begin{cases}
      \dwl(\xi, \xi'),& \text{ if } \xi, \xi' \in \Pac(\Omega) \times \mathbb{R}^{dN} \times \Delta_{N-1}\\ 
      +\infty, &\text{ otherwise}
    \end{cases}
  \]
  and 
  \[
    \Phi(\xi) 
    \eqdef 
    \begin{cases}
      \mathscr{F}(\varrho) + \norm{P(\varrho)}_{\BV(\Omega)} + \norm{\mathbf{x}}^2 + \norm{\mathbf{a}}^2
      ,& \text{ if } \xi \in \Pac(\Omega) \times \mathbb{R}^{dN} \times \Delta_{N-1}\\ 
      +\infty, &\text{ otherwise},
    \end{cases}
  \]
  as shown in~\cite{carlier2017splitting,di2014curves}. As a result, there is a subsequence such that $\varrho^\tau \cvstrong{\tau \to 0}{} \varrho$ in measure in $[0,T]$, therefore up to a further subsequence, convergence also holds in $L^1\left([0,T]; L^1(\Omega)\right)$. From Lebesgue's dominated convergence theorem, it also holds that
  $\varrho^\tau$ converges to $\varrho$ in $L^1([0,T]\times \Omega)$ and pointwise for a.e. $(t,x) \in [0,T]\times \Omega$. 

  Moving on to the convergence of $P(\varrho^\tau)$, first notice that since $\rho \mapsto P(\rho)$ is continuous, we have a.e.~convergence of $P(\varrho^\tau)$ to $P(\varrho)$. In addition, notice that Lemma~\ref{lemma.a_priori_estimates} gives a uniform bound 
    \[
      \norm{P(\varrho^\tau)}_{L^1([0,T];\BV(\Omega))} \le C,
    \]
    and thanks a Sobolev embedding, a uniform bound on $\norm{P(\varrho^\tau)}_{L^1\left([0,T];L^{d/(d-1)}\right)}$. A suitable interpolation inequality, proved for instance in~\cite[Lemma~5.3]{carlier2017splitting}, gives that
    \[
      \norm{P(\varrho^\tau)}_{L^{(d+1)/d}([0,T]\times \Omega)} \le C,
    \]
    so that by De la Vallée Poussin's Theorem~\cite[Thm.~2.29]{fonseca2007modern} ${\left(P(\varrho^\tau)\right)}_{\tau > 0}$ is uniformly integrable in $[0,T]\times \Omega$, since
    \[
      \sup_{\tau > 0} \int_{[0,T]\times \Omega} 
      \gamma\left(|P(\varrho^\tau(t,x))|\right)\dd t \dd x < +\infty
    \]
    with the superlinear function $\rho \mapsto \gamma(\rho) = |\rho|^{\frac{d+1}{d}}$. If follows from Vitali's convergence theorem~\cite[Thm.~2.2.4]{fonseca2007modern} that $P(\varrho^\tau)$ converges to $P(\varrho)$ in $L^1([0,T]\times \Omega)$. 

    In addition, the uniform bound in $L^1([0,T]\times \BV(\Omega))$ along with the strong convergence in $L^1$, allows us to use weak compactness in $\BV$~\cite[Prop.~3.13]{ambrosioFunctionsBoundedVariation2000} to conclude that
    \[
      \nabla P(\varrho^\tau) \cvstar{\tau \to 0} \nabla P(\varrho) \text{ in } \mathscr{M}^d([0,T]\times \Omega),
    \]
    the result follows. 
\end{proof}

\subsection{The limit PDE-ODE system}
Now we can capitalize on the previous results to conclude with the convergence of the proposed minimizing movement scheme and consequently to the existence of a weak solution to the coupled system~\eqref{eq.PDE_ODEsysthem}.

\begin{theorem}\label{thm.convergenceJKO}
    Assume that $\mathscr{E}(\varrho_0, \mu_0) < +\infty$ and either
    \begin{itemize}
      \item the initial conditions $\mathbf{x}_0 \in {(\intt \Omega)}^{\otimes N}$;
      \item $\Omega$ has a $\mathscr{C}^1$ boundary.
    \end{itemize}
    Under these assumptions, the minimizing movement scheme defined in~\eqref{eq.MMS_def} admits subsequences converging to weak solutions of~\eqref{eq.PDE_ODEsysthem}.  
\end{theorem}
\begin{proof}
    Recall that from Proposition~\ref{prop.interpo_apriori_limits}, the pair $(\varrho, E)$ solves the continuity equation. Hence, to obtain the desired limit it suffices to characterize the momentum measure $E$. 

    To this end, recall that $E^\tau \cvweak{\tau \to 0}{} E$ in $\mathscr{M}^d([0,T]\times \Omega)$, where $E^\tau$ is the family of staircase momenta, characterized from the Euler-Lagrange equations (Prop.~\ref{prop.optimality_conditions}) as
    \[
      E^\tau = \mathbf{v}^\tau\varrho^\tau = 
      - \nabla P(\varrho^\tau) 
      -\varrho^\tau \sum_{i= 1}^N (x-\xtau{t,i})\mathbbm{1}_{\Omega^\tau_{t,i}}. 
    \]
    In the sequel, consider $\varphi \in \mathscr{C}^\infty_c((0,T)\times \Omega; \mathbb{R}^d)$, we have that 
    \begin{align*}
      \int_{[0,T]\times \Omega} \varphi\cdot \dd E 
      &= 
      \lim_{\tau \to 0}
      \int_{[0,T]\times \Omega} \varphi\cdot \dd E^\tau \\ 
      &= 
      \lim_{\tau \to 0}
      \left(
        \int_{[0,T]\times \Omega} 
        \divv \varphi P(\varrho^\tau)\dd x 
        - 
        \sum_{i= 1}^N 
        \int_{[0,T]\times \Omega} 
        \varphi \cdot (x-\xtau{t,i})\mathbbm{1}_{\Omega^\tau_{t,i}}\dd \varrho^\tau
      \right)
      . 
    \end{align*}
    From the strong convergence of $P(\varrho^\tau)$ to $P(\varrho)$ in $L^1([0,T]\times \Omega)$, the first term has the desired limit. As for the second, recall that $\xtau{\cdot,i} \cvstrong{\tau \to 0}{} x_{\cdot,i}$ uniformly in $[0,T]$. As for the Kantorovitch potentials, we recall as in the proof of Proposition~\ref{prop.optimality_conditions} that the potentials are unique up to constants and that the potentials used in the optimality conditions for $\mathbf{a}_{k+1}^\tau$ are for the form
    \[
      \psi_{k+1,i}^\tau = \bar\psi_{k+1,i}^\tau + c_{k+1}^\tau,
    \]
    where the first term corresponds to another potential bounded by a constant depending on the domain, $|\bar\psi_{k+1,i}^\tau| \le C_\Omega$, and $c_{k+1}^\tau$ is a constant that might explode as $\tau \to 0$. We can therefore define the curves 
    \[
      \psi^\tau_{t} \eqdef \psi^\tau_{k+1} \text{ and } \bar\psi^\tau_{t} \eqdef \bar\psi^\tau_{k+1}, \text{ if } t \in (k\tau,(k+1)\tau]. 
    \]
    So that, $\bar\psi^\tau_t$ is uniformly bounded and converges point-wise in $[0,T]$ to $\bar \psi_t$, a bounded Kantorovitch potential associated with $W_2^2(\varrho_t, \mu_t)$. But recall from the Lemma~\ref{lemma.cont_laguerre} that
    \[
      (\psi, \mathbf{x}) \mapsto \Lag_i(\psi, \mathbf{x})
    \]
    is continuous from $\mathbb{R}^d\times \mathbb{R}^d$ to $L^1$, and since Laguerre cells are invariant by the addition of constants, it follows that 
    \[
      \Omega^\tau_{t,i} = 
      \Lag_i(\psi^\tau_t, \mathbf{x}^\tau_t) \cvstrong{\tau \to 0}{L^1} 
      \Lag_i(\psi_t, \mathbf{x}_t) 
      =
      \Omega_{t,i} \text{ point-wise in }[0,T]. 
    \]

    From Egorov's Theorem, for each $\varepsilon>0$, there a set $I_\varepsilon \subset [0,T]$, with $|[0,T]\setminus I_\varepsilon| \le \varepsilon$, where this point-wise convergence can be strengthened to uniform converge. It then follows that for every $\varepsilon> 0$, we have 
    \begin{align*}
      &
      \limsup_{\tau \to 0}
      \left|
        \int_{[0,T]\times \Omega} 
        \varphi \cdot (x-\xtau{t,i})\mathbbm{1}_{\Omega^\tau_{t,i}}\dd \varrho^\tau
        -
        \int_{[0,T]\times \Omega} 
        \varphi \cdot (x-x_{t,i})\mathbbm{1}_{\Omega_{t,i}}\dd \varrho
      \right|\\ 
      &\quad \le 
      \limsup_{\tau \to 0}
      \left|
        \int_{I_\varepsilon \times \Omega} 
        \varphi \cdot (x-\xtau{t,i})\mathbbm{1}_{\Omega^\tau_{t,i}}\dd \varrho^\tau
        -
        \int_{I_\varepsilon \times \Omega} 
        \varphi \cdot (x-x_{t,i})\mathbbm{1}_{\Omega_{t,i}}\dd \varrho
      \right|
      + C_{\varphi, \Omega, \varrho} \varepsilon.
    \end{align*}
    From the strong convergence of $\varrho^\tau$ to $\varrho$ in $L^1([0,T]\times\Omega)$, Prop.~\ref{prop.convPrho}, the uniform convergence of $x^\tau_{\cdot,i}$ and of $\mathbbm{1}_{\Omega^\tau_{t,i}}$, the second $\limsup$ above is null. As a consequence, we conclude that
    \[
      E = -\nabla P(\varrho) - \varrho \sum_{i= 1}^N (x-x_{t,i})\mathbbm{1}_{\Omega_{t,i}}.
    \]

    Moving on to the dynamics of the atoms, recall that from Proposition~\ref{prop.interpo_apriori_limits} it holds that $\mathbf{\bar x}^\tau$ converges to $\mathbf{x}$ weakly in $H^1([0,T])$ both $\mathbf{\bar x}^\tau, \mathbf{x}^\tau$ converge uniformly (and hence also strongly in $L^2([0,T])$) to $\mathbf{x}$. To characterize the limit dynamics of the $\mathbf{x}$ variables, consider $\varphi \in \mathscr{C}^\infty_c((0,T))$ and we treat two cases: either
    \begin{itemize}
      \item the initial conditions $\mathbf{x}_0 \in {(\intt \Omega)}^{\otimes N}$;
      \item $\Omega$ has a $\mathscr{C}^1$ boundary.
    \end{itemize}

    For simplicity, let us discuss the first one. In this case, Prop.~\ref{prop.atoms_interior} ensures that the entire sequence ${\left(\mathbf{x}_k\right)}_{k}$ is contained in the interior of $\Omega$ and therefore the boundary effects of the normal cone does not intervene in the Euler-Lagrange equations from Prop.~\ref{prop.optimality_conditions}, so that we have
    \begin{align*}
      -
      \int_0^T\varphi'(t) x_{t,i} \dd t 
      &=
      \lim_{\tau \to 0} 
      -\int_0^T\varphi'(t) \bar x_{t,i}^\tau \dd t 
      =
      \lim_{\tau \to 0}
      \int_0^T\varphi(t) \dot{\bar x}^\tau_{t,i} \dd t \\
      &= 
      \lim_{\tau \to 0} - 
      \int_0^T\varphi(t) 
      \left(
        \atau{k+1,i}\cdot\xtau{k+1,i} 
        - \int_{\Omega^\tau_{k+1,i}}x \dd \rhotau{k+1}
      \right)\dd t\\
      &=
      -
      \int_0^T\varphi(t) 
      \left(
        a_{t,i}\cdot x_{t,i} 
        - \int_{\Omega_{t,i}}x \dd \varrho_t
      \right)\dd t,
    \end{align*}
    where the last limit hold, as before, because of the uniform convergence in $t$ of the Laguerre cells and the convergence of $\varrho^\tau$ to $\varrho_t$.   This gives the desired characterization of the flow of $\mathbf{x}_t$. The calculations are analogous in the second case, with the addition of the normal vector $n_\Omega(x_{k+1,i}^\tau)$, which converges to a vector in $N_\Omega(x_{t,i})$ since the boundary is smooth.  

    The dynamics of the weights can be deduced in a similar way to the convergence of the atoms. Also from Proposition~\ref{prop.interpo_apriori_limits}, we know that $\mathbf{\bar a}^\tau$ converges to $\mathbf{a}$ weakly in $H^1([0,T])$ and both $\mathbf{\bar a}^\tau, \mathbf{a}^\tau$ converge uniformly to $\mathbf{a}$. On the other hand, we must take into account the singularity at zero presented by the optimality conditions of the $a_i$ variables given by the term $g'(a_i)$, which is compensated with the Kantorovitch potentials. The issue as we have seen is that the Kantorovitch potentials consists of a bounded potential added by a constant that might diverge to compensate for the singularity in any single $g'(a_i(t))$ that goes to infinity. 
    
    As in the proof of Proposition~\ref{prop.estimate_ai}, define the quantities
    \[
      t_i \eqdef \inf \left\{
        t \ge 0 : a_i(t) = 0
      \right\} \text{ for } i = 1,\dots, N,
    \]
    and for a fixed $i$, take some other index $j$ such that $t_j \ge t_i$. If such $j$ does not exist it means that $a_i \equiv 1$ and $t_1 = +\infty$. Hence, to characterize the dynamics of $a_i$ it suffices to consider test functions in $\varphi \in \mathscr{C}^\infty_c((0,t_i))$, since after this time we know that $a_i$ vanishes. As a result, for such test function we obtain that $\displaystyle \inf_{\supp \varphi} a_i, \inf_{\supp \varphi} a_j > 0$ and from the uniform convergence the same holds for $a_i^\tau$ and $a_j^\tau$. As a result we obtain that
    \begin{align*}
      \int_0^T (a_{i}(t) - a_{j}(t)) \dot{\varphi}(t) \dd t 
      &= 
      \lim_{\tau \to 0}
      \int_0^T
      (a_{i}^\tau(t) - a_{j}^\tau(t)) \dot{\varphi}(t) \dd t \\
      &= 
      \lim_{\tau \to 0} 
      - 
      \int_0^T
      (\dot{a}_{i}^\tau(t) - \dot{a}_{j}^\tau(t)) \varphi(t) \dd t \\ 
      &=
      \lim_{\tau \to 0} 
      \int_0^T\varphi(t) 
      \left(
        g'(a_{i}^\tau(t)) - \psi^\tau_{t,i} - g'(a_{j}^\tau(t)) + \psi^\tau_{t,j}
      \right)\dd t\\
      &=
      \lim_{\tau \to 0} 
      \int_0^T\varphi(t) 
      \left(
        g'(a_{i}^\tau(t)) - \bar\psi^\tau_{t,i} - g'(a_{j}^\tau(t)) + \bar\psi^\tau_{t,j}
      \right)\dd t\\ 
      &= 
      \int_0^T\varphi(t) 
      \left(
        g'(a_{i}(t)) - \bar\psi_{t,i} - g'(a_{j}(t)) + \bar\psi_{t,j}
      \right)\dd t,
    \end{align*}
    since $g \in \mathscr{C}_{loc}^1((0,1])$, where $a_i^\tau,a_j^\tau$ remains bounded from bellow in the support of $\varphi$, and the Kantorovitch potentials $\bar \psi_t$ being uniformly bounded and Lipschitz, converge to the Kantorovitch potential $\bar \psi_t$.

    We then conclude that for any pair $i,j$ such that $t_j \ge t_i$ their corresponding dynamics satisfy 
    \[
      \dot{a}_i(t) + g'(a_i(t)) - \bar \psi_{t,i} 
      = 
      \dot{a}_j(t) + g'(a_j(t)) - \bar \psi_{t,j}
      = 
      c_t, 
    \] 
    where $c_t \in L^2((0,t))$ for $t < t_i$. We can now do as in Proposition~\ref{prop.optimality_conditions} and define a new Kantorovitch potential $\psi_t = \bar \psi_t + c_t$, in such a way that 
    \[
      \dot{a}_i(t) = - g'(a_i(t)) + \psi_{t,i} \text{ in } L^2((0,t_i)).
    \]
    In addition, since $a_h(t) = 0$ for $t \ge t_h$, we obtain that 
    \[
      0 = \frac{\dd}{\dd t}\sum_{i \sim a_i(t) > 0}a_i(t) 
      =
      \sum_{i \sim a_i(t) > 0}
      -g'(a_i(t)) + \psi_{t,i},
    \]
    which finishes the desired characterization of the dynamics and the result follows. 
\end{proof}

\subsection{Strong \texorpdfstring{$L^2H^1$}{L2H1} convergence in Porous Medium case}\label{sec.conv_porous_medium}

In the previous section, as in seminal papers that introduced the JKO scheme~\cite{jordan1998variational,otto2001geometry}, the convergence of the minimizing movement scheme was proven in the $L^1$ topology. This is natural when working with probability measures, but certainly the limit PDE (\textit{e.g.} heat, Fokker-Plack, Porous Medium equation) enjoys much more regularity, which suggests that convergence in stronger topologies also holds. In recent papers~\cite{dimarino2022jko,santambrogio2024strong} the authors show that the JKO scheme of the Boltzman entropy functional
\begin{equation}\label{eq.entropy_functional}
    \mathscr{F}(\varrho) 
    = 
    \mathscr{H}(\varrho) 
    \eqdef 
    \begin{cases}
      \displaystyle
      \int_\Omega \varrho\log\varrho \dd x,& \text{ if } \varrho \ll \mathscr{L}^d\mres \Omega,\\ 
      +\infty,& \text{otherwise,}
    \end{cases}
\end{equation}
(resp.~with a potential energy term), converges to the unique solution of the heat (resp.~Fokker-Plank) equation in the strong $L^2 H^2$ topology. 

In the case of the PDE from~\eqref{eq.PDE_ODEsysthem}, the advection term coming from the semi-discrete transport is not sufficiently smooth to adapt the arguments from~\cite{dimarino2022jko,santambrogio2024strong}, so that we can expect $L^2 H^1$ convergence instead. In fact, we shall prove convergence of the staircase interpolation of the pressure variable $P(\varrho^\tau)$ to $P(\varrho)$ in the strong topology of $L^2 H^1$ when the internal energy is given by a Porous Medium term of the form
\begin{equation}\label{eq.PME_functional}
  \mathscr{F}_m(\varrho) 
  \eqdef 
  \begin{cases}
    \displaystyle
    \frac{1}{m-1}\int_\Omega \varrho^m \dd x,& \text{ if } \varrho \ll \mathscr{L}^d\mres \Omega,\\ 
    +\infty,& \text{otherwise,}
  \end{cases}
\end{equation}
obtained for $F(\rho) = F_m(\rho) = \displaystyle \frac{1}{m-1}\rho^m$, and yielding a pressure $P_m(\rho) = \rho^m$.  

Following~~\cite{dimarino2022jko,santambrogio2024strong}, the strategy of the proof will consist on first proving $L^p$ estimates for iterates of the JKO scheme, by bounding their $L^p$ norms by the one of the previous estimate multiplied by a factor that remains summable as $\tau \to 0$. This allows to pass to the limit as $p \to \infty$ to obtain also $L^\infty$ bounds whenever the initial condition is bounded. Our argument consists of exploiting the displacement convexity of $\rho \to \norm{\rho}_p^p$ as in~\cite{dimarino2022jko} in order to deal with the Porous Medium term, which also gives the desired bounds for the gradient. The strong convergence then will follow with an adaptation of the arguments in~\cite{santambrogio2024strong}. The same convergence can be obtained for linear diffusion, but we have chosen to focus on the Porous Medium case for a more concise presentation. 

\begin{theorem}\label{thm.L2H1convergence}
  Assume that $\Omega$ be a convex and bounded domain of $\mathbb{R}^d$ with Lipschitz boundary. Let ${\left(\varrho_k^\tau, \mathbf{x}_k^\tau, \mathbf{a}_k^\tau\right)}_{k \in \mathbb{N}}$ be a sequence obtained via the JKO scheme, with the internal energy given by the Boltzmann entropy~\eqref{eq.entropy_functional} or the Porous Medium term $\mathscr{F}_m$ from~\eqref{eq.PME_functional}. Then the following estimates hold
  \begin{enumerate}
    \item for all $1< p < +\infty$, if $\varrho_0 \in L^p(\Omega)$ it holds that for all $k \ge 1$ that
    \begin{equation*}
      \norm{\varrho_{k}^\tau}_{L^p(\Omega)} 
      \le 
      {\left(
        1 - \tau d(p-1)
      \right)}^{-1/p} 
      \norm{\varrho_{k-1}^\tau}_{L^p(\Omega)}.
    \end{equation*}
    If $\varrho_0 \in L^\infty(\Omega)$, taking the limit as $p\to \infty$ we obtain 
    \begin{equation*}
      \norm{\varrho_k^\tau}_{L^\infty(\Omega)} 
      \le 
      e^{\tau d} \norm{\varrho_{k-1}^\tau}_{L^\infty(\Omega)}.
    \end{equation*}
    \item For $\varrho_0 \in L^{m+1}(\Omega)$, there exists a constant $C$ depending on $m,T,\Omega$ and $\varrho_0$ such that
    \[
      \int_0^T\int_\Omega|\nabla P_m(\varrho^\tau)|^2\dd x \dd t
      \le C, 
    \]
    for all $\tau$ small enough. 
  \end{enumerate}

  In addition, if $\varrho_0 \in L^{m+1}(\Omega)$, up to subsequences the pressures $P_m(\varrho^\tau)$ converge in the strong topology of $L^2 H^1$ to $P(\varrho)$, where $(\varrho,\mathbf{x},\mathbf{a})$ is a solution of~\eqref{eq.PDE_ODEsysthem}. 
\end{theorem}

These estimates can be proved with the same techniques for the classical JKO scheme for the Porous Medium, without the semi-discrete transport term and with a Lipschitz potential independent of the $\varrho$. In this case, such estimates imply that any solution of the PME obtained with the JKO scheme is an \textit{energy solution}, see~\cite[Section 5.3.2]{vazquez2007porous} and~\cite[Def.~7.1]{bertsch1986density}, defined as follows 
\begin{definition}\label{def.energy_solution}
   Let $m>1$, and $V \in W^{1,\infty}(\Omega)$ be a vector field, which is Lipschitz in space for a.e.~$t$ and for all $\rho \in \mathscr{P}(\Omega)$. Then $\varrho$ is an energy solution of the Porous Medium equation with advection
   \begin{equation}\label{eq.PME_m}
    \tag{$PME_m$}
      \partial_t \varrho = \Delta \varrho^m + \divv\left(\varrho \nabla V\right),
   \end{equation}
   if it is a solution in the sense of distributions such that
   \begin{enumerate}
    \item $\varrho \in \mathscr{C}([0,T];L^1(\Omega))\cap L^\infty([0,T]\times \Omega)$;\\
    \item $ P_m(\varrho) \in L^2([0,T];H^1(\Omega))$.
   \end{enumerate}
\end{definition} 
Although the PME might admit more than one weak solution, there is a unique energy solution, see~\cite[Section 5.3.2]{vazquez2007porous} and~\cite{bertsch1986density}. This discussion gives the following
\begin{corollary}\label{cor.L2H1_PME}
  Let $\Omega$ be a convex and bounded domain of $\mathbb{R}^d$ with Lipschitz boundary. Given a potential $V \in W^{1,\infty}(\Omega)$ and an initial condition $\varrho_0 \in L^{m+1}(\Omega) \cap \mathscr{P}(\Omega)$. Let ${\left(\varrho^\tau\right)}_{\tau > 0}$ be the staircase interpolation of JKO scheme referent to the Porous Medium energy with advection 
  \[
    \rho \mapsto \frac{1}{m-1}\int_\Omega \rho^m \dd x + \int_\Omega V \rho \dd x,
  \]
  then the pressures ${\left(P_m(\varrho^\tau)\right)}_{\tau > 0}$ converge in the strong topology of $L^2\left([0,T];H^1(\Omega)\right)$ to the unique energy solution of the PME with advection introduced by the potential $V$.   
\end{corollary} 

The $L^2H^2$ convergence for the Porous Medium case remains open, even with a smooth advection term. The main ingredient to prove such strong convergence in the Fokker-Planck case is the 5-gradients inequality introduced in~\cite{de2016bvestimates}. In~\cite[Lemma 5.1]{santambrogio2024strong}, the inequality is derived with a remainder term which can be controlled and then exploited to obtain the stronger convergence.

The major difficulty of the case with a semi-discrete transportation term is that it introduces to the dynamics a vector field with a singular Laplacian, which appear in the computations even for the discrete-time scheme. Hence, our first step in the following Lemma is to compute this Laplacian explicitly.  
\begin{lemma}\label{lemma.divergence_computation}
    Given $(\varrho, \mathbf{x}, \mathbf{a})$, define the potential 
    \begin{equation}\label{eq.Kant_potential}
        \Phi[\varrho, \mathbf{x}, \mathbf{a}](x)
        \eqdef 
        \sum_{i = 1}^N 
        \left(
          \frac{1}{2}|x - x_i|^2 - \psi_i
        \right)\mathds{1}_{\Omega_i},
    \end{equation}
    where ${(\Omega_i)}_{i = 1}^N$ corresponds to the optimal Laguerre tessellation associated to the transportation of $\varrho$ to $\mu_{\mathbf{x},\mathbf{a}}$ and ${(\psi_i)}_{i = 1}^N$ is the corresponding Kantorovitch potential. 

    The gradient vector field of $\Phi[\varrho, \mathbf{x}, \mathbf{a}]$ is given by 
    \begin{equation}\label{eq.Kant_potential_vf}
        \nabla \Phi[\varrho, \mathbf{x}, \mathbf{a}](x)
        = 
        \sum_{i = 1}^N 
        \left(
          x - x_i
        \right)\mathds{1}_{\Omega_i},
    \end{equation}
    and its Laplacian is the following measure 
    \begin{equation}\label{eq.Kant_potential_Laplacian}
      \Delta \Phi[\varrho, \mathbf{x}, \mathbf{a}] 
      = 
      d 
      + 
      \sum_{j < i} 
      (x_i-x_j)\cdot n_{\Omega_i}
      \mathscr{H}^{d-1}\mres \partial\Omega_j \cap \partial\Omega_i.
    \end{equation}
    In particular, it holds that $\Delta \Phi[\varrho, \mathbf{x}, \mathbf{a}] \le d$. 
\end{lemma}
\begin{proof}
    For simplicity of notation we shall refer to $\Phi[\varrho, \mathbf{x}, \mathbf{a}]$ as $\Phi$ omitting the arguments. Notice that from the definition of Laguerre tessellations, the potential defined in~\eqref{eq.Kant_potential} can be rewritten as 
    \[
      \Phi[\varrho, \mathbf{x}, \mathbf{a}](x)
      = 
      \min_{i = 1,\dots,N} 
      \frac{1}{2}|x - x_i|^2 - \psi_i.
    \]
    As an optimal Kantorovitch potential, it is Lipschitz continuous, and therefore it is a.e.~differentiable. As a result, by the envelope theorem its gradient must coincide with the expression in~\eqref{eq.Kant_potential_vf}. We must only compute its Laplacian in the sense of distributions, for this consider some $\varphi \in \mathscr{C}^1(\Omega)$ and compute 
    \begin{align*}
      \int_\Omega \varphi \dd(\Delta \Phi) \dd x
      &= 
      - 
      \int_\Omega \nabla \varphi \cdot \nabla \Phi \dd x
      + 
      \int_{\partial\Omega} \varphi \nabla \Phi\cdot n_\Omega \dd\mathscr{H}^{d-1}\\ 
      &= 
      - 
      \sum_{i = 1}^N 
      \int_{\Omega_i} \nabla \varphi \cdot (x-x_i) \dd x
      + 
      \sum_{i = 1}^N
      \int_{\partial\Omega \cap \partial \Omega_i} \varphi (x-x_i)\cdot n_\Omega \dd\mathscr{H}^{d-1}.
    \end{align*}
    Each integral in the first sum above can be developed as 
    \begin{align*}
      \int_{\Omega_i} \nabla \varphi \cdot (x-x_i) \dd x
      = 
      &-d \int_{\Omega_i}\varphi\dd x 
      + 
      \int_{\partial \Omega_i \cap \partial \Omega} 
      \varphi (x - x_i)\cdot n_\Omega \dd \mathscr{H}^{d-1}\\
      &+ 
      \int_{\partial \Omega_i \setminus \partial \Omega} 
      \varphi (x - x_i)\cdot n_{\Omega_i} \dd \mathscr{H}^{d-1},
    \end{align*}
    where we have used the fact that $n_{\Omega_i}$ coincides with $n_\Omega$ at $\partial \Omega$. Define also $\Sigma_{ij} \eqdef \partial \Omega_i \cap \partial \Omega_j$ for $i \neq j$, and set the convention that $n_{\Sigma_{ij}} = n_{\Omega_i} = - n_{\Omega_j}$. Since the intersections $\partial \Omega_i \cap \partial \Omega_j \cap \partial \Omega$ are $\mathscr{H}^{d-1}$ negligible, summing all these contributions we obtain the following expression for the Laplacian
    \begin{align*}
      \Delta \Phi
      = 
      d 
      + 
      \sum_{j < i} 
      (x_i-x_j)\cdot n_{\Omega_i}
      \mathscr{H}^{d-1}\mres \partial\Omega_j \cap \partial\Omega_i.
    \end{align*}
    However, notice that for any point $x \in \partial\Omega_j \cap \partial\Omega_i$, we have from convexity that 
    \begin{align*}
      (x_i-x_j)\cdot n_{\Omega_i}(x) 
      &= 
      (x_i-x)\cdot n_{\Omega_i}(x)
      -
      (x_j-x)\cdot n_{\Omega_i}(x)\\
      &= 
      (x_i-x)\cdot n_{\Omega_i}(x)
      +
      (x_j-x)\cdot n_{\Omega_j}(x) \le 0,
    \end{align*} 
    and we conclude that $\Delta \Phi \le d$. 
\end{proof} 

Now we pass to the question of $L^p,L^\infty$ estimates for iterates of the JKO scheme and to $H^1$ estimates for the pressure gradient. The arguments are an adaptation of the ones in~\cite{dimarino2022jko,santambrogio2024strong}, relying on the \textit{flow interchange technique}. 

\begin{proof}[Proof of items (1) and (2) of Theorem~\ref{thm.L2H1convergence}:]
    To simplify notation we drop the dependence on $\tau$ from the iterates of the JKO scheme. As in~\cite{dimarino2022jko}, the strategy of the proof consists of combining Lemma~\ref{lemma.divergence_computation} above with the geodesic convexity of the functional 
    \[
      \varrho \mapsto 
      \mathscr{F}_p(\varrho) = \frac{1}{p-1}\int_\Omega \varrho^p\dd x. 
    \]
    Indeed, if $P_p(\varrho), P_p(\eta) \in W^{1,1}(\Omega)$ and $\varrho_t \eqdef {((1-t)\id + tT)}_\sharp \varrho$ is the geodesic between them, it holds, for instance from~\cite[Lemma~10.4.4.]{Ambrosio2005Gradient}, that 
    \begin{align*}
      \left.\frac{\dd}{\dd t}\right|_{t = 0^+} 
      \mathscr{F}_p(\varrho_t) 
      &= 
      \int_\Omega \nabla P_p(\varrho)\cdot (T - \id)\dd x
      - 
      \int_{\partial \Omega} 
      P_p(\varrho) (T-\id)\cdot n_\Omega \dd \mathscr{H}^{d-1}\\ 
      &\ge \int_\Omega \nabla P_p(\varrho)\cdot (T - \id)\dd x . 
    \end{align*}

    As in the notation of Lemma~\ref{lemma.divergence_computation}, we set $\Phi_{k+1} \eqdef \Phi[\varrho_{k+1}, \mathbf{x}_{k+1}, \mathbf{a}_{k+1}]$ and let $T_{k+1} = \id - \nabla \varphi_{k+1}$ denote the optimal transportation map from $\varrho_{k+1}$ to $\varrho_k$. The optimality conditions for the sequence obtained from the JKO scheme tells us that
    \[
      m\varrho_{k+1}^{m-1} = - \Phi_{k+1} - \frac{\varphi_{k+1}}{\tau} \text{ in } \{\varrho_{k+1}>0\}. 
    \]    
    As a result, we conclude that $\varrho_{k+1}^{m-1}$ is Lipschitz continuous and hence $\varrho_{k+1}$ is bounded. Hence for all $p>1$ the pressure $P_p(\varrho_{k+1}) \in L^1(\Omega)$ and
    \begin{align*}
        \nabla P_p(\varrho_{k+1}) 
        &=
        p\frac{m-1}{m}\varrho_{k+1}^{p-m+1}\nabla F'_m(\varrho_{k+1}) \\
        &= 
        -p\frac{m-1}{m}\varrho_{k+1}^{p-m+1}
        \left(
            \nabla \Phi_{k+1} + \frac{\nabla \varphi_{k+1}}{\tau}
        \right)
        \in L^{\infty}(\Omega). 
    \end{align*}

    It follows that $P_p(\Omega) \in W^{1,\infty}(\Omega)$ for all $p>1$, and the geodesic convexity of $\mathscr{F}_p$ gives us that
    \begin{align*}
      \frac{1}{p-1}\int_\Omega \varrho_k^p \dd x -
      \frac{1}{p-1}\int_\Omega \varrho_{k+1}^p \dd x 
      &\ge 
      \int_\Omega \nabla P_p(\varrho_{k+1})\cdot(T_{k+1} - \id)\dd x \\ 
      &= 
      p\int_\Omega 
      \varrho_{k+1}^{p-1} \nabla \varrho_{k+1}\cdot(-\nabla \varphi_{k+1})\dd x\\ 
      &= 
      \tau p\int_\Omega 
      \varrho_{k+1}^{p-2} \nabla \varrho_{k+1}
      \cdot
      \left(\nabla P_m(\varrho_{k+1}) + \varrho_{k+1} \nabla \Phi_{k+1}\right)\dd x
    \end{align*}
    And we obtain the key estimate 
    \begin{equation}\label{eq.key_estimate}
      \int_\Omega \varrho_k^p \dd x -
      \int_\Omega \varrho_{k+1}^p \dd x 
      \ge 
      \tau(p-1)p
      \left(
        m\int_\Omega 
        \varrho_{k+1}^{p+m-3} |\nabla \varrho_{k+1}|^2\dd x 
        + 
        \int_\Omega 
        \nabla{\left(\varrho_{k+1}\right)}^p\cdot \nabla \Phi_{k+1}\dd x
      \right)
    \end{equation}

    To obtain the $L^p$ estimates, we bound from below the first term in the RHS of~\eqref{eq.key_estimate} by $0$ and use integration by parts in the second term in order to obtain 
    \begin{align*}
      \int_\Omega \varrho_k^p \dd x -
      \int_\Omega \varrho_{k+1}^p \dd x 
      &\ge 
      \tau(p-1)
      \left(
        \int_\Omega 
        \varrho_{k+1}^p (-\Delta \Phi_{k+1})\dd x 
        + 
        \int_{\partial \Omega}
        \varrho_{k+1}^p  \nabla \Phi_{k+1}\cdot n_\Omega\dd \mathscr{H}^{d-1} 
      \right)\\
      &\ge 
      -\tau(p-1)d\int_{\Omega}\varrho_{k+1}^p\dd x,
    \end{align*}
    where we have used the fact that $-\Delta \Phi_{k+1}\ge -d$ from Lemma~\ref{lemma.divergence_computation} and the boundary terms are non-negative from the convexity of the domain as done many times in the proof of Lemma~\ref{lemma.divergence_computation}. This gives the desired estimate from item (1), and passing to the limit as $p\to \infty$, we obtain the $L^\infty$ bounds from item (2). 

    Moving on to the pressure gradient estimates, assuming that $\varrho_0 \in L^{m+1}(\Omega)$ and considering $p = m + 1$ in~\eqref{eq.key_estimate} we obtain
    \begin{align*} 
      \int_\Omega \varrho_k^{m+1} \dd x 
      &-
      \int_\Omega \varrho_{k+1}^{m+1} \dd x
      \ge 
      \tau m
      \left(
        m\int_\Omega 
        \varrho_{k+1}^{2m-2}|\nabla \varrho_{k+1}|^2\dd x
        + 
        \int_\Omega 
        \nabla (\varrho_{k+1}^{m+1})\cdot \nabla \Phi_{k+1}\dd x
      \right)\\ 
      &= 
      \tau
      \left(
        m^2\int_\Omega |\nabla P_m(\varrho_{k+1})|^2 \dd x 
        + 
        (m+1)m
        \int_\Omega
        \nabla P_m(\varrho_{k+1})\cdot \nabla \Phi_{k+1}\varrho_{k+1}\dd x 
      \right)
      \\ 
      &\ge 
      \tau 
      \left(
        \frac{m^2}{2}\int_\Omega 
        |\nabla P_m(\varrho_{k+1})|^2\dd x 
        -
        \frac{m^2}{2}
        \int_\Omega
        |\nabla \Phi_{k+1}|^2\varrho_{k+1}^2\dd x
      \right),
     \end{align*}
     where the last step comes from the $\varepsilon$-Young's inequality, with $\varepsilon = (m+1)/m$. Finally, summing both sides of the above inequality over $k$, the LHS telescopes and and recalling that $|\nabla \Phi_{k+1}| \le \diam(\Omega)$ and that $\varrho^\tau = \varrho_{k+1}$ over the interval $(k\tau,(k+1)\tau)$ we obtain
     \begin{align*}
      \int_0^T\int_\Omega |\nabla P(\varrho^\tau)|^2 \dd x\dd t
      &=
      \sum_{k} \tau\int_\Omega |\nabla P(\varrho_{k+1})|^2 \dd x\\ 
      &\le 
      \sum_{k} \tau\int_\Omega |\nabla \Phi_{k+1}|^2\varrho_{k+1}^2 \dd x 
      + 
      \frac{2}{m}
      \int_\Omega \varrho_0^{m+1}\dd x
      \le C_{m,\Omega,\varrho_0}.
     \end{align*}

     Finally, to get an $L^2H^1$ bound on $P_m(\varrho_{k+1})$, notice that since $\varrho_0 \in L^{m}(\Omega)$, from item (1), this remains true for all subsequent $\varrho_k$, for $\tau$ small enough. So the mean of $P_m(\varrho_{k+1})$ remains uniformly bounded and Poincaré-Wirtinger inequality gives 
     \begin{align*}
        \norm{P_m(\varrho_k)}_{L^2(\Omega)} 
        &\le 
        \norm{
          P_m(\varrho_k)
          - 
          \fint_\Omega P_m(\varrho_k)
        }_{L^2(\Omega)} 
        + 
        \fint_\Omega P_m(\varrho_k)\\
        &\le 
        C \norm{\nabla P_m(\varrho_k)}_{L^2(\Omega)}
        + 
        \fint_\Omega P_m(\varrho_k). 
     \end{align*}
     It follows that $P_m(\varrho^\tau)$ remains uniformly bounded in $\tau$ inside $L^2([0,T];H^1(\Omega))$.
\end{proof} 

Consider now a subsequence, not relabelled, of ${(\varrho^\tau, \mathbf{x}^\tau, \mathbf{a}^\tau)}_{\tau > 0}$ for which the convergence to the limiting equation~\eqref{eq.PDE_ODEsysthem} holds from Thm.~\ref{thm.convergenceJKO}. From the previous reasoning ${(\varrho^\tau)}_{\tau > 0}$ is bounded in $L^2([0,T];H^1(\Omega))$, so we conclude that it converges in the weak topology of this Hilbert space to $\varrho$. Indeed, since Thm.~\ref{thm.convergenceJKO} implies $L^1([0,T];\Omega)$ convergence of this fixed subsequence of ${(\varrho^\tau)}_{\tau > 0}$ to some solution, we can extract a further convergent subsequence in $L^2([0,T];H^1(\Omega))$. This limit must then be the same solution, otherwise we would have a contradiction with the $L^1$ convergence. From the Urysohn property of the weak $L^2([0,T];H^1(\Omega))$-convergence, it follows that the original sequence must also converge weakly in $L^2([0,T];H^1(\Omega))$.

Since it is a Hilbert space, to obtain strong convergence it suffices to prove convergence of the norms. This is done with the interpolation introduced in~\cite[Section~4]{santambrogio2024strong} given by 
\begin{equation}
  \varrho_t^{\tau, \varepsilon} 
  \eqdef 
  \begin{cases}
    \varrho^\tau_{k+1},
    & t \in (k\tau, k\tau + (1-\varepsilon)\tau],\\ 
    \varrho^\tau_{k+1}\frac{(k+1)\tau - t}{\varepsilon\tau} 
    + 
    \varrho^\tau_{k+2}\frac{t - (k+1)\tau + \varepsilon\tau}{\varepsilon\tau},
    & t \in (k\tau + (1-\varepsilon)\tau, (k+1)\tau],\\
    \varrho^\tau_{\left\lceil\frac{T}{\tau}\right\rceil}, 
    & t \in \left(\left\lceil\frac{T}{\tau}\right\rceil,T\right]. 
  \end{cases}
\end{equation} 
The advantage of this new family is that it is well adapted to apply a suitable version of the Aubin-Lions-Simon compactness theorem in $L^p([0,T];B)$, where $B$ is a Banach space, while still being uniformly close to $\varrho^\tau$ in the $L^2([0,T],H^1(\Omega))$ topology. 

Since the application of this compactness theorem depends only on the construction of the interpolation, we summarize in the following Lemma the convergence properties obtained in~\cite[Prop.~4.2 and Cor.~4.4]{santambrogio2024strong}. 

\begin{lemma}\label{lemma.L2L2_strongconv}
    Let ${\left(\varrho^\tau\right)}_{\tau > 0}$ be a subsequence of the staircase interpolation of the JKO scheme which is strongly convergent in $L^1([0,T]\times \Omega)$, as in Thm.~\ref{thm.convergenceJKO}, to $\varrho$. Then both ${\left(\varrho^\tau\right)}_{\tau > 0}$ and ${\left(P_m(\varrho^\tau)\right)}_{\tau > 0}$ converge to $\varrho$ and $P_m(\varrho)$, respectively, strongly in $L^2\left([0,T];L^2(\Omega)\right)$.
\end{lemma}

With these elements we can finish the proof the strong $L^2\left([0,T];H^1(\Omega)\right)$ convergence, whose proof structure is strongly inspired by~\cite[Thm.~4.5]{santambrogio2024strong}.

\begin{proof}[Proof of Thm.~\ref{thm.L2H1convergence}]
  Let ${\left(\varrho, \mathbf{x}, \mathbf{a}\right)}_{t \in [0,T]}$ be a solution of the coupled system obtained as the limit of a staircase interpolation ${\left(\varrho^\tau, \mathbf{x}^\tau, \mathbf{a}^\tau\right)}_{\tau > 0}$. To facilitate notation, recall the definition of the potential $\Phi[\varrho, \mathbf{x},\mathbf{a}]$ introduced in~\eqref{eq.Kant_potential} and define
  \begin{align*}
     \Phi[\varrho] &\eqdef \Phi[\varrho,\mathbf{x},\mathbf{a}],\\
     \Phi^\tau &\eqdef \Phi[\varrho^\tau,\mathbf{x}^\tau,\mathbf{a}^\tau],\\ 
     \Phi_{k+1} &\eqdef \Phi[\varrho^\tau_{k+1},\mathbf{x}^\tau_{k+1},\mathbf{a}^\tau_{k+1}]. 
  \end{align*}
  
  In Thm~\ref{thm.convergenceJKO}, we show that $\varrho$ is a solution in a weak sense to the equation 
  \begin{equation}\label{eq.auxPME}
    \partial_t \varrho 
    = 
    \nabla P_m(\varrho) + \divv\left(\varrho \nabla \Phi[\varrho]\right), 
  \end{equation}
  with non-flux boundary conditions and such that $P_m(\varrho) \in L^1\left([0,T];\BV(\Omega)\right)$. In the first part of the proof of the current theorem, we actually show that $P_m(\varrho) \in L^2\left([0,T];H^1(\Omega)\right)$. As a result, we can use $P_m(\varrho)$ as a test function for~\eqref{eq.auxPME} giving that 
  \begin{align*}
    -\int_0^T\int_\Omega &\left(
      |\nabla P_m(\varrho)|^2 + 
      \varrho \nabla P_m(\varrho)\cdot \nabla \Phi[\varrho]
    \right)\dd x \dd t
    = 
    \int_0^T\int_\Omega P_m(\varrho) \partial_t \varrho \dd x \dd t\\ 
    &=
    \int_0^T \frac{\dd}{\dd t} 
    \left(
      \frac{1}{m+1}\int_\Omega \varrho_t^{m+1}\dd x
    \right)\dd t
    = 
    \frac{1}{m+1}\int_\Omega \varrho_T^{m+1}\dd x 
    - 
    \frac{1}{m+1}\int_\Omega \varrho_0^{m+1}\dd x.
  \end{align*}
  In other words we have the energy relation
  \begin{equation}\label{eq.energy_relation}
    \int_\Omega \varrho_T^{m+1}\dd x 
    - 
    \int_\Omega \varrho_0^{m+1}\dd x
    = 
    -(m+1)\int_0^T\int_\Omega \left(
      |\nabla P_m(\varrho)|^2 + 
      \varrho \nabla P_m(\varrho)\cdot \nabla \Phi[\varrho]
    \right)\dd x \dd t
  \end{equation}

  On the other hand, using the case $p = m+1$ in the proof of the first part of the present theorem, for instance in equation~\eqref{eq.key_estimate}, it holds for the discrete iterates of the JKO scheme that 
  \begin{align*}
      \int_\Omega \left(\varrho_k^\tau\right)^{m+1} \dd x
      &-
      \int_\Omega \left(\varrho_{k+1}^\tau\right)^{m+1} \dd x \\
      & \ge 
      \tau
      \left(
        \int_\Omega 
        |\nabla P_m\left(\varrho_{k+1}^\tau\right)|^2\dd x 
        + 
        \int_\Omega 
        \varrho_{k+1}^\tau
        \nabla P_m\left(\varrho_{k+1}^\tau\right)\cdot \nabla \Phi_{k+1}\dd x
      \right).
  \end{align*}
  Hence, summing over $k$ the LHS telescopes, and we obtain that
  \begin{equation}\label{eq.discrete_energy_relation}
    \int_\Omega \left(\varrho_0^\tau\right)^{m+1} \dd x -
    \int_\Omega \left(\varrho_{T}^\tau\right)^{m+1} \dd x
    \ge 
    \int_0^T\int_\Omega 
    \left(
      |\nabla P_m\left(\varrho^\tau\right)|^2 
        + 
        \varrho^\tau
        \nabla P_m\left(\varrho^\tau\right)\cdot \nabla \Phi^\tau
    \right)
    \dd x \dd t. 
  \end{equation}

  From the pointwise convergence of $\varrho^\tau_t$ to $\varrho_t$ in Wasserstein topology, for every time $t$, and the weak convergence of $P_m(\varrho^\tau)$ we obtain that 
  \begin{equation}\label{eq.semi_cont_conv}
    \begin{aligned}
      \int_\Omega {\left(\varrho_{T}\right)}^{m+1} \dd x
    &\le 
    \liminf_{\tau \to 0^+}
    \int_\Omega {\left(\varrho_{T}^\tau\right)}^{m+1} \dd x, \\ 
    \int_0^T\int_\Omega 
      |\nabla P_m\left(\varrho\right)|^2 
    \dd x \dd t
    &\le 
    \liminf_{\tau \to 0^+}
    \int_0^T\int_\Omega 
      |\nabla P_m\left(\varrho^\tau\right)|^2 
    \dd x \dd t
    \end{aligned}
  \end{equation}
  We can rewrite the following term as 
  \begin{align*}
    \lim_{\tau \to 0}
    \int_0^T\int_\Omega 
        \varrho^\tau
        \nabla P_m\left(\varrho^\tau\right)\cdot \nabla \Phi^\tau
    \dd x \dd t
    &= 
    \sum_{i = 1}^N 
    \lim_{\tau \to 0}
    \int_0^T\int_\Omega 
    \mathds{1}_{\Omega^\tau_i(t)}
    \varrho^\tau
        \nabla P_m\left(\varrho^\tau\right)\cdot(x - x_i^\tau)
    \dd x \dd t\\ 
    &= 
    \sum_{i = 1}^N 
    \int_0^T\int_\Omega 
    \mathds{1}_{\Omega_i(t)}
    \varrho
        \nabla P_m\left(\varrho\right)\cdot(x - x_i)
    \dd x \dd t\\ 
    &= 
    \int_0^T\int_\Omega 
        \varrho
        \nabla P_m\left(\varrho\right)\cdot \nabla \Phi[\varrho]
    \dd x \dd t,
  \end{align*}
  where the limit is computed with weak-strong convergence, since $\mathds{1}_{\Omega_i^\tau} \cvstrong{\tau \to 0}{L^1([0,T]\times \Omega)} \mathds{1}_{\Omega_i}$ for all $i = 1,\dots, N$, so that $\varrho^\tau \mathds{1}_{\Omega_i^\tau}$ converges strongly in $L^2\left([0,T];L^2(\Omega)\right)$ and $\nabla P_m\left(\varrho^\tau\right)\cdot(x - x_i^\tau)$ converges weakly in the same space. 

  As a result, going back to~\eqref{eq.discrete_energy_relation}, using the previous convergence,~\eqref{eq.semi_cont_conv} and the energy relation~\eqref{eq.energy_relation}, we obtain
  \begin{align*}
    \int_0^T\!\!\!\!\int_\Omega 
      &|\nabla P_m\left(\varrho\right)|^2 
    \dd x \dd t
    \!\le \!
    \liminf_{\tau \to 0^+} \!
    \int_0^T\!\!\!\!\int_\Omega 
      |\nabla P_m\left(\varrho^\tau\right)|^2 
    \dd x \dd t
    \!\le \!
    \limsup_{\tau \to 0^+}\!
    \int_0^T\!\!\!\!\int_\Omega 
      |\nabla P_m\left(\varrho^\tau\right)|^2 
    \dd x \dd t\\ 
    &\le 
    \limsup_{\tau \to 0^+} 
    \left(
      \int_\Omega {\left(\varrho_0^\tau\right)}^{m+1} \dd x -
      \int_\Omega {\left(\varrho_{T}^\tau\right)}^{m+1} \dd x
      -
      \int_0^T\int_\Omega 
          \varrho^\tau
          \nabla P_m\left(\varrho^\tau\right)\cdot \nabla \Phi^\tau
      \dd x \dd t
    \right)\\
    &\le 
    \int_\Omega {\left(\varrho_0\right)}^{m+1} \dd x -
      \int_\Omega {\left(\varrho_{T}\right)}^{m+1} \dd x
      -
      \int_0^T\int_\Omega 
          \varrho
          \nabla P_m\left(\varrho\right)\cdot \nabla \Phi[\varrho]
      \dd x \dd t\\ 
      &= 
      \int_0^T\int_\Omega 
      |\nabla P_m\left(\varrho\right)|^2 
      \dd x \dd t.
  \end{align*}

  From Lemma~\ref{lemma.L2L2_strongconv} we already had the strong convergence of $P_m(\varrho^\tau)$ to $P_m(\varrho)$ in $L^2\left([0,T];L^2(\Omega)\right)$, therefore the previous computation implies convergence of the norms of $P_m(\varrho^\tau)$ in $L^2\left([0,T]; H^1(\Omega)\right)$, which combined with the weak convergence gives the desired result. 
\end{proof}

\section{Qualitative Properties}\label{sec.qualitative_props}
In this section we investigate the qualitative behavior of a simplified PDE obtained as the limit equation of the minimizing movement scheme introduced earlier, but without the boundary effects encoded by the normal cone. All assumptions on the domain and data are those stated in the introduction. The resulting coupled PDE–ODE system reads
\begin{equation}\label{eq.PDE_ODEsysthem_smoother}
  \begin{cases}
      &\displaystyle
      \partial_t \varrho_t 
      = \Delta P(\varrho_t)
      + \ddiv\left(\varrho_t \left(\sum_{i = 1}^N(x - x_i(t))\mathbbm{1}_{\Omega_i(t)}\right)\right)\\
      &\displaystyle
      \dot{x}_{i}(t) 
      =
      - a_i(t)x_i(t) + \int_{\Omega_i(t)}x \dd \varrho_t \\ 
      &\displaystyle 
      \dot{a}_i(t) = \left(-g'(a_i(t)) + \psi_i(t)\right)\mathds{1}_{\{a_i>0\}}
      \\ 
      &\displaystyle 
      \Omega_i(t) = \Lag_i(\psi_t, \mathbf{x}_t), \ \psi(t) \text{ is a potential for } W_2^2(\varrho_t, \mu_t).
  \end{cases}
\end{equation}

We first prove that for any solution of~\eqref{eq.PDE_ODEsysthem_smoother}, not only those obtained from the minimizing movement scheme, the atoms remain strictly inside the domain. Moreover, when the boundary $\partial \Omega$ is smooth, any atom initially on the boundary is instantaneously pushed into the interior (see Theorem~\ref{thm.qualitative}). We then focus on the uniform quantization case $a_i \equiv 1/N$, where we show that the atoms stay uniformly separated and that the distance between any point and its Laguerre cell's barycenter converges to zero.

\subsection{Invariant properties of the atomic dynamics}\label{sec.qualitative_atoms}
In this subsection we establish several qualitative properties of the coupled system~\eqref{eq.PDE_ODEsysthem}. The main result is the following theorem.

\begin{theorem}\label{thm.qualitative}
Let $\Omega$ be an open, bounded, convex set, and let ${\left(\varrho_t, \mathbf{x}_t, \mathbf{a}_t\right)}_{t \in [0,T]}$ be a solution of~\eqref{eq.PDE_ODEsysthem}. Then the following properties hold:
\begin{enumerate}
\item The family ${\left(\varrho_t\right)}_{t \ge 0}$ is uniformly integrable.
\item If for some $\bar t \ge 0$ we have $a_i(\bar t) = 0$, then $a_i(t) = 0$ for all $t \ge \bar t$.
\item If $\partial \Omega$ is of class $\mathscr{C}^1$ and $x_i(0) \in \partial \Omega$, then $x_i(t)$ is instantaneously pushed towards the interior, that is $x_i(t) \in \operatorname{int}\Omega$ for all $t$ in a sufficiently small neighborhood of $0$.
\item Suppose that either $\partial \Omega$ is of class $\mathscr{C}^1$, or ${\left(x_i(0)\right)}_{i = 1}^N \subset \operatorname{int}\Omega$. Then, for any $t>0$, if $x_i(t) \in \partial \Omega_i(t)$ it follows that $a_i(t) = 0$.
\end{enumerate}
\end{theorem}

\begin{figure}[ht]\label{figure.invariant_domain}
\centering
\tikzset{every picture/.style={line width=0.75pt}} 

\begin{tikzpicture}[x=0.75pt,y=0.75pt,yscale=-1,xscale=1]

\draw   (458.16,131.32) .. controls (480.83,198.54) and (431.46,271.47) .. (347.88,294.21) .. controls (264.3,316.95) and (178.16,280.89) .. (155.49,213.67) .. controls (132.81,146.44) and (182.19,73.51) .. (265.77,50.77) .. controls (366.13,23.46) and (416.39,10.03) .. (416.53,10.44) .. controls (417.05,10.3) and (430.93,50.59) .. (458.16,131.32) -- cycle ;
\draw  [dash pattern={on 4.5pt off 4.5pt}] (423.19,146.29) .. controls (423.19,146.29) and (423.19,146.29) .. (423.19,146.29) .. controls (440.56,196.89) and (402.2,251.91) .. (337.51,269.18) .. controls (272.82,286.46) and (206.29,259.44) .. (188.92,208.84) .. controls (171.54,158.24) and (209.9,103.21) .. (274.59,85.94) .. controls (274.59,85.94) and (274.59,85.94) .. (274.59,85.94) .. controls (352.28,65.2) and (391.17,54.98) .. (391.28,55.3) .. controls (391.69,55.2) and (402.32,85.52) .. (423.19,146.29) -- cycle ;
\draw  [color={rgb, 255:red, 208; green, 2; blue, 27 }  ,draw opacity=1 ][line width=1.5]  (175.1,186.2) .. controls (174.1,163.2) and (189.1,133.2) .. (204.1,117.2) .. controls (219.1,101.2) and (257.43,85.33) .. (279.1,76.6) .. controls (300.77,67.87) and (334.1,59.6) .. (348.1,56.6) .. controls (362.1,53.6) and (388.1,46.6) .. (394.1,50.6) .. controls (400.1,54.6) and (402.1,57.6) .. (406.1,68.6) .. controls (410.1,79.6) and (412.1,84.6) .. (414.1,92.6) .. controls (416.1,100.6) and (428.1,127.6) .. (432.1,168.6) .. controls (436.1,209.6) and (403.1,244.2) .. (388.1,252.2) .. controls (373.1,260.2) and (369.1,269.6) .. (320.1,279.6) .. controls (271.1,289.6) and (241.1,277.2) .. (218.1,262.2) .. controls (195.1,247.2) and (193.1,243.2) .. (185.1,230.2) .. controls (177.1,217.2) and (176.1,209.2) .. (175.1,186.2) -- cycle ;
\draw  [fill={rgb, 255:red, 0; green, 0; blue, 0 }  ,fill opacity=1 ] (273.07,114.89) .. controls (273.07,113.77) and (273.89,112.86) .. (274.91,112.86) .. controls (275.92,112.86) and (276.74,113.77) .. (276.74,114.89) .. controls (276.74,116) and (275.92,116.91) .. (274.91,116.91) .. controls (273.89,116.91) and (273.07,116) .. (273.07,114.89) -- cycle ;
\draw  [fill={rgb, 255:red, 0; green, 0; blue, 0 }  ,fill opacity=1 ] (287.03,72.88) .. controls (287.03,71.77) and (287.85,70.86) .. (288.87,70.86) .. controls (289.88,70.86) and (290.7,71.77) .. (290.7,72.88) .. controls (290.7,74) and (289.88,74.91) .. (288.87,74.91) .. controls (287.85,74.91) and (287.03,74) .. (287.03,72.88) -- cycle ;
\draw    (288.87,70.86) -- (286.38,60.57) -- (284.55,53.72) ;
\draw [shift={(284.03,51.79)}, rotate = 75] [color={rgb, 255:red, 0; green, 0; blue, 0 }  ][line width=0.75]    (10.93,-3.29) .. controls (6.95,-1.4) and (3.31,-0.3) .. (0,0) .. controls (3.31,0.3) and (6.95,1.4) .. (10.93,3.29)   ;
\draw    (288.87,72.88) -- (275.51,115) ;
\draw [shift={(274.91,116.91)}, rotate = 287.59] [color={rgb, 255:red, 0; green, 0; blue, 0 }  ][line width=0.75]    (10.93,-3.29) .. controls (6.95,-1.4) and (3.31,-0.3) .. (0,0) .. controls (3.31,0.3) and (6.95,1.4) .. (10.93,3.29)   ;
\draw  [dash pattern={on 0.84pt off 2.51pt}]  (217.64,79.15) -- (274.7,162.76) ;
\draw  [dash pattern={on 0.84pt off 2.51pt}]  (274.7,162.76) -- (330.55,150.24) ;
\draw  [dash pattern={on 0.84pt off 2.51pt}]  (330.55,150.24) -- (414,14.61) ;
\draw  [dash pattern={on 4.5pt off 4.5pt}] (442.23,140.14) .. controls (462.44,198.64) and (417.8,262.26) .. (342.51,282.24) .. controls (267.22,302.21) and (189.8,270.98) .. (169.59,212.48) .. controls (149.38,153.98) and (194.03,90.36) .. (269.31,70.39) .. controls (359.73,46.4) and (404.99,34.59) .. (405.12,34.96) .. controls (405.58,34.83) and (417.96,69.89) .. (442.23,140.14) -- cycle ;

\draw (389.09,115.98) node [anchor=north west][inner sep=0.75pt]  [font=\footnotesize]  {$\Omega _{\delta }$};
\draw (250.93,33.29) node [anchor=north west][inner sep=0.75pt]  [font=\footnotesize]  {$n_{D}( y)$};
\draw (294.66,81.93) node [anchor=north west][inner sep=0.75pt]  [font=\footnotesize]  {$y\in \partial D$};
\draw (276.03,117.92) node [anchor=north west][inner sep=0.75pt]  [font=\footnotesize]  {$\beta _{i}( t,y)$};
\draw (341.89,136.13) node [anchor=north west][inner sep=0.75pt]  [font=\footnotesize]  {$\Omega _{i}( t)$};
\draw (410.1,171) node [anchor=north west][inner sep=0.75pt]  [font=\footnotesize,color={rgb, 255:red, 208; green, 2; blue, 27 }  ,opacity=1 ]  {$\mathscr{D}$};
\draw (436.47,113.53) node [anchor=north west][inner sep=0.75pt]  [font=\footnotesize]  {$\Omega _{\frac{\delta }{2}}$};

\end{tikzpicture}
\caption{Construction of the invariant region $\mathscr{D}$. The uniform integrability of $\varrho_t$ ensures that the barycenter $b_i(t,y)$ remains inside $\Omega_\delta$, allowing the construction of a smooth and convex invariant set $\mathscr{D}$ between the boundaries of $\Omega_\delta$ and $\Omega_{\delta/2}$.}
\end{figure}
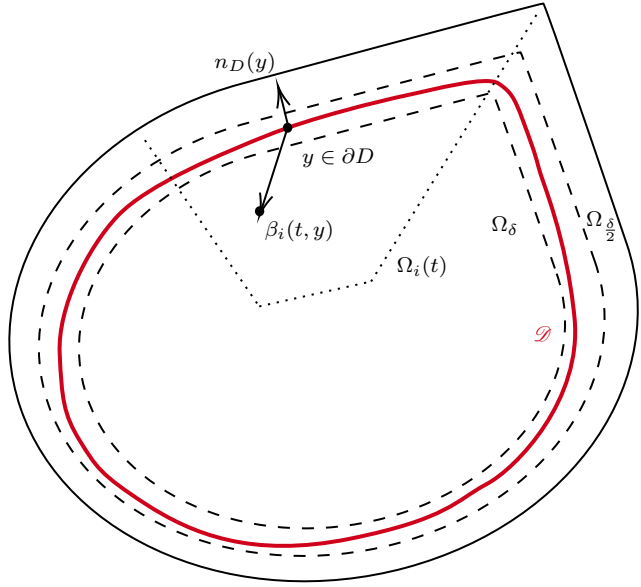

The proof relies on a dynamical systems approach: we construct a positively invariant domain for the dynamics of the atomic positions ${\left(x_i\right)}_{i = 1}^N$, while keeping ${\left(\varrho_t, a_i(t)\right)}$ fixed. In this context, a positively invariant domain is a region that cannot be left by the trajectories generated by the system.

\begin{definition}\label{def.positively_invariant_set}
Let $V: \mathbb{R}+ \times \mathbb{R}^d \to \mathbb{R}^d$ be a vector field, measurable in time and continuous in space. A set $A \subset \mathbb{R}^d$ is said to be positively invariant for the dynamics induced by $V$ if every solution ${\left(x(t)\right)}{t \ge 0}$ of the ODE $\dot{x}_t = V(t,x_t)$ with $x(0) \in A$ satisfies $x(t) \in A$ for all $t \ge 0$.
\end{definition}

The theory of positively invariant sets originates from the work of Nagumo~\cite{nagumo1942lage} (see also~\cite[Chap.~4]{blanchini2008set}). Here we exploit the convexity of $\Omega$ together with the uniform integrability of ${\left(\varrho_t\right)}_{t \ge 0}$, inherited from the energy functional $\mathscr{F}$, to apply a simpler classical criterion. In particular, if $A$ has a smooth boundary, a sufficient condition for $A$ to be invariant under the flow induced by $V$ is that $\inner{V(t,y), n_A(y)} < 0$ for all $y \in \partial A$, where $n_A(y)$ denotes the outwards unity normal of $\partial A$ evaluated at $y$, see for instance~\cite[Lemma~1]{fernandes1987remarks}.

\begin{proof}
    \textbf{\underline{Item (1)}:} The uniform integrability of ${\left(\varrho_t\right)}_{t \in [0,T]}$ follows from De la Vallée Poussin's Theorem and the gradient flow structure of the equation on $\varrho_t$. Indeed, using the gradient flow structure of the equation satisfied by $\varrho_t$ we have 
    \begin{align*}
      \frac{\dd}{\dd t} \mathscr{E}(\varrho_t, \mu_t) \le 0, 
      \text{ so that }
      \mathscr{F}(\varrho_t) \le \mathscr{E}(\varrho_t, \mu_t) \le \mathscr{E}(\varrho_0, \mu_0). 
    \end{align*}

    \textbf{\underline{Item (3)}:} In the case that $\partial \Omega$ is $\mathscr{C}^1$, the normal cone at each point of the boundary is a subspace of dimension $1$, that is given by $N_\Omega(x) = n_\Omega(x)\mathbb{R}$, whereas the tangent space to $\partial \Omega$ at $x$, is given by the orthogonal subspace of this vector $T_x\partial \Omega = \{n_\Omega(x)\}^\perp$, but should not be mistaken with Bouligand's tangent cone~\eqref{eq.tangent_cone} $T_\Omega(x)$. To perform a proof by contradiction, assume that there is an interval $(0,\tau)$ for which $x_i(t) \in \partial \Omega$, then it follows that $\dot{x}_i(0) \in T_{x_i(0)}\partial \Omega$, and we have that $\dot{x}_i(0) \cdot n_\Omega(x_i(0)) = 0$. 

    Define the family of convex sets 
    \[
        \Omega^\delta \eqdef \left\{
            x \in \Omega : \dist(x, \partial \Omega) \ge \delta
        \right\},
    \]
    so that $|\Omega \setminus \Omega^\delta| \le C \delta$. We conclude from the uniform integrability of ${\left(\varrho_t\right)}_{t \in [0,T]}$ that, the baricenter 
    \[
        b_i(t) \eqdef 
        \frac{1}{a_i(t)} \int_{\Omega_i(t)} x \dd \varrho_t
        \in \intt \Omega_\delta,
    \]
    for all $\delta$ sufficiently small. 

    To finish our construction that will lead to a contradiction, notice that $\dot{x}_i(0)$ can be written as
    \[
        \dot{x}_i(0) = a_i(0)(b_i - x_i(0)) + n_i, \text{ where } n_i \text{ is parallel to } n_\Omega(x_i(0)). 
    \] 
    Let $x_\delta$ denote the projection of $x_i(0)$ onto the convex set $\Omega^\delta$, so that $x_i(0) - x_\delta = \delta \frac{n_i}{\norm{n_i}}$. As a result, the fact that $\dot{x}_i(0)\cdot n_i = 0$ implies, for all $\delta>0$ small enough so that $b_i \in \Omega^\delta$, that 
    \begin{align*}
        0 &= \norm{n_i}^2 + a_i(0)\inner{b_i - x_\delta, n_i} + a_i(0)\inner{x_\delta - x_i(0), n_i}\\ 
          &=  \norm{n_i}^2 + a_i(0)\inner{b_i - x_\delta, n_i} - a_i(0)\norm{n_i}\delta,\\ 
          &=  \norm{n_i}^2 + \frac{a_i(0)}{2}\left(\delta^2 + \norm{b_i - x_\delta}^2 - \norm{x_i(0) - b_i}^2\right) - a_i(0)\norm{n_i}\delta,
    \end{align*}
    where the last equality was obtained via the cosine law for the triangle formed by the points $b_i,x_i(0),x_\delta$. But notice that this equality cannot be true for infinitely many values of $\delta$, which leads to a contradiction with the hypothesis that $x_i(t) \in \partial \Omega$ for all $t \in (0,\tau)$, and we conclude that $x_i$ enters the interior instantaneously.

    \textbf{\underline{Item (4)}:} 
    Assume by contradiction that $\bar t > 0$ is the first time that $x_i$ reaches the boundary $\partial \Omega$ and is such that $a_i(\bar t) > 0$. Our approach is to construct an invariant region for the dynamics satisfied by $x_i$. Fixing the trajectories of ${\left(\varrho_t, \mathbf{a}_t\right)}_{t \in [0,T]}$, notice that for each $i = 1,\dots,N$ the ODE describing the evolution of $t \mapsto x_i(t)$ is given by 
  \[
    \dot{x}_i(t) = V_i(t,x_i(t)), 
    \text{ where }
    V_i(t,y) \eqdef 
    -\nabla_y W_2^2\left(
      \varrho_t, a_i(t)\delta_y + \sum_{j \neq i} a_j(t) \delta_{x_j(t)}
    \right).  
  \]
  Therefore, we have that
  \[
    V_i(t,y)
    = 
    a_i(t)(
      \beta_i(t,y) - y
    )
    \text{ where }
    \beta_i(t,y) \eqdef \frac{1}{a_i(t)}\int_{\Gamma_i(t,y)} x \dd \varrho_t
  \]
  and $\Gamma_i(t,y)$ corresponds to the $i$-th Laguerre cell of the tessellation induced by the semi-discrete transport of $\varrho_t$ to $ \displaystyle a_i(t)\delta_y + \sum_{j \neq i} a_j(t) \delta_{x_j(t)}$. As a result, $V_i$ is clearly measurable in time and the $\mathscr{C}^0$ regularity in space comes from the differentiability properties of the semi-discrete transport term~\cite[Thm~1 and Prop.~2]{de2019differentiation}.


   
  To construct an invariant region, set $\delta \eqdef \min_{i = 1,\dots, N} \dist(x_i(0), \partial \Omega) > 0$ and define 
  \[
    \Omega_s \eqdef \left\{
      x \in \Omega : \dist(x, \partial \Omega) \ge s
    \right\}.
  \]
  Hence, from the uniform integrability and the convexity of $\Omega$, for any $\varepsilon> 0$ there is some $\delta < \bar \delta$ and a compact, convex set with smooth boundary $\mathscr{D}_{\bar \delta}$ such that 
  \[
    \Omega_{\bar \delta/2} \subseteq \mathscr{D}_{\bar \delta} \subseteq \Omega, 
    \text{ with }
    \varrho_t(\Omega \setminus \Omega_{\bar \delta}) \le \varepsilon 
    \text{ for all } t\in[0,T]. 
  \]

  Set $\bar t \eqdef \min\left\{ t\ge 0 : x_i(t) \in \partial \Omega \right\}$ and suppose by contradiction that $\bar a_i \eqdef a_i(\bar t) > 0$. Performing the previous construction with $\varepsilon \ll \frac{\bar a_i}{4}$ it follows that
  \[
    \beta_i(t,y) \in \Omega_{\bar \delta} \text{ for all $t \in [0,T]$ and all $y \in \Omega$}.
  \]
  As a result, since $\mathscr{D}_{\bar \delta}$ has smooth boundary and is convex, for any $y \in \partial \mathscr{D}_{\bar \delta}$ we get that
  \[
    n_{\mathscr{D}_{\bar \delta}}(y)\cdot a_i(t)\left(
        \beta_i(t) - y
    \right) < 0,
  \]
  meaning that $\mathscr{D}_{\bar \delta}$ is invariant for the flow associated with the dynamics of $t \mapsto x_i(t)$. This contradicts the definition of $\bar t$, so that $x_i(\bar t) \in \partial \Omega$ can only happen if $a_i(\bar t) = 0$. 
\end{proof}

\subsection{On the dynamics of the optimal quantization equation}\label{sec.quantization}

In this paragraph we focus on a simplified form of the coupled system by fixing the weights \(a_i = \frac{1}{N}\) for all \(i = 1, \ldots, N\). In this case, equation~\eqref{eq.PDE_ODEsysthem_smoother} becomes the gradient flow of the energy 
\[
    \mathscr{E}(\varrho, \mathbf{x}) 
    \eqdef 
    \mathscr{H}(\varrho) 
    + 
    F_N(\varrho, \mathbf{x}), 
    \text{ where } 
    F_N(\varrho, \mathbf{x}) 
    \eqdef 
    \frac{1}{2} W_2^2\left(\varrho, \frac{1}{N}\sum_{i = 1}^N \delta_{x_i}\right),
\] 
which corresponds to the uniform quantization problem. This way the equation~\eqref{eq.PDE_ODEsysthem_smoother} assumes the form
\begin{equation}\label{eq:longtime}
    \begin{cases}
        \partial_t \varrho = \Delta \varrho + \divv
        \left(\varrho \nabla \Phi_{t}[\varrho_N, \mathbf{x}_t]\right) , & \text{in } (0,\infty) \times \Omega, \\
        -\dot{\mathbf{x}}_{t} = \nabla_{\mathbf{x}}F_N(\varrho_t, \mathbf{x}_t) 
        =
        \frac{1}{N}(\mathbf{x}_{t} - \mathbf{b}_{t}).
    \end{cases}
\end{equation}

Existence of solutions is guaranteed from the JKO scheme, provided that the atoms ${\left(x_{i,t}\right)}_{i = 1}^N$ remain apart and away from the boundary $\partial \Omega$ for all $t>0$. Using item (4) of Theorem~\ref{thm.qualitative} with $a_i(t) \equiv 1/N$, this holds if the initial atoms $\mathbf{x}_0$ are distinct and away from the boundary. In addition, since the equation satisfied by $\mathbf{x}_t$ is a gradient flow of a semi-concave functional, atoms cannot collide in finite time. Therefore, we have existence of solutions to~\eqref{eq:longtime} for all times $t \ge 0$. Similar statements to this were already remarked for instance in~\cite[Section~6.4.2]{santambrogio2015optimal},~\cite{leclerc2020lagrangian}, but we resume this discussion in further details in Lemma~\ref{lem:longtime_wellposed} below.

A particularly relevant case is when the density is frozen $\varrho_t \equiv \varrho$, so the dynamics of $\mathbf{x}_t$ become a continuous-time analog of Loyd's algorithm for uniform barycenter quantization. For the discrete Lloyd algorithm, it is known (see~\cite{merigot2021non, bourne2015centroidal, portales2025sequential}) that the iterates converge toward configurations where each $x_i$ coincides with its barycenter. Our goal here is to establish the analogous continuous result,
\begin{equation}\label{eq.distance_points_barycenters}
    \norm{\mathbf{x}_t - \mathbf{b}_t} \cvstrong{t \to \infty}{} 0. 
\end{equation}
and to describe the regularity and stability properties of the flow leading to this convergence.

The evolution of $\varrho_t$ will only affect the analysis through mild regularity assumptions: as the solution of a Fokker–Planck equation with bounded drift, it is H\"older continuous in time and has uniformly bounded densities. Hence several intermediate results, such as Lemma~\ref{lem:longtime_wellposed}, remain valid for any continuous curve $\varrho_t$ with these properties, including the stationary case $\varrho_t \equiv \varrho$.

Finally, the gradient flow structure yields finite kinectic energy for $\mathbf{x}_t$
\begin{equation}\label{eq.finite_kinectic_energy}
  \int_0^\infty \norm{\dot{\mathbf{x}}_t}^2 \dd t 
  = 
  \int_0^\infty \norm{\mathbf{x}_t - \mathbf{b}_t}^2 \dd t < + \infty, 
\end{equation}
but this alone does not imply~\eqref{eq.distance_points_barycenters}. Uniform continuity of $t \mapsto \mathbf{x}_t - \mathbf{b}_t$ is also needed. While $\cdot{\mathbf{x}}_t$ it is clearly Lipschitz with constant $\diam\Omega$, the evolution of $\mathbf{b}_t$ can degenerate close to the generalized diagonal
\begin{equation}\label{eq.generalized_diagonal}
    \mathbb{D}_N \eqdef 
    \left\{\mathbf{x} \in \Omega^N : x_i = x_j \text{ for some } i \neq j\right\},
\end{equation}
see Figure~\ref{figure.instability_barycenters}. Indeed, the dynamics of the atoms become singular when they approach each other. A refined analysis of neighboring cells $i \sim j$, that is cells such that $\mathscr{H}^{d-1}(\Sigma_{ij}) > 0$ where $\Sigma_{ij} \eqdef \partial \Omega_i\cap \partial \Omega_j$, eads to quantitative lower bounds on pairwise distances and to global well-posedness, which we establish in Lemma~\ref{lem:longtime_wellposed} below.

\begin{figure}[t]\label{figure.instability_barycenters}
  \centering
  \tikzset{every picture/.style={line width=0.6pt}} 

\begin{tikzpicture}[x=0.75pt,y=0.75pt,yscale=-0.8,xscale=0.8]

\draw   (203.69,38) -- (253.83,104.19) -- (189.55,171.67) -- (125.26,136.63) -- (152.26,50.98) -- cycle ;
\draw   (98.26,1.67) -- (152.26,50.98) -- (125.26,136.63) -- (59.69,161.28) -- (1.83,108.08) -- (24.98,27.62) -- cycle ;
\draw  [fill={rgb, 255:red, 208; green, 2; blue, 27 }  ,fill opacity=1 ] (71.26,76.07) .. controls (71.26,74.28) and (69.82,72.82) .. (68.05,72.82) .. controls (66.27,72.82) and (64.83,74.28) .. (64.83,76.07) .. controls (64.83,77.86) and (66.27,79.31) .. (68.05,79.31) .. controls (69.82,79.31) and (71.26,77.86) .. (71.26,76.07) -- cycle ;
\draw  [fill={rgb, 255:red, 208; green, 2; blue, 27 }  ,fill opacity=1 ] (126.55,94.24) .. controls (126.55,92.44) and (125.11,90.99) .. (123.33,90.99) .. controls (121.56,90.99) and (120.12,92.44) .. (120.12,94.24) .. controls (120.12,96.03) and (121.56,97.48) .. (123.33,97.48) .. controls (125.11,97.48) and (126.55,96.03) .. (126.55,94.24) -- cycle ;
\draw  [fill={rgb, 255:red, 74; green, 144; blue, 226 }  ,fill opacity=1 ] (202.4,105.92) .. controls (202.4,104.12) and (200.97,102.67) .. (199.19,102.67) .. controls (197.42,102.67) and (195.98,104.12) .. (195.98,105.92) .. controls (195.98,107.71) and (197.42,109.16) .. (199.19,109.16) .. controls (200.97,109.16) and (202.4,107.71) .. (202.4,105.92) -- cycle ;
\draw  [fill={rgb, 255:red, 74; green, 144; blue, 226 }  ,fill opacity=1 ] (150.98,104.62) .. controls (150.98,102.83) and (149.54,101.37) .. (147.76,101.37) .. controls (145.99,101.37) and (144.55,102.83) .. (144.55,104.62) .. controls (144.55,106.41) and (145.99,107.86) .. (147.76,107.86) .. controls (149.54,107.86) and (150.98,106.41) .. (150.98,104.62) -- cycle ;
\draw   (603.69,39) -- (653.83,105.19) -- (589.55,172.67) -- (525.26,137.63) -- (552.26,51.98) -- cycle ;
\draw   (498.26,2.67) -- (552.26,51.98) -- (525.26,137.63) -- (459.69,162.28) -- (401.83,109.08) -- (424.98,28.62) -- cycle ;
\draw  [fill={rgb, 255:red, 74; green, 144; blue, 226 }  ,fill opacity=1 ] (471.26,77.07) .. controls (471.26,75.28) and (469.82,73.82) .. (468.05,73.82) .. controls (466.27,73.82) and (464.83,75.28) .. (464.83,77.07) .. controls (464.83,78.86) and (466.27,80.31) .. (468.05,80.31) .. controls (469.82,80.31) and (471.26,78.86) .. (471.26,77.07) -- cycle ;
\draw  [fill={rgb, 255:red, 74; green, 144; blue, 226 }  ,fill opacity=1 ] (526.55,95.24) .. controls (526.55,93.44) and (525.11,91.99) .. (523.33,91.99) .. controls (521.56,91.99) and (520.12,93.44) .. (520.12,95.24) .. controls (520.12,97.03) and (521.56,98.48) .. (523.33,98.48) .. controls (525.11,98.48) and (526.55,97.03) .. (526.55,95.24) -- cycle ;
\draw  [fill={rgb, 255:red, 208; green, 2; blue, 27 }  ,fill opacity=1 ] (602.4,106.92) .. controls (602.4,105.12) and (600.97,103.67) .. (599.19,103.67) .. controls (597.42,103.67) and (595.98,105.12) .. (595.98,106.92) .. controls (595.98,108.71) and (597.42,110.16) .. (599.19,110.16) .. controls (600.97,110.16) and (602.4,108.71) .. (602.4,106.92) -- cycle ;
\draw  [fill={rgb, 255:red, 208; green, 2; blue, 27 }  ,fill opacity=1 ] (550.98,104.62) .. controls (550.98,102.83) and (549.54,101.37) .. (547.76,101.37) .. controls (545.99,101.37) and (544.55,102.83) .. (544.55,104.62) .. controls (544.55,106.41) and (545.99,107.86) .. (547.76,107.86) .. controls (549.54,107.86) and (550.98,106.41) .. (550.98,104.62) -- cycle ;
\draw    (263.83,110.67) -- (384.83,110.67) ;
\draw [shift={(386.83,110.67)}, rotate = 180] [color={rgb, 255:red, 0; green, 0; blue, 0 }  ][line width=0.75]    (10.93,-3.29) .. controls (6.95,-1.4) and (3.31,-0.3) .. (0,0) .. controls (3.31,0.3) and (6.95,1.4) .. (10.93,3.29)   ;

\draw (51.33,57.64) node [anchor=north west][inner sep=0.75pt]  [font=\footnotesize]  {$b_{i}$};
\draw (206.9,91.38) node [anchor=north west][inner sep=0.75pt]  [font=\footnotesize]  {$b_{j}$};
\draw (106.62,95.27) node [anchor=north west][inner sep=0.75pt]  [font=\footnotesize]  {$x_{i}$};
\draw (152.05,110.85) node [anchor=north west][inner sep=0.75pt]  [font=\footnotesize]  {$x_{j}$};
\draw (451.33,58.64) node [anchor=north west][inner sep=0.75pt]  [font=\footnotesize]  {$b_{j}$};
\draw (606.9,92.38) node [anchor=north west][inner sep=0.75pt]  [font=\footnotesize]  {$b_{i}$};
\draw (507.62,96.27) node [anchor=north west][inner sep=0.75pt]  [font=\footnotesize]  {$x_{j}$};
\draw (552.05,111.85) node [anchor=north west][inner sep=0.75pt]  [font=\footnotesize]  {$x_{i}$};
\draw (283,89) node [anchor=north west][inner sep=0.75pt]   [align=left] {{\footnotesize swap atoms}};

\end{tikzpicture}
  \caption{A small variation of the atoms does not impact the barycenter too much, unless they are too close. In this situation, where the points are very close, swapping their labels implies a small variation in their positions, but exchanges the barycenters that were very far apart.}
\end{figure}
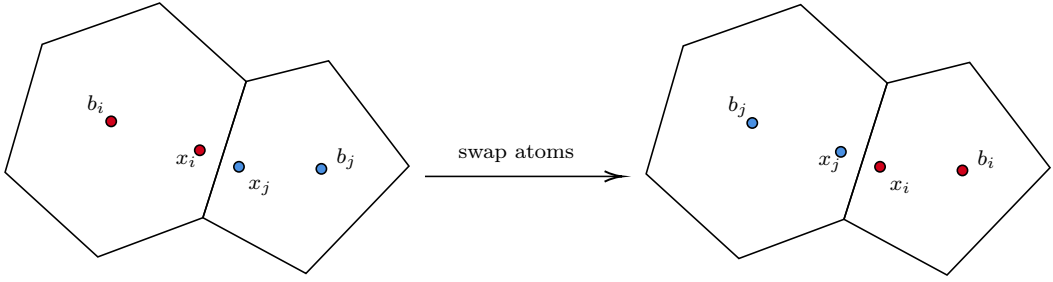

\begin{lemma}\label{lem:longtime_wellposed}
    Let $\Omega$ be a compact and convex subset of $\mathbb{R}^d$, and fix a continuous curve ${\left(\varrho_t\right)}_{t \ge 0}$ in $\Pac(\Omega)$ with uniformly bounded densities. Then the following hold:
    \begin{enumerate}
        \item The maximal interval of existence for the ODE 
        \[
            -\dot{\mathbf{x}}_t = \nabla_{\mathbf{x}} F_N(\varrho_t, \mathbf{x}_t) 
        \]
        is given by $\mathbb{R}_+$, provided that the initial atoms $\mathbf{x}_0$ are distinct and either $\partial \Omega$ is $\mathscr{C}^1$ or the initial conditions are away from the boundary $\partial \Omega$.
        
        \item For all times $t \ge 0$ the barycenters are away from each other, \textit{i.e.} there exists $C_\Omega$ such that
        \[
            |b_i(t) - b_j(t)| \ge \frac{C_{\Omega,\varrho}}{N^{d-1}}.        
        \]

        \item If ${\left(\varrho_t, \mathbf{x}_t\right)}_{t \ge 0}$ has finite kinectic energy~\eqref{eq.finite_kinectic_energy}, then there exists a constant $c$ for which 
        \[
            |x_{i}(t) - x_j(t)|\ge c,        
        \]
        for all $i \sim j$ and for all $t \ge 0$. 
    \end{enumerate} 
\end{lemma}
\begin{proof}
    To prove the first item, we recall that from~\cite[Thm.~1]{de2019differentiation}, if $\varrho \in \Pac(\Omega)$ then $\mathbf{x} \mapsto F_N(\varrho, \mathbf{x})$ is has up to $2$ derivatives in $\Omega^N \setminus \mathbb{D}_N$, where $\mathbb{D}_N$ is the generalized diagonal defined in~\eqref{eq.generalized_diagonal}, the set of colliding atoms, and the Hessian is given by 
    \[
        \nabla_{x_i,x_j} F_N(\varrho, \mathbf{x})
        = 
        \int_{\Sigma_{ij}} 
        (x - x_i) \otimes (x - x_j) \frac{\varrho(x)}{|x_i - x_j|} \dd \mathscr{H}^{d-1}(x), 
        \text{ for } i \neq j,
    \] 
    and 
    \[
        \nabla_{x_i,x_i} F_N(\varrho, \mathbf{x})
        = 
        \frac{1}{N} I - \sum_{j \neq i} \nabla_{x_i,x_j} F_N(\varrho, \mathbf{x}), 
        \text{ for } i = 1, \ldots, N.
    \]
    If, in addition, $\varrho$ is $\mathscr{C}^{0}(\Omega)$ then $F_N(\varrho, \cdot)$ is $\mathscr{C}^{2}$ in $\Omega^N \setminus \mathbb{D}_N$. On the other hand, $t \mapsto \nabla F_N(\varrho_t, \mathbf{x})$ is continuous, given $\mathbf{x} \in \Omega^N \setminus \mathbb{D}_N$. Either way, one can check that the Hessian is bounded in compact subsets of $\Omega^N \setminus \mathbb{D}_N$, so that $\mathbf{x} \mapsto \nabla_{\mathbf{x}F_N(\varrho, \mathbf{x})}$ is locally Lipschitz, so by the Cauchy-Lipschitz theorem the flow is locally well-posed in $\intt\Omega^N \setminus \mathbb{D}_N$.

    The maximal interval of existence theorem, see for instance~\cite[Thm.~2.3]{bressan2007introduction}, says that if $T$ is the suppremum of the times for which a solution $\mathbf{x}_t$ to the ODE exists, then either $T = +\infty$ or 
    \[
        \lim_{t \to T^-} 
        \left(
            |\mathbf{x}_t| + 
            \frac{1}{\dist(\mathbf{x}_t, \partial (\intt \Omega^N \setminus \mathbb{D}_N))}
        \right)
        = +\infty.
    \]

    For all $t < T$, the atoms $\mathbf{x}_t$ remain from a positive distance to the boundary $\partial \Omega$ due to item (4) of theorem~\ref{thm.qualitative}, provided that $\partial \Omega$ is $\mathscr{C}^1$ or the initial atoms are away from the boundary. In addition, to control the distance of $\mathbf{x}_t$ away from $\mathbb{D}_N$, we proceed as in~\cite{leclerc2020lagrangian,merigot2021non}: since $b_{i,t} \in \Omega_{i,t}$ we have that for $i \neq j$ that
    \begin{align*}
        \frac{\dd}{\dd t} \frac{1}{2}|x_{i,t} - x_{j,t}|^2 
        &= 
        -\inner{x_{i,t} - x_{j,t}, 
                b_{i,t} - x_{i,t} - (b_{j,t} - x_{j,t})
        }\\
        &= 
        |x_{i,t} - x_{j,t}|^2 + 
        \underbrace{\inner{x_{i,t} - x_{j,t}, b_{j,t} - b_{i,t}}}_{\ge 0}
        \ge |x_{i,t} - x_{j,t}|^2.
    \end{align*} 
    From the reverse Grönwall inequality it follows that $|x_{i,t} - x_{j,t}|^2 \ge e^{-t}|x_{i,0} - x_{j,0}|^2$ for all $t < T$. We conclude that $\dist(\mathbf{x}_t, \partial (\intt \Omega^N \setminus \mathbb{D}_N)) \ge c_T > 0$ for all $t < T$ and $T > 0$. Therefore, the only possibility is that $T = +\infty$.

    To prove item (2), we use the upper bound on the density $\norm{\varrho_{t}}_\infty \le C_\varrho$. Given $i\sim j$, assume w.l.o.g.~that $\Sigma_{ij} \subset \{x_d = 0\}$. Set $d_{ij}$ such that $\varrho_t( \{x_d > d_{ij} 0\}) = 1/2N$. Then it must hold that $b_i \in \{x_d = d_{ij}\}$ and $\dist(b_i, \partial \Sigma_{ij}) = d_{ij}$. As a result, we have that 
    \[
      \frac{1}{2N} = \varrho_t( \{x_d > d_{ij} 0\}) 
      \le 
      C_\varrho|\{x_d > d_{ij}\}\cap\Omega| \le C_\varrho \diam\Omega d_{ij}^{d-1}, 
    \]
    and the result follows by rearranging the above inequality and bounding $|b_i - b_j| \ge \dist(b_i, \Sigma_{ij}) + \dist(b_j,\Sigma_{ij})$.

    Moving on to item (3), suppose by contradiction that for every $\varepsilon > 0$ there exists a sequence $t_k \cvstrong{k \to \infty}{} +\infty$ for which there is a pair $i_k\sim j_k$ such that 
    \[
      |x_{i_k}(t_k) - x_{j_k}(t_k)| \le \frac{\varepsilon}{3}.
    \]
    Since $t \mapsto \mathbf{x}_t$ is Lipschitz continuous with constant $\diam \Omega$, we see that there is an interval $I_k$ centered at $t_k$ of size $|I_k| = 2\varepsilon/3\diam\Omega$ such that 
    \begin{align*}
        |x_{i_k}(t) - x_{j_k}(t)| 
        &\le 
        |x_{i_k}(t_k) - x_{i_k}(t)| + |x_{i_k}(t_k) - x_{j_k}(t_k)| + 
        |x_{j_k}(t_k) - x_{j_k}(t)| 
        \le \varepsilon
    \end{align*}
    for all $t \in I_k$. In particular, taking $\varepsilon$ much smaller than $\frac{C_{\Omega,\varrho}}{N^{d-1}} \le |b_i - b_j|$, so that $\frac{C_{\Omega,\varrho}}{N^{d-1}} \le \norm{\mathbf{x}_t - \mathbf{b}_t}$ inside $I_k$ for all $k \in \mathbb{N}$. In this case we would have 
    \[
      \int_0^\infty \norm{\mathbf{x}_t - \mathbf{b}_t}^2 \dd t 
      \ge 
      \sum_{k \in \mathbb{N}} \int_{I_k} \norm{\mathbf{x}_t - \mathbf{b}_t}^2 \dd t 
      = +\infty.
    \]
    This gives a contraction and hence $|x_i - x_j|$ must be uniformly bounded from below. 
\end{proof}

The above Lemma gives a uniform lower bound on the distance of $\mathbf{x}_t$ to the generalized diagonal, meaning that the dynamics does not degenerate asymptotically. However, the evolution of the $\mathbf{b}_t$ depends on the smoothness of the evolution of the free boundaries $\Omega_i(t) = \Lag_i(\mathbf{x}_t, \psi_t)$, that is on the regularity of $\psi_t$. Since $t\mapsto \mathbf{x}_t$ is uniformly Lipschitz, we shall see below that the smoothness of $t \mapsto \psi_t$ depends essentially on the evolution of $\varrho_t$. 

\begin{lemma}\label{lemma.smoothness_psi_t}
    Let $\Omega$ be a compact and convex subset of $\mathbb{R}^d$, and fix a continuous curve ${\left(\varrho_t\right)}_{t \ge 0}$ in $\Pac(\Omega)$ with uniformly bounded densities, and such that $t \mapsto \varrho_t$ is uniformly continuous in time, with modulus of continuity $\omega_\varrho$. Then $t \mapsto \psi_t$, and consequently $t \mapsto \mathbf{b}_t$, are uniformly continuous with modulus of continuity $\omega_\varrho$.
\end{lemma}
\begin{proof}
  The regularity of evolution of barycenters is dictated by the regularity of the evolution of the curve $\varrho_t$ and of the optimal Kantorovich potentials $\psi_t$. To see this, define the function $G_t : \Omega^N \times \mathbb{R}^N \to \mathbb{R}^{N+1}$ as 
  \[
    \begin{array}{rl}
      G_t : \Omega^N \times \mathbb{R}^N 
      &\to 
      E \eqdef 
      \left\{
        \sum \psi_i = 0  
      \right\}\\
      (\mathbf{x}, \psi)
      &\mapsto 
      {\left(
        G_{i,t}(\mathbf{x},\psi)
      \right)}_{i = 1}^N,
    \end{array}
    \text{ where }
    G_{i,t}(\mathbf{x},\psi)
    \eqdef 
    \varrho_t(\Lag_i(\mathbf{x},\psi)) - \frac{1}{N}.
  \]

  It then follows that for all $t\ge 0$, at the evolution of the gradient flow we have $G_t(\mathbf{x}_t,\psi_t) \equiv 0$. In addition, the Jacobian of $G_t$ can be computed explicitly as 
  \begin{align*}
    \frac{\partial G_{t,i}}{\partial \psi_j} = 
    - \int_{\Sigma_{ij}}\frac{\varrho_t(x)}{|x_j - x_i|}\dd \mathscr{H}^{d-1}(x) 
    \text{ for $j\neq i$, }
    &\quad 
    \frac{\partial G_{t,i}}{\partial \psi_i}
    = 
    - \sum_{j \neq i} \frac{\partial G_{t,i}}{\partial \psi_j},\\ 
    \frac{\partial G_{t,i}}{\partial x_j} = 
    - 
    \int_{\Sigma_{ij}}\frac{x - x_i}{|x_j - x_i|}\varrho_t(x) \dd \mathscr{H}^{d-1}(x)
    \text{ for $j\neq i$, }
    &\quad 
    \frac{\partial G_{t,i}}{\partial x_i}
    = 
    - \sum_{j \neq i} \frac{\partial G_{t,i}}{\partial x_j}.
  \end{align*}

  Therefore, for each $t\ge 0$ the Jacobian $D_\psi G_t$ has rank $N-1$ and we can apply the implicit function theorem, because $E$ is a $N-1$-dimensional hyperplane, to write $\psi_t$ as a $\mathscr{C}^1$ function of $\mathbf{x}_t$, $\psi_t = g_t(\mathbf{x}_t)$. In addition, from Lemma~\ref{lem:longtime_wellposed}, since the points $x_i$ are uniformly separated, the Jacobian $D_\psi G_t$ has a bounded inverse in $E$. We can therefore estimate the sensitivity of $\psi_t$ on time by noticing that 
  \begin{align*}
    0 &= 
    \underbrace{
      G_t(\mathbf{x}_t,\psi_t) - G_t(\mathbf{x}_s,\psi_s)
    }_{\eqdef \Delta_{x,\psi}G} 
    + 
    \underbrace{
      G_t(\mathbf{x}_s,\psi_s) - G_s(\mathbf{x}_s,\psi_s)
    }_{\eqdef \Delta_{t,s}G} 
  \end{align*}
  and estimating each term. 

  For the first one, from the $\mathscr{C}^1$ dependence of $\psi_t$ on $\mathbf{x}_t$ for $t$ fixed, we can apply the mean value theorem to find $(\mathbf{\bar x},\psi)$ such that 
  \[
    \Delta_{x,\psi}G = 
    D_{\mathbf{x}}G_t(\mathbf{\bar x},\bar \psi)(\mathbf{x}_t-\mathbf{x}_s) 
    + 
    D_{\psi}G_t(\mathbf{\bar x},\bar \psi)(\psi_t-\psi_s). 
  \]
  For the second term, each component is controlled by the modulus of continuity in time of $\varrho_t$: 
  \[
    \Delta_{t,s}G_i = 
    \varrho_t(\Lag_i(\mathbf{x}_s,\psi_s))
    - 
    \varrho_s(\Lag_i(\mathbf{x}_s,\psi_s)).
  \]
  As a result, since the Jacobian w.r.t.~$\psi$ has bounded inverse and $t\mapsto \mathbf{x}_t$ is Lipschitz, we get that 
  \[
    |\psi_t - \psi_s|
    \le 
    \norm{D_{\psi}G_t(\mathbf{\bar x},\bar \psi)^{-1}}_\infty 
    \left(
      \diam\Omega|t - s| + \omega_{\varrho}(|t - s|)
    \right). 
  \]

  This implies that $t \mapsto \psi_t$ has the same modulus of continuity of $t \mapsto \varrho_t$ in $L^1$, and from the uniform continuity of $(\mathbf{x}, \psi) \mapsto \Lag_i(\mathbf{x}_t,\psi_t)$ in the strong topology of $L^1(\Omega)$, we conclude that the curve of barycenters $t \mapsto \mathbf{b}_t$ also inherits the same modulus of continuity of $\varrho_t$.
\end{proof}

We can now sinthesize these arguments into the desired convergence of the distance to barycenters~\eqref{eq.distance_points_barycenters}. 
\begin{theorem}\label{thm.conv_bary}
  If either $\varrho_t \equiv \varrho$ is fixed, or ${(\varrho_t, \mathbf{x}_t)}_{t \ge 0}$ solve~\eqref{eq:longtime}, then 
    \[
      \norm{\mathbf{x}_t - \mathbf{b}_t} \cvstrong{t \to \infty}{} 0.
    \]
\end{theorem}
\begin{proof}
    If either $\varrho_t\equiv\varrho$ is fixed, or if ${(\varrho_t, \mathbf{x}_t)}_{t \ge 0}$ solves~\eqref{eq:longtime}, then a simple gradient flow argument gives that $\mathbf{x}_t$ has finite kinectic energy~\eqref{eq.finite_kinectic_energy}. For instance, in the latter case notice that the time derivative of the energy is given by 
    \[
      \frac{\dd}{\dd t} 
      \mathscr{E}(\varrho_t,\mathbf{x}_t) 
      = 
      - I(t) - \frac{1}{N^2}\norm{\mathbf{x}_t - \mathbf{b}_t}^2,
    \]
    where $I(t)$ corresponds to the Fisher information of $\varrho_t$ w.r.t.~the Gibbs measure associated with the Kantorovich potential $\Phi_t$: 
    \[
      I(t) = 
      \int_{\Omega}
      \left|
        \frac{\nabla \varrho_t}{\varrho_t} + \nabla \Phi_t(x)
      \right|^2\dd \varrho_t \ge 0.
    \] 
    Therefore integrating over $\mathbb{R}_+$ we get that 
    \[
      \frac{1}{N^2}\int_0^\infty \norm{\mathbf{x}_t - \mathbf{b}_t}^2\dd t 
      \le 
      \mathscr{E}(\varrho_0,\mathbf{x}_0) < +\infty.
    \]
    In addition, it also holds that $t\mapsto \varrho_t$ is uniformly continuous in $L^1(\Omega)$. In fixed density case this holds trivially and in the latter case it is a consequence of being the solution of a Fokker-Planck equation with bounded drift having therefore a bounded and H\"older continuous density, see for instance~\cite{chizat2025convergence} and~\cite{bogachev2015fokker}. 

    Using Lemmas~\ref{lem:longtime_wellposed} and~\ref{lemma.smoothness_psi_t} above, we then have from the previous discussion that
  \[
    \int_0^\infty\norm{\mathbf{x}_t - \mathbf{b}_t}^2\dd t < + \infty 
    \text{ and } 
    t\mapsto \norm{\mathbf{x}_t - \mathbf{b}_t} \text{ is uniformly continuous,}
  \]
  as a consequence $\norm{\mathbf{x}_t - \mathbf{b}_t} \cvstrong{t \to \infty}{} 0$, and the result follows. 
\end{proof}


\section{Numerical Simulations and Conjectures}\label{sec.numerics}
In this section we describe a splitting scheme for the numerical simulation of the dynamic quantization equation~\eqref{eq.dynamic_quantization}. Our goal with the numerical experiments is to help elaborate conjectures on the long-time behavior of our system. The scheme consists on discretizing the time interval where one wishes to compute solutions and update alternatively the approximations for the continuous density $\varrho$ and the atoms' positions $\mathbf{x}$. 

The PDE that gives the evolution of $\varrho$ is discretized on a mesh $\mathcal{T}$ that is independent of the evolution of the points $\mathbf{x}$ or the Laguerre tessellation induced by the semi-discrete optimal transportation problem. This allows for a splitted scheme; at each time step we first evolve the positions $\mathbf{x}$ with one step of an Euler scheme for their underlying ODE with the density fixed, which can be easily computed with the \texttt{pysdot} package~\cite{pysdot}. This evolution has the advantage of being performed off the grid $\mathcal{T}$. Afterwards the evolution of $\varrho$ can be computed efficiently with general purpose solvers of Fokker-Planck type equations such as \texttt{fipy}~\cite{guyer2009fipy}. 

From one hand, better integrating the evolution of Laguerre cells and the underlying mesh $\mathcal{T}$ which determines the resolution of the continuous density could lead to numerical methods that take into account the flow of mass entering and exiting each Laguerre cell, for instance with a finite volume approach~\cite{cances2020variational}. On the other hand, this approach with a independent mesh and Laguerre tesselation allows for an off the grid approach for the atomic measure's evolution, since then the Laguerre tessellation is uniquely determined by the scalar Kantorovitch potentials ${\left(\psi_i\right)}_{i = 1}^N$. This description of our method is summarized in Algorithm~\ref{algorithm}.
\begin{algorithm}[h]
\SetAlgoLined
\SetKwInOut{Input}{Input}
\SetKwInOut{Output}{Output}
\Input{initial density \(\varrho_0\) on a mesh $\mathcal{T}$, initial points \(\mathbf{x}_0\), time step \(\tau\)}
\Output{sequence of approximations \((\varrho_{n}, \mathbf{x}_{n})\)}
\For{$n=0,1,2,\dots$}{
  \tcp{Compute OT (semi-discrete) with pysdot} 
    \vspace{0.1cm}
    \quad With \texttt{pysdot} to compute $\psi_n = {\left(\psi_{n,i}\right)}_{i = 1}^N$ and $\mathbf{b}_n = {\left(b_{n,i}\right)}_{i = 1}^N$ 
    \quad corresponding to $W_2^2\left(\varrho_n, \frac{1}{N}\sum_{i = 1}^N \delta_{x_{n,i}}\right)$ \;
    \vspace{0.1cm}
  \tcp{ Compute OT map via optimality conditions} 
    \[
        T_n(x)= x_{n,i^\star} \text{ where }
         i^\star = \argmin_{i = 1,\dots,N} \tfrac12\|x-x_{n,i}\|^2-\psi_{n,i} 
    \]
    \quad from optimality conditions for semi-discrete OT\;
    \vspace{0.1cm}
  \tcp{Point ODE update with explicit Euler} 
    \vspace{0.1cm}
    \quad Update atoms $\mathbf{x}_{n+1} = \mathbf{x}_{n} - \alpha (\mathbf{x}_{n} - \mathbf{b}_{n}) $ (e.g.\ \(\alpha=\tau N\))\;
    \vspace{0.1cm}
  \tcp{Fix advection over interval \([\tau n,\tau(n+1)]\)}
    \vspace{0.1cm}
    \quad Set advection $\nabla \Phi_n(x) = x - T_n(x)$\;
    \vspace{0.1cm}
   \tcp{Solve PDE with \texttt{fipy} over interval \([\tau n,\tau(n+1)]\)}
    \vspace{0.1cm}
    \quad Solve $\partial_t \varrho^n = 
        \Delta \varrho^n + 
        \divv\left(
            \varrho^n \nabla \Phi_n
        \right)$
    with $\varrho^n(0) = \varrho_n$ and set $\varrho_{n+1} = \varrho^n(\tau)$\;
}
\caption{Splitting scheme for~\eqref{eq.dynamic_quantization}}
\label{algorithm}
\end{algorithm}

On all experiments, the dynamics of $\mathbf{x}$ are slightly modified in order to observe the long-time behavior more quickly; more specifically we multiply their dynamics by a multiplicative factor 
\[
    \dot{\mathbf{x}}_t = \alpha\left(\mathbf{b}_t - \mathbf{x}_t\right).
\]
Since we know that the distance to the barycenters will go to zero from Theorem~\ref{thm.conv_bary}, it makes sense to add the multiplicative constant in the dynamics since for sufficiently large times $|x_i(t) - b_i(t)|$ is at most of the other of the $\diam \Omega_i(t)$ which, heuristically is of the order $N^{-1/d}$, the typical distance between minimizers of the uniform quantization problem. This speed parameter serves therefore to accelerate the convergence of the distance to the barycenters and observe the asymptotic behavior more quickly. 

Before discussing the experiments, we briefly emphasize that the numerical results below are exploratory and primarily intended to illustrate qualitative features of the long-time behavior of~\eqref{eq:longtime}.

\subsection*{A crystallization phenomenon}

On Figures~\ref{fig:crystallizationFK}--\ref{fig:crystallizationPME}, we display the approximate steady states obtained for increasing numbers of points $N = 50, 100, 200, 300, 400, 500$ and a velocity scaling $\alpha_N = \sqrt{N}$. In the case of linear diffusion, shown in Figure~\ref{fig:crystallizationFK}, the steady configurations of the atoms tend toward a uniform triangular lattice as $N$ increases, revealing a clear dynamic crystallization effect. The color map in this figure is not uniform across subplots, in order to highlight the tendency of the diffuse density $\varrho$ to concentrate according to a Gibbs-type weight with respect to the optimal Kantorovich potential inside each Laguerre cell. Formally, for large $N$ and $t$, the stationary configuration is well approximated by
\begin{equation}\label{eq.stationary}
    \varrho_{N}(t) \propto 
    \sum_{i = 1}^N e^{-\left(\frac{1}{2}|x - x_i(t)|^2 - \psi_i(t)\right)} \mathds{1}_{\Omega_i(t)}.
\end{equation}

As $N$ grows, the Laguerre cells $\Omega_i$ shrink and $\varrho_N$ becomes nearly constant within each cell, which makes~\eqref{eq.stationary} visually appear as a uniform density. This behavior is consistent with the crystallization phenomena established in~\cite{bourne2014optimality,bourne2021asymptotic} for models where the diffuse component is homogeneous. The trend is further illustrated in Figures~\ref{fig:exp1N=3} and~\ref{fig:exp3N=50}, which use a common color map across time to show that the densities flatten as equilibrium is approached.

\begin{figure}[ht]
    \centering
    \includegraphics[width=\textwidth]{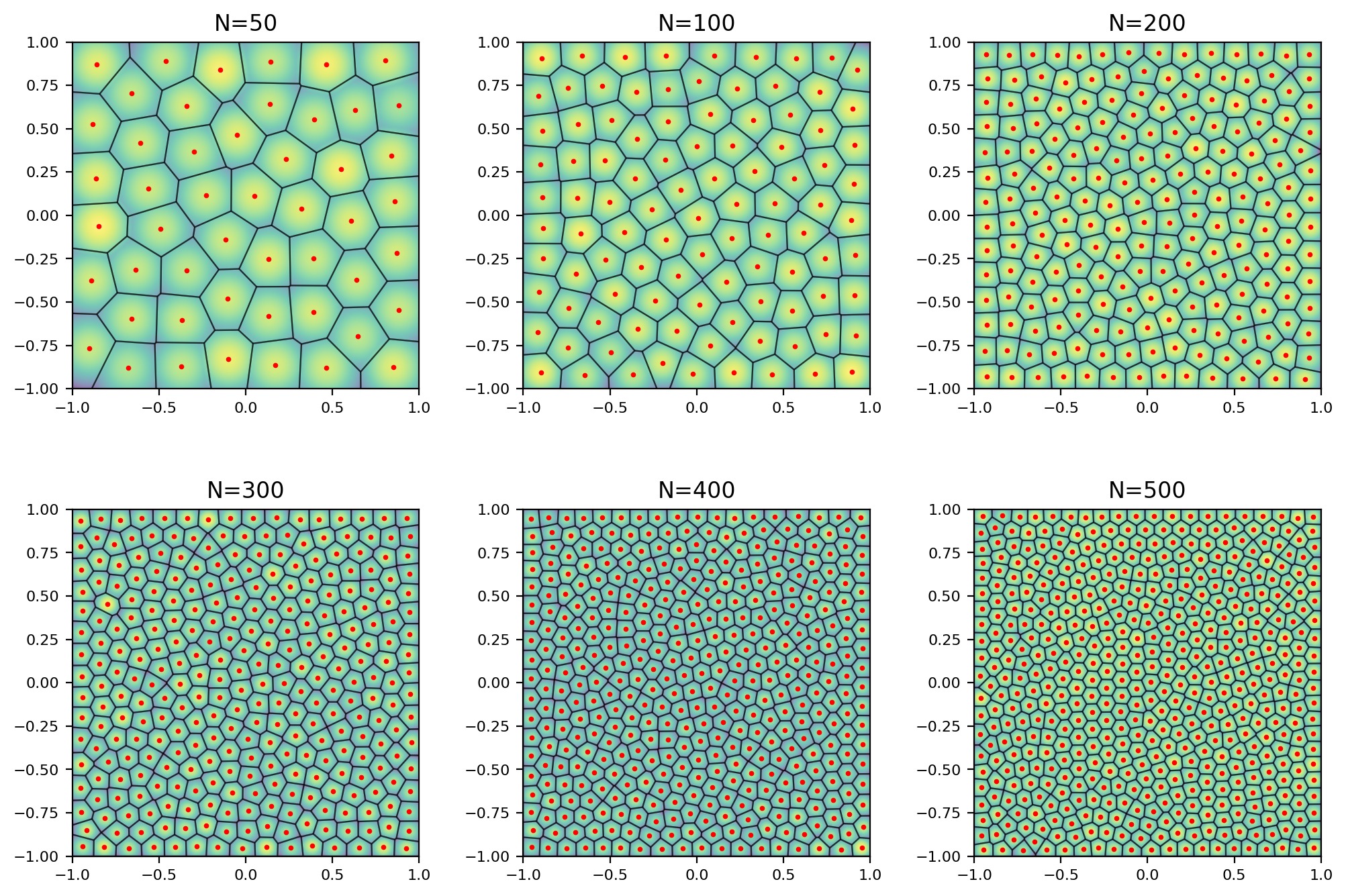}
    \caption{
        Crystallization phenomenon for the model with linear diffusion and speed $\alpha = \sqrt{N}$. 
        Long-time configurations for $N = 50,100,200,300,400,500$ after $200$ time steps with $\tau = 0.01$. 
        The density $\varrho$ stabilizes to an almost uniform profile, leading to a static triangular lattice of Laguerre cells.
    }
    \label{fig:crystallizationFK}
\end{figure}

A similar but less pronounced crystallization effect is observed for the evolution with a Porous Medium-type diffusion term, shown in Figure~\ref{fig:crystallizationPME}. 
Here we set $P(\varrho) = \varrho^m$ with $m = 10$, again with $\alpha = \sqrt{N}$. 
The nonlinear diffusion favors locally concentrated regions, producing a sharper crystalline pattern near zones of higher density.

\begin{figure}[ht]
    \centering
    \includegraphics[width=\textwidth]{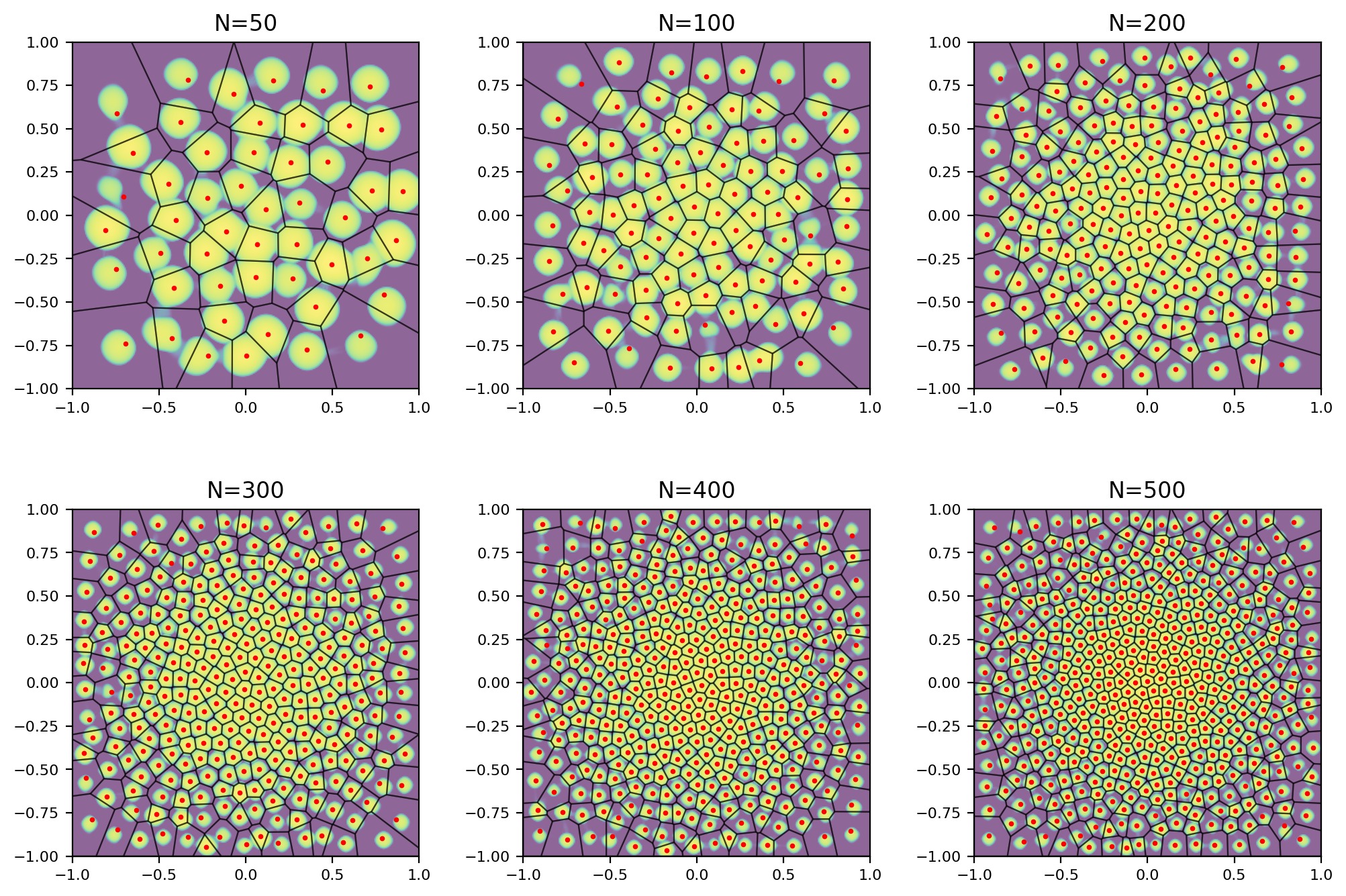}
    \caption{
        Crystallization phenomenon for the model with PME-type diffusion, i.e.\ $\Delta \varrho^m$ with $m = 10$ and speed $\alpha = \sqrt{N}$. 
        Shown are steady configurations for $N = 50,100,200,500,800,1000$ after $200$ time steps with $\tau = 0.01$. 
        The nonlinear diffusion enhances localization and produces sharper crystalline regions in high-density zones, though small fluctuations remain visible even at long times.
    }
    \label{fig:crystallizationPME}
\end{figure}

\subsection*{Discussion on the long-time behavior of the densities}

We now turn to a closer inspection of the long-time behavior of the diffuse component $\varrho_t$. 
As suggested by the previous section, it is natural to conjecture that stationary configurations of the coupled system should approach the Gibbs-type profile~\eqref{eq.stationary}. 
However, this correspondence does not always manifest clearly for small numbers of atoms.

Figure~\ref{fig:exp1N=3} shows the evolution of~\eqref{eq:longtime} with only three atoms under linear diffusion. 
While the atoms converge to the barycenters of their respective Laguerre cells, the resulting tessellation does not align perfectly with the visible structure of the density. 
This suggests that for small $N$, the coupling between the discrete and continuous components is too coarse for the equilibrium state to fully reflect the formal stationary form~\eqref{eq.stationary}.
\begin{figure}[ht]
    \centering
    \includegraphics[width=\textwidth]{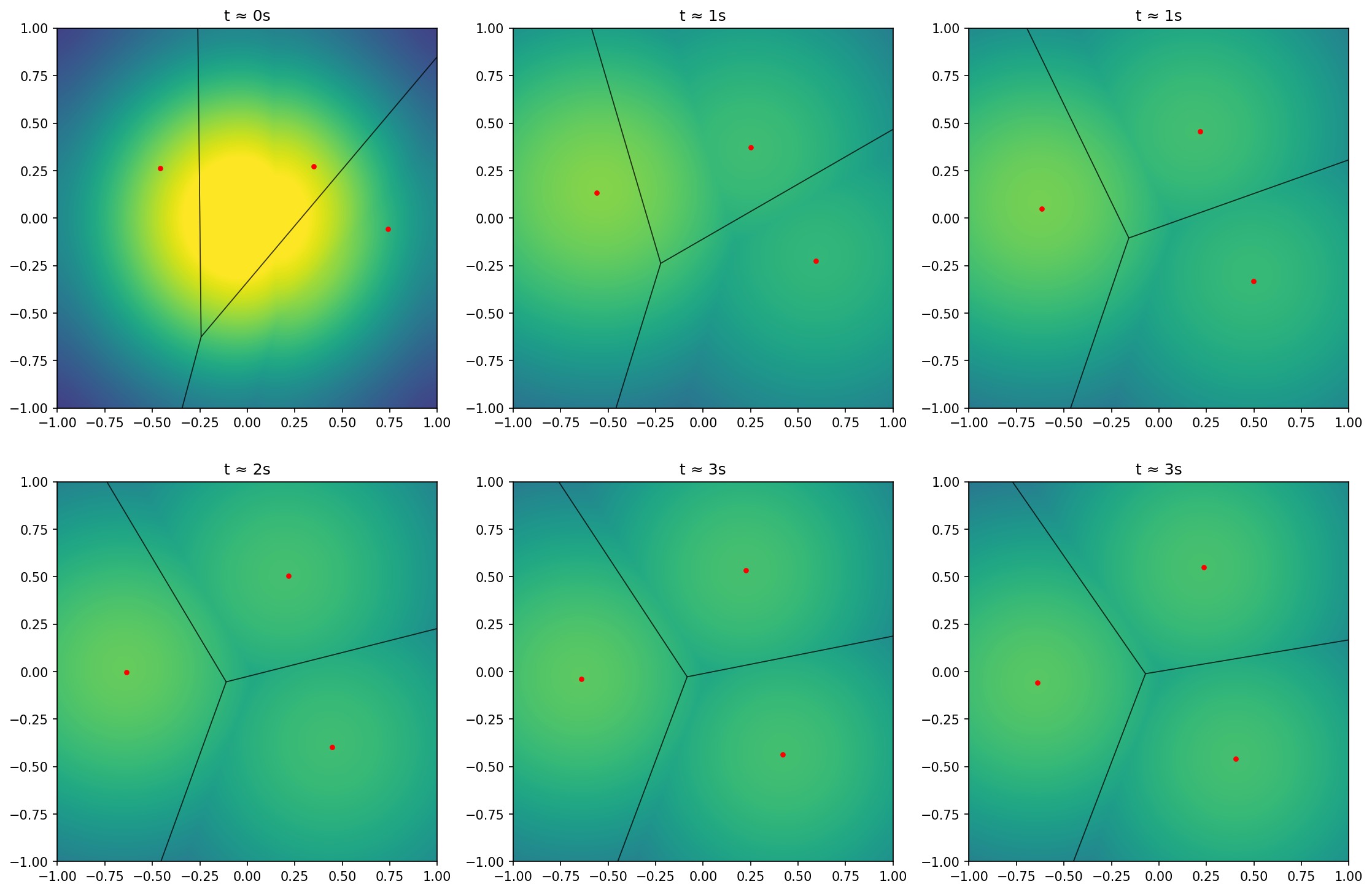}
    \caption{
        Long-time behavior for the model with linear diffusion and $N=3$ atoms. 
        The atoms converge to their barycenters, but the diffuse density $\varrho_t$ does not coincide with the Gibbs-type stationary profile associated with $\Phi_t[\mathbf{x},\psi]$.
    }
    \label{fig:exp1N=3}
\end{figure}

The situation changes markedly for larger $N$. 
In Figure~\ref{fig:exp3N=50}, we repeat the same experiment with $N=50$ atoms. 
Although the color map is kept fixed across all times (which visually exaggerates the flattening), the final configuration shows an almost constant density, consistent with the stationary structure~\eqref{eq.stationary} and with the crystalline organization observed in Figure~\ref{fig:crystallizationFK}. 
This supports the idea that as $N \to \infty$, the discrete measure of atoms becomes dense enough to recover the expected macroscopic equilibrium.

\begin{figure}[ht]
    \centering
    \includegraphics[width=\textwidth]{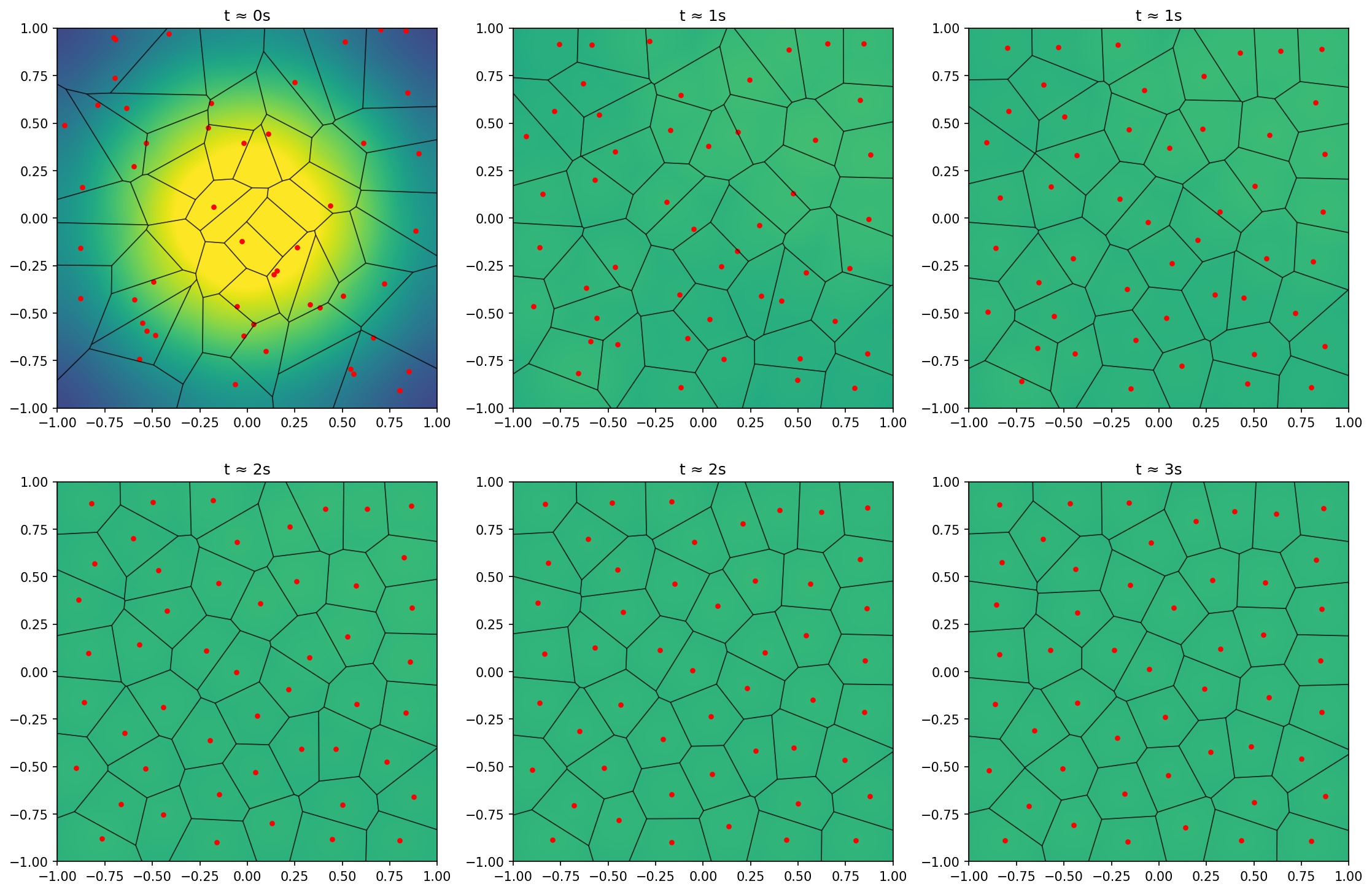}
    \caption{
        Long-time behavior for the model with linear diffusion and $N=50$ atoms. 
        The diffuse component $\varrho_t$ becomes nearly uniform, consistent with the asymptotic configuration described in~\eqref{eq.stationary}. 
        The apparent constancy of the density is accentuated by using the same color scale for all times.
    }
    \label{fig:exp3N=50}
\end{figure}

Figures~\ref{fig:evolPME_N=5} and~\ref{fig:evolPME_N=500} illustrate the corresponding dynamics for the porous-medium variant of the model, in which the diffusion term is replaced by $\Delta P(\varrho)$ with $P(\varrho) = \varrho^m$ and $m = 10$. 
The first feature we observe in this evolution is that, even for a well-spread initial condition (here a truncated Gaussian), the early dynamics are dominated by the attraction toward the atomic measure. 
This leads to a rapid concentration of mass around the points $\mathbf{x}_t$, followed by a slower diffusion-driven relaxation toward profiles typically favored by the porous-medium equation. 
This two-stage behavior highlights the subtle balance between aggregation and diffusion inherent to the coupled system.

Although the nonlinear diffusion case is not directly covered by the analysis of Theorem~\ref{thm.conv_bary}, the numerical results consistently show convergence of each atom $x_i$ toward the barycenter of its Laguerre cell. 
This behavior is in line with the known regularizing effects of porous-medium equations, which yield H\"older-continuous densities~\cite{caffarelli1979continuity,caffarelli1980regularity,vazquez2007porous}. 
Even though rigorous results for the coupled setting are lacking, the persistence of convergence strongly suggests that this regularity continues to stabilize the underlying ODE dynamics.

A different picture emerges when the velocity scaling is reduced to $\alpha = 1$, as opposed to the accelerated $\alpha_N = \sqrt{N}$ used in the crystallization experiments. 
In this slower regime, especially for large $N$ (e.g.\ $N=500$), the density no longer spreads uniformly but instead tends to form a single large connected component of higher concentration, as shown in Figure~\ref{fig:evolPME_N=500}. 
This indicates a competing effect between nonlinear diffusion, which promotes compactly supported profiles, and the attraction to barycenters, which drives spatial homogenization.

\begin{figure}[ht]
    \centering
    \includegraphics[width=\textwidth]{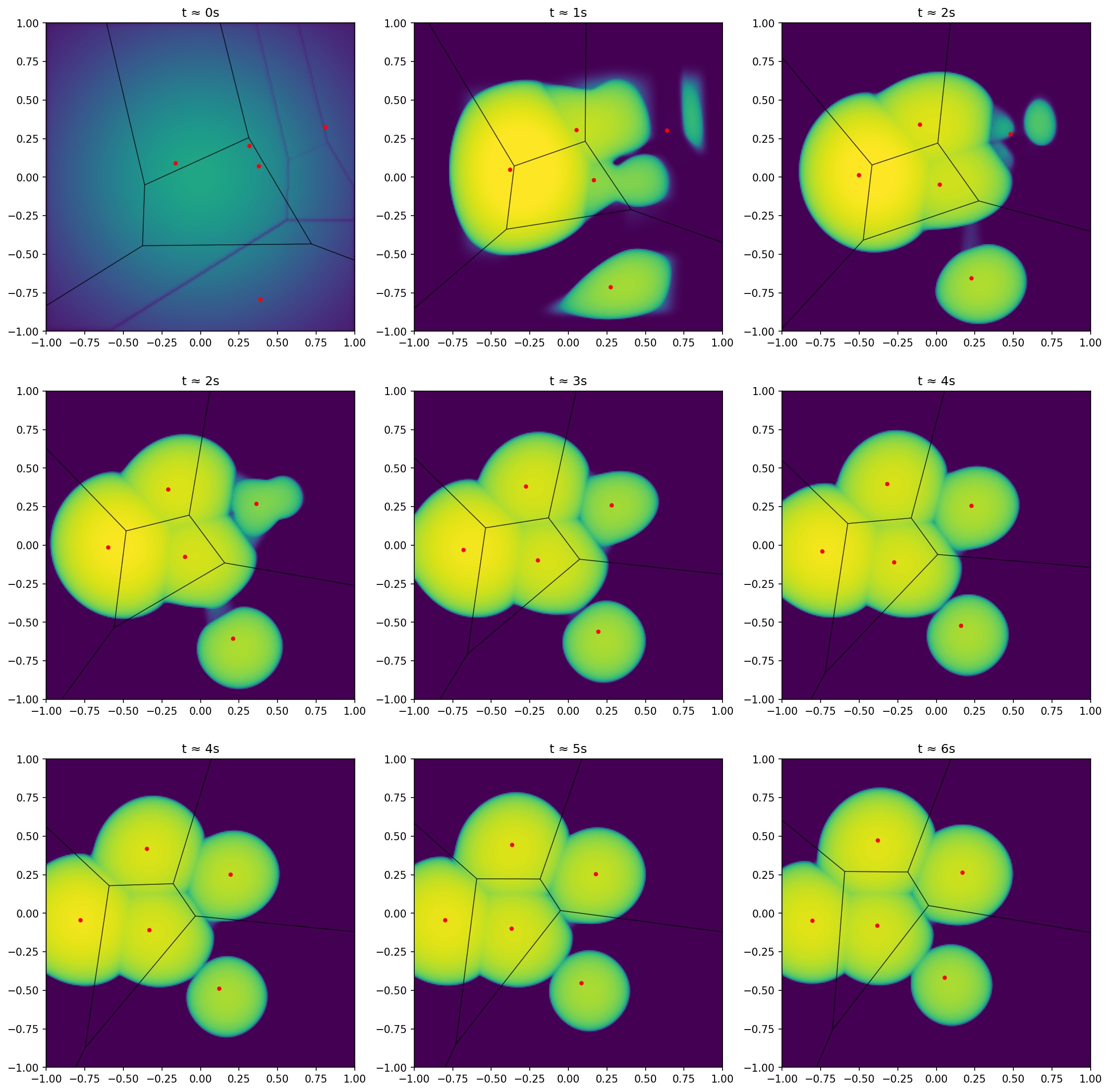}
    \caption{
        Long-time behavior for the nonlinear diffusion model with $P(\varrho)=\varrho^{10}$, $N=5$ atoms, and $\alpha=1$. 
        The convergence of atoms to barycenters persists, although the equilibrium density remains spatially nonuniform.
    }
    \label{fig:evolPME_N=5}
\end{figure}
\begin{figure}[b]
    \centering
    \includegraphics[width=\textwidth]{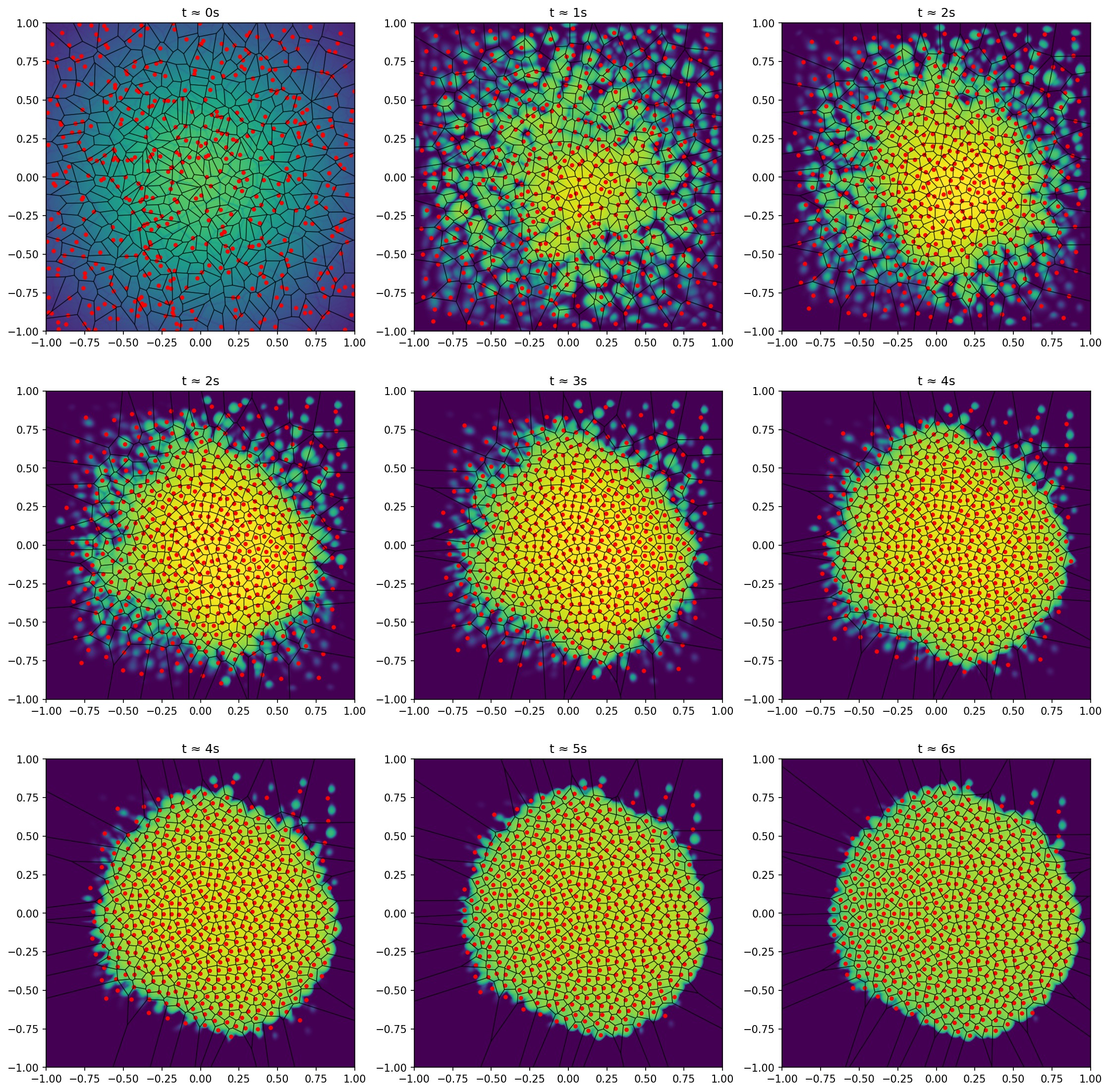}
    \caption{
        Long-time behavior for the nonlinear diffusion model with $P(\varrho)=\varrho^{10}$, $N=500$ atoms, and $\alpha=1$. 
        The slower dynamics lead to the formation of a large, connected concentration region, contrasting with the uniform crystallization pattern observed for faster scaling.
    }
    \label{fig:evolPME_N=500}
\end{figure}


\bibliographystyle{abbrv}
\bibliography{main.bib}

\end{document}